\documentclass[english,reqno]{amsart}

\usepackage[margin=2cm]{geometry}
\usepackage{amsrefs}

\usepackage{color}
\usepackage{amsmath,amssymb, amsfonts, bbm}
\usepackage{mathtools}
\usepackage{verbatim}
\usepackage{enumerate}
\usepackage{comment}
\usepackage{esint}
\usepackage{multicol}
\usepackage{graphicx}

\usepackage{mathrsfs}
\usepackage[pdfstartpage=1,bookmarks=false,pdfstartview={FitH},hidelinks]{hyperref}
\usepackage{cleveref}

\makeatletter
\def\l@subsection{\@tocline{2}{0pt}{2pc}{5pc}{}}
\makeatother
\newcommand{\COMMENT}[1]{}
\DeclarePairedDelimiter\abs{\lvert}{\rvert}
\makeatletter
\let\oldabs\abs
\def\abs{\@ifstar{\oldabs}{\oldabs*}}
\newcommand{\vast}{\bBigg@{4}}
\newcommand{\Vast}{\bBigg@{5}}
\makeatother
\newcommand{\eps}{\varepsilon}

\newcommand{\N}{\mathbb{N}}
\newcommand{\R}{\mathbb{R}}

\newcommand{\bS}{\mathbb{S}}

\newcommand{\norm}[2][]{\left\|{#2}\right\|_{#1}}

\newcommand{\normbig}[2][]{\big\|{#2}\big\|_{#1}}

\newcommand{\set}[1]{\left\{#1\right\}}

\newcommand{\dist}{{\rm dist}\, }
\newcommand{\supp}{{\rm supp}\, }

\newcommand{\cA}{\mathcal{A}}
\newcommand{\cB}{\mathcal{B}}

\newcommand{\cD}{\mathcal{D}}
\newcommand{\cE}{\mathcal{E}}
\newcommand{\cF}{\mathcal{F}}
\newcommand{\cN}{\mathcal{N}}

\newcommand{\cH}{\mathcal{H}}

\newcommand{\cK}{\mathcal{K}}

\newcommand{\cQ}{\mathcal{Q}}
\newcommand{\cR}{\mathcal{R}}

\newcommand{\cZ}{\mathcal{Z}}

\newcommand{\FB}{{\rm FB}}

\newcommand{\bc}{{\mathbf c}}
\newcommand{\rb}{r_\star}
\newcommand{\bN}{{\boldsymbol N}}
\newcommand{\bE}{{\boldsymbol E}}
\newcommand{\bbE}{\widetilde {\boldsymbol E}}

\newcommand{\wW}{{\widetilde W}}

\newcommand{\zz}{\mathbf{z}}
\newcommand{\epk}{{\widetilde \eps_k}}
\newcommand{\Rk}{{\widetilde R_k}}
\newcommand{\zk}{{\widetilde \zz_k}}
\newcommand{\rmin}{r_{{\rm min}} }

\newcommand{\nAC}{n_*}
\newcommand{\cttc}{\chi}

\newcommand{\lsup}[2]{{}^{#1}\mkern-1.5mu #2}

\DeclareMathOperator{\diam}{diam}

\DeclareMathOperator{\Per}{Per}

\theoremstyle{plain}
\newtheorem{thm}{Theorem}[section]
\newtheorem{lem}[thm]{Lemma}
\newtheorem{cor}[thm]{Corollary}
\newtheorem{prop}[thm]{Proposition}

\newtheorem*{prop*}{prop}

\theoremstyle{definition}
\newtheorem{defn}[thm]{Definition}
\newtheorem{assumption}[thm]{Assumption}

\theoremstyle{remark}
\newtheorem{remark}[thm]{Remark}
\newcommand{\bremark}{\begin{remark} \em}
\newcommand{\eremark}{\end{remark} }

\numberwithin{equation}{section}
\crefname{assumption}{assumption}{assumptions}
\Crefname{assumption}{Assumption}{Assumptions}

\title[]{The stability conjecture for the Bernoulli problem \\ in dimension four}

\author{Xavier Fern\'andez-Real}
\address{EPFL SB, Station 8, 1015 Lausanne, Switzerland}
\email{xavier.fernandez-real@epfl.ch}

\author{Joaquim Serra}
\address{ETH Zurich, R\"amistrasse 101, 8092 Zurich, Switzerland}
\email{joaquim.serra@math.ethz.ch}



\theoremstyle{definition}
\newtheorem{notation}[thm]{Notation}

\crefname{thm}{Theorem}{Theorems}
\Crefname{thm}{Theorem}{Theorems}
\crefname{lem}{Lemma}{Lemmas}
\Crefname{lem}{Lemma}{Lemmas}
\crefname{cor}{Corollary}{Corollaries}
\Crefname{cor}{Corollary}{Corollaries}
\crefname{prop}{Proposition}{Propositions}
\Crefname{prop}{Proposition}{Propositions}
\crefname{claim}{Claim}{Claims}
\Crefname{claim}{Claim}{Claims}
\crefname{conj}{Conjecture}{Conjectures}
\Crefname{conj}{Conjecture}{Conjectures}
\crefname{prob}{Problem}{Problems}
\Crefname{prob}{Problem}{Problems}
\crefname{defn}{Definition}{Definitions}
\Crefname{defn}{Definition}{Definitions}
\crefname{assumption}{Assumption}{Assumptions}
\Crefname{assumption}{Assumption}{Assumptions}
\crefname{remark}{Remark}{Remarks}
\Crefname{remark}{Remark}{Remarks}
\crefname{notation}{Notation}{Notations}
\Crefname{notation}{Notation}{Notations}

\begin{document}

\begin{abstract}
We prove that every global classical stable solution to the one-phase Bernoulli problem in $\mathbb R^4$ is one-dimensional. In particular, if nonempty, its free boundary consists of one or two parallel hyperplanes. This settles the stability conjecture for the Bernoulli problem in dimension four. As consequences, we obtain universal local Hessian and free-boundary curvature estimates for classical stable solutions in $\mathbb R^4$, as well as axial symmetry of global classical solutions with finite Morse index in $\mathbb R^4$.
\end{abstract}

\maketitle

\tableofcontents

\section{Introduction}

\label{sec:introduction}

Given an open set $\Omega\subset\mathbb R^n$, the Alt--Caffarelli functional is
\begin{equation}\label{eq:Alt-Caffarelli-functional}
\mathcal E(u;\Omega) :=\int_\Omega\bigl(|\nabla u|^2+\mathbbm{1}_{\{u>0\}}\bigr)\,dx, \qquad u\ge0.
\end{equation}
Its classical critical points (see \eqref{eq:stationary}) solve the one-phase Bernoulli problem
\begin{equation}\label{eq:intro-bernoulli}
\Delta u=0\quad\text{in }\{u>0\}, \qquad |\nabla u|=1\quad\text{on }\FB(u):=\partial\{u>0\}.
\end{equation}
We call such a solution \emph{stable} if the second inner variation of $\mathcal E$ is nonnegative for every compactly supported deformation (see \Cref{defn:notions-of-solution}). The basic rigidity question is whether a global classical stable solution must be one-dimensional.

First introduced in 1981 by Alt and Caffarelli in \cite{AC81}, the problem has received a lot of attention to date (see the monographs \cite{CS05, Vel23} for an introduction or the survey \cite{Fer26} for an account of recent progress). For minimizers, blow-up and improvement-of-flatness arguments reduce regularity to the classification of homogeneous minimizers. The results of Caffarelli--Jerison--Kenig \cite{CJK04} and Jerison--Savin \cite{Jerison-Savin} rule out singular minimizing cones through dimension four, while De Silva and Jerison \cite{DJ09} constructed a singular minimizing cone in dimension seven. The status of dimensions five and six remains open.

For stable critical points, rather than minimizers, the corresponding Bernstein-type conjecture (the stability conjecture for the one-phase Bernoulli problem; cf. \cite{CL24, CLMS24, Maz24}) asks whether every global classical stable solution of \eqref{eq:intro-bernoulli} is one-dimensional in dimensions $n\le 4$ (or $n \le 6$ in case the missing minimizer classification is completed). Equivalently, whether stable solutions have free boundaries with (quantitatively) bounded curvature, and other recent advances for non-minimizing solutions are \cite{ERZ25, KW24, kriventsov2025min, liu2021smooth, kamburov2022nondegeneracy}. Related Bernstein-type results for solutions with graphical free boundaries were obtained in \cite{EFY23}. In dimension $n = 2$, the result follows by a (now standard) log cut-off argument, valid for a wide class of semilinear PDEs \cites{FV09, FR19, kamburov2022nondegeneracy}. For dimension $n = 3$, instead, the conjecture was recently proved by Chan, Figalli, and the authors in \cite{CFFS25}.

Here we address the problem in dimension four, which is the highest dimension for which minimizers are classified.

\begin{thm}[Classification of global stable solutions]
\label{thm:main-classification}
Let $u:\mathbb R^4\to[0,\infty)$ be a global classical stable solution to the one-phase Bernoulli problem. Then
\[
D^2u\equiv0\qquad\text{in }\{u>0\}.
\]
In particular, $u$ is one-dimensional. If $\FB(u)\ne\varnothing$, then it consists of either one or two parallel hyperplanes.
\end{thm}

\begin{remark}
We state the theorem in dimension $n=4$, which is the new case. The arguments in this paper also yield a new proof of the corresponding result in dimension $n=3$, originally established in \cite{CFFS25}.
\end{remark}

The global rigidity result, \Cref{thm:main-classification}, also yields the following local, scale-invariant estimate. Indeed, if the estimate failed, rescaling at points where the Hessian is maximal would produce a non-flat global classical stable solution.

\begin{cor}[Local Hessian estimate]
\label{cor:main-local-hessian}
Let $u:B_1\to[0,\infty)$ be a classical stable solution to the one-phase Bernoulli problem in $B_1\subset\mathbb R^4$, and assume that $0\in\FB(u)$. Then
\[
\|D^2u\|_{L^\infty(B_{1/2}\cap\{u>0\})}\le C,
\]
where $C$ is universal. Consequently, the principal curvatures of $\FB(u)\cap B_{1/2}$ are universally bounded.
\end{cor}

There is also a consequence for solutions of finite Morse index. Such solutions are stable outside a compact set, and the main theorem of \cite{FFS26} gives axial symmetry in every dimension in which both global stable solutions and homogeneous stable solutions are flat. Combining that result with \Cref{thm:main-classification} and \cite{Jerison-Savin} yields:

\begin{cor}[Finite-index solutions]
\label{cor:main-finite-index}
Every global classical solution to the one-phase Bernoulli problem in $\mathbb R^4$ with finite Morse index is axially symmetric, after a translation and a rotation.
\end{cor}
\medskip

\subsection{Relation to previous work and new ingredients}

We briefly describe the paper's main new ideas; a fuller proof overview appears in \Cref{sec:proof-overview}. Our argument begins with the linearization strategy introduced in \cite{CFFS25}: away from a controlled collection of neck regions, the positivity set separates into two sides on which the problem can be linearized independently.

In dimension three, \cite{CFFS25} controls the number and size of the necks by testing stability with the Jerison--Savin function $F(D^2u)^{1/3}$; see \cite{Jerison-Savin}. Its special subsolution property yields a superlinear estimate of the largest neck scale in terms of the ``vee excess''
\[
R^{-2}\fint_{B_R(z)} |u-V_{z,e}|^2, \quad \mbox{for } V_{z,e}(x): = |e\cdot(x-z)|.
\]
This use of the Jerison--Savin function is dimension-specific: $F(D^2u)^{1/3}$ need not be a subsolution in dimensions $n\ge4$.\footnote{Jerison and Savin \cite{Jerison-Savin} also use this test function in dimension four, but only for $1$-homogeneous solutions. Homogeneity eliminates the radial direction, effectively reducing the relevant computation to the three-dimensional one.}

Instead, we use the more flexible tangential stability inequality \Cref{lem:tangential-sternberg-zumbrun-inequality}, suggested by \cites{Ton05, FS25} and based on the Jacobi fields generated by translations along a ridge. Here we face a new difficulty: unlike in the Allen--Cahn setting of \cite{FS25}, the Bernoulli vee excess does not directly control the tangential tilt
\[
\fint_{B_R(z)\cap\{u>0\}} \bigl(|\nabla u|^2-|\partial_eu|^2\bigr).
\]
We handle this by introducing a \emph{ridgewise Weiss deficit}. It is almost monotone up to a logarithmic factor and controls all gradient components tangent to the ridge. Together with the Monneau-type vee excess, it gives the combined excess that we track through the proof. Its construction rests on Weiss monotonicity and the symmetries of the model, rather than on an ad hoc expression specific to the Bernoulli equation.

However, this is not enough. A second new idea is needed because dimension four is critical (numerologically, dimension five for Allen--Cahn). Here, tangential stability gives only a linear relation between the neck charge---the sum of suitable powers of the neck radii---and the excess, so the superlinear gain used in \cite{CFFS25} is no longer available. Our way around this is a power--center--scale selection that balances the combined excess against the neck charge. We use the same optimized selection twice: first, phasewise linearization improves the excess on the two sides of the positivity set; second, stability tested with a suitable family of translation and dilation Jacobi fields improves the neck charge. The optimality of the selected exponent then shows that these two improvements put us in a strictly subcritical regime and restore a genuine power gain.

These ideas extend substantially the methods of \cites{CFFS25,FS25} beyond the settings considered there. They rely on a Weiss-type monotonicity formula, the symmetries and rigidity of the model solutions, and an Allard-type regularity theory near those models. We therefore expect the ridgewise deficit, power--center--scale selection, ball-tree decomposition, and the layer-wise linearization argument to be adaptable to other free-boundary and phase-transition problems with analogous structures. They may also provide a route toward the remaining dimensions; preliminary calculations in dimension five support this possibility (assuming the classification of cones), although that argument lies beyond the scope of the present paper.

\subsection{Use of artificial intelligence}

To complete this work, we built our own experimental proof-development environment primarily around OpenAI’s Codex, with occasional use of Anthropic’s Claude Code. 

The evolving body of mathematical information which is required to discover and construct a long mathematical proof---not merely the streamlined argument presented a posteriori in the finished paper---spans many model context windows. Our environment therefore served as an external memory and coordination layer, as well as the interface through which we interacted with the models and directed the development of the proof. It decomposed our evolving proof into over a hundred blocks, distributed across auxiliary files, and a ``master'' file assembling them. Models could operate on individual blocks largely independently, using only the locally relevant context and dependencies. 

An HTML-based visualizer displayed the proof's structure and current review status of each block, while Python tools tracked their dependencies and propagated status changes downstream. This setup allowed us to retain full mathematical control while using frontier models to draft and check technical passages. It also allowed us to quickly propagate changes in the proof architecture through the entire argument. When, for example, we changed an early definition or lemma, the models carried out routine downstream revisions and identified any genuinely new mathematical obstacles that arose. This allowed us to fully focus on the genuine blockers at each moment.

We next give further details on the workflow we used:
\begin{itemize}

\item \textbf{30-page draft input:}
This was neither autonomous theorem proving nor a single continuous conversation with a model. Before the model-assisted phase, the authors had already produced an incomplete 30-page draft, the global strategy, many key computations in \TeX, and earlier papers containing related methods. The original draft, strategy, and computations were written without LLM assistance.

\item \textbf{105 blocks:}
The final proof was divided into 105 blocks, developed sequentially. Each definition, lemma, and proposition underwent its own cycle of drafting, review, and revision; we never asked the models to generate batches of statements or proofs from a single prompt. A master file recorded the complete logical sequence, while companion files contained the discussion, proof, and references for each block. 

\item \textbf{5,000+ prompts:}
The environment allowed us to retain full mathematical control while directing the development of the proof step by step. Our input was injected through thousands of instructions to Codex. Every time a blocker was identified by the system, we focused on it and provided detailed instructions to unlock it.   We retained the prompts, instructions, drafts, objections, blockers, and discarded routes as an audit trail.

Typical instructions from the authors supplied a candidate new statement, its role, or a detailed proof plan, often including the relevant estimates, computations, case division, and intended conclusion. One run drafted the argument; fresh runs checked hypotheses, numerology, circularity, unjustified steps, and possible simplifications; their objections were then incorporated and reviewed again.

\item \textbf{3x acceleration:} We view this workflow primarily as an accelerator of mathematical writing. The project proceeded in two phases. The first was entirely traditional and lasted more than one year, resulting in the 30-page draft that contained a promising but still highly preliminary strategy for the proof. At that stage, based on our experience, we estimated that completing the argument would require at least one additional year. We then introduced the proof-development workflow described above and completed the proof within approximately three to four months. We attribute much of this acceleration to the environment’s ability to keep our attention focused on the central mathematical obstacles, while the models handled routine downstream revisions, rapidly propagated changes in the proof architecture across dependent blocks, and identified new obstacles as they emerged.

\item \textbf{Exportability:} Such workflow is not specific to this project: it can be exported elsewhere, and we have already done so successfully in other ongoing projects.

\item \textbf{Use of frontier models:}
Most development used GPT-5.5 in Codex at the \texttt{xhigh} reasoning setting, with occasional later use of Claude Code with Claude Fable 5. GPT-5.6 Sol became available after the proof was essentially complete and was used in \texttt{ultra} mode mainly to improve exposition, remove redundancy, and reorganize the manuscript, still under intensive author guidance.

\end{itemize}

We emphasize that we, the authors, fully designed all statements and proof strategies, and provided all the input that guided its subsequent revisions. The models completed and checked specified local arguments, exposed obstructions, tested our hypotheses, checked the numerology, compared implementations, propagated revisions, and assisted with exposition. At several points, encouraged by the models’ capabilities, we tried to extend their role by asking for new ideas, simplifications, or longer chains of argument. To our surprise, these attempts produced no genuinely new mathematical insight in this project; the models’ mathematical contributions remained confined to technical arguments already outlined by the authors. 

This is a factual account of the division of labor in this project, not a general claim about the limits of the models.  We record it not only for transparency, but also because it illustrates a mode of AI-assisted mathematics that is obscured by a focus on one-shot theorem proving. In this project, the strategy and key computations remained human-generated, while AI substantially accelerated their local completion, systematic criticism, revision, and propagation.

As the models and the surrounding workflows improve, the division of labor described here may also change; this statement records the balance in the present project. We realize, however, that the mathematics to come may emerge from a different balance between human and machine.

\subsection*{Acknowledgments}

X. F. was supported by the Swiss National Science Foundation (SNF grant PZ00P2\_208930), by the Swiss State Secretariat for Education, Research and Innovation (SERI) under contract number MB22.00034, and by the AEI project PID2024-156429NB-I0 (Spain).

J. S. was supported by the European Research Council under the Grant Agreement No 948029.

\section{Proof overview}
\label{sec:proof-overview}

The proof begins along the same broad lines as the Bernoulli argument in \cite{CFFS25}. Section~3 and most of Sections~4 and~5 recall results from \cite{CFFS25} that remain valid here. New estimates and constructions appear in Section~6, and from Section~7 onward the arguments are new. We now describe the complete proof.

\subsection{Ingredients from previous work}
We begin as in \cite{CFFS25}. Suppose that a non-flat global classical stable solution exists. By \Cref{lem:reduction,lem:lipschitz-bound}, after recentering and rescaling we may fix a solution $u$ such that
\begin{equation}\label{eq:overview-normalization}
\text{$0\in \FB(u)$, \quad $|\nabla u|\le 1$, \quad $|D^2 u(0)| = 1$, \qquad and \qquad $|D^2 u|\le 1\quad\text{in}\quad \{u > 0\}$.}
\end{equation}
Here $D^2u(0)$ denotes the one-sided limit from $\{u>0\}$. Using the classification of homogeneous limits in \Cref{prop:classification-homogeneous-classical-stable-limits}, \Cref{prop:blow-down-of-non-flat-solutions} shows that at every sufficiently large scale $u$ is close to a unit-slope vee $|e\cdot x|$, with a direction $e$ that may depend on the scale. This suggests a two-sheet picture, but does not by itself produce one. Sections~\ref{sec:neck_set} and \ref{sec:fb-tree} make that geometry quantitative.

Fix a universal constant $\eta_0>0$. For each $z\in\FB(u)$, we define (exactly as in \cite{CFFS25}) the threshold radius
\begin{equation}\label{eq:overview-threshold-radius}
r_\star(z):= \inf\left\{r>0: \int_{B_r(z)\cap\{u>0\}}|D^2u|^n\,dx\ge\eta_0^n \right\}.
\end{equation}
We choose a dyadically separated family $\cZ$ of points with finite threshold radius and call them \emph{neck centers}; the associated threshold balls have bounded overlap at each scale; see \Cref{defn:neck-centers}. Up to Section~\ref{sec:neck_set}, the argument differs little from \cite{CFFS25}, and we mainly explain how the estimates should be read here.

At this point we introduce also elements from \cite{Ton05, FS25}: For $e\in\mathbb S^{n-1}$, let $\nabla^e u$ be the projection of $\nabla u$ onto $e^\perp$, and let $\cA_e(u)$ denote the corresponding Sternberg--Zumbrun defect. For every compactly supported cutoff $\eta$, the tangential stability inequality gives
\[
\int \cA_e^2(u)\eta^2 \le \int |\nabla^e u|^2|\nabla\eta|^2;
\]
see \Cref{lem:tangential-sternberg-zumbrun-inequality}. A compactness and dimension-reduction argument shows that this defect cannot vanish at a neck scale: otherwise the limit would be two-dimensional and hence flat by \cite{CFFS25}*{Theorem~1.5}; see \Cref{lem:neck-scale-compactness,lem:zero-defect-two-dimensional-reduction}. \Cref{lem:positive-excess-at-the-neck-scale} carries out this compactness argument and proves the gap
\begin{equation}\label{eq:overview-positive-neck-defect}
(2r_\star(\zz))^{2-n} \min_{e\in\mathbb S^{n-1}} \int_{B_{2r_\star(\zz)}(\zz)\cap\{u>0\}}\cA_e^2(u)\,dx \ge c, \qquad \zz\in\cZ.
\end{equation}

It is useful to record this information additively. We define the square of the neck charge as follows. Write $\cZ_R:=\{\zz\in\cZ:r_\star(\zz)\le R\}$ and set
\begin{equation}\label{eq:overview-neck-charge}
\varrho_z(R)^2 := \sum_{\zz'\in\cZ_{R/2}\cap B_{R/2}(z)} \left(\frac{r_\star(\zz')}{R}\right)^{n-2}.
\end{equation}
Summing \eqref{eq:overview-positive-neck-defect} over the bounded-overlap neck balls and applying tangential stability, we control this aggregate charge by the transverse $L^2$ energy; see \Cref{prop:symmetric-excess-controls-neck-radii}. Thus stability controls all nearby neck radii together, not only the largest one.

\subsection{A logarithmically propagating excess}
From here we strongly deviate from previous papers. The transverse energy supplied by stability does not itself propagate when either the scale or the center changes. In Section~\ref{sec:weiss-type-excess}, we introduce an excess that propagates across scales with a logarithmic loss and can be recentered in a controlled way. Let $V_{y,e}(x):=|e\cdot(x-y)|$ and let $W_{R,x}$ be the Gaussian Weiss energy. By \Cref{lem:universal-gaussian-weiss-limit}, its large-scale limit is the center-independent vee value $W_\infty=\pi^{n/2}/2$. We set $\wW_{R,x}:=W_\infty-W_{R,x}\ge0$; see \Cref{lem:gaussian-weiss-deficit-bound}. For a candidate vee direction $e$, we define
\begin{equation}\label{eq:overview-combined-excess}
\Theta_{R,y}^e := \sup_{x\in(y+e^\perp)\cap B_R(y)}\wW_{R,x}, \qquad \bbE_y(R;e)^2 := \Theta_{R,y}^e +R^{-n-2}\int_{B_R(y)}|u-V_{y,e}|^2.
\end{equation}
We then minimize in $e$ to obtain $\bbE_y(R)$; see \Cref{defn:monneau-excess}.

As \Cref{lem:monneau-excess-controls-transverse-energy} shows, moving the Gaussian Weiss center along the candidate ridge detects every component of the tangential gradient. Together with the neck-scale gap \Cref{lem:positive-excess-at-the-neck-scale} and tangential stability \Cref{lem:tangential-sternberg-zumbrun-inequality}, it gives
\begin{equation}\label{eq:overview-combined-controls-charge}
\varrho_y(R)\le C\bbE_y(4R).
\end{equation}
Thus $\bbE$ measures the vee approximation and, at a slightly larger scale, also controls the aggregate contribution of nearby necks.

At a fixed center, the Gaussian Weiss quantity is monotone; see \Cref{lem:fixed-center-gaussian-weiss-monotonicity}. More importantly, \Cref{lem:shifted-center-gaussian-weiss-almost-monotonicity} controls it when the center moves proportionally to the scale. Together with the radial trace estimate \eqref{eq:boundary-monneau-trace} for the vee-excess term, this yields
\begin{equation}\label{eq:overview-log-propagation}
\bbE_y(R;e)\le C|\log\eta|\,\bbE_y(\eta R;e), \qquad 0<\eta\le\tfrac12;
\end{equation}
see \Cref{prop:fixed-direction-monneau-doubling}; giving the corresponding logarithmic estimate for $\bbE_y$.

Thus $\bbE$ controls all gradient components tangent to the selected ridge and propagates logarithmically in scale. This links tangential stability to the center--scale selection below and is one of the paper's key contributions.

\subsection{Optimizing the power, center, and scale}

Next comes another key new idea of the paper. We compare the combined excess with the additive neck charge through
\begin{equation}\label{eq:overview-selection-functional}
F_u^\alpha(R) := \sup_{\zz\in\cZ_R} \frac{\bbE_\zz(8R)}{\varrho_\zz(2R)^\alpha}, \qquad \alpha_\star := \inf\left\{\alpha>0: \limsup_{R\to\infty}F_u^\alpha(R)=\infty\right\}.
\end{equation}
The exponent $\alpha$ measures the exchange rate between geometric excess and the ``neck charge''. Together with decay of the combined excess at infinity, \eqref{eq:overview-combined-controls-charge} gives the initial bound $\alpha_\star\le1$; see \Cref{lem:decay-of-monneau-excess,rem:critical-selection-exponent-at-most-one}.

We choose the power $\alpha$ before the center and scale:
\[
\alpha=
\begin{cases}
\dfrac{1+\max\{\alpha_\star,\frac12\}}2,&\alpha_\star<1,\\[2mm]
1+\dfrac1{100},&\alpha_\star=1.
\end{cases}
\]
Thus $\alpha_\star<\alpha\le1+\frac1{100}$.

With this $\alpha$ fixed, \Cref{lem:selection-of-initial-center-and-scale} gives large balls $B_{R_k}(\zz_k)$ on which $F_u^\alpha(R_k)$ nearly realizes its running supremum. Write $\eps_k:=\bbE_{\zz_k}(8R_k)$. A refined choice gives a ball $B_{\Rk}(\zk)$ and a second small parameter $\epk\le\eps_k$; see \Cref{lem:refined-center-and-scale-selection}. Besides $R_k,\Rk\to\infty$, we have $\eps_k,\epk\to0$ and $\bbE_{\zk}(8\Rk)\le C\epk$. Moreover, for every $0<\zeta<1/2$ and every $\mathcal F\subset\cZ_{\zeta\Rk}\cap B_{\Rk}(\zk)$ such that the balls $B_{\zeta\Rk}(\zz')$, $\zz'\in\mathcal F$, are pairwise disjoint,
\begin{equation}\label{eq:overview-selected-carleson}
\sum_{\zz'\in\mathcal F} \bbE_{\zz'}(8\zeta\Rk)^{2/\alpha} \le C_\alpha\zeta^{2-n}\epk^{2/\alpha}.
\end{equation}
On the ball tree, this Carleson-type estimate controls the aggregate vee error generation by generation; see \Cref{lem:fb-tree-inner-counting-package} for the resulting tree estimates.

The optimized denominator also prevents the excess from becoming too small at every scale $0<r\le8R_k$. Combining the near-maximal choice, recentering, and \eqref{eq:overview-log-propagation}, \Cref{prop:optimized-denominator-logarithmic-lower-bounds} gives
\begin{equation}\label{eq:overview-selected-lower-bound}
\bbE_{\zk}(r) \ge \frac{\eps_k}{C\bigl(1+\log(8R_k/r)\bigr)} \quad\text{for}\quad 0<r\le8R_k, \qquad \Rk\simeq R_k.
\end{equation}
The proof ends by comparing this lower bound with the improved upper bound.

\subsection{The construction of $U_+$ and $U_-$ and flux estimate}

At the selected scale, $u$ is close to a vee $V_{\zk,e}(x)=|e\cdot(x-\zk)|$. Thus one expects the positivity set to have two large portions on opposite sides of an approximate ridge, except near small regions where the two free-boundary sheets join. Sections~\ref{sec:fb-tree} and~\ref{sec:fb-tree-quantitative} make this picture quantitative: they construct two disjoint open sets $U_+$ and $U_-$, the two signed portions of the same positivity phase, and a controlled union $\mathcal X_{\rm neck}$ of small neck balls such that
\[
\bigl(\{u>0\}\cap B_{\Rk/2}(\zk)\bigr)\setminus\mathcal X_{\rm neck} \subset U_+\cup U_- \subset \{u>0\}\cap B_{\Rk/2}(\zk).
\]
Away from $\mathcal X_{\rm neck}$, the free boundary consists of two nearly flat sheets, one associated with each signed region; inside the neck balls the two sides may reconnect. Above the neck scale, $U_+$ and $U_-$ occupy opposite sides of a thin slab, and every localized signed region lies in a uniform John domain; see \Cref{lem:fb-tree-signed-decomposition,lem:fb-tree-mesoscopic-signed-slab-flatness,lem:fb-tree-localized-raw-signed-shadows-john}.

We build these sets using a dyadic ball tree that refines the ball tree of \cite{CFFS25}*{Proposition~5.12}.

The free boundary is covered by families of dyadic balls. Each free-boundary ball $B = B_{r_B}(x_B)$ carries a best centered vee $V_B$, its dimensionless error
\[
h_B^2 := \inf_{e\in\mathbb S^{n-1}} \frac{1}{r_B^{n+2}} \int_{\lsup{16}{B}} \big|u-V_{x_B,e}\big|^2\,dx.
\]
(so $h_Br_B$ is the corresponding height error), and an oriented slope one plane $\ell_B$ such that $V_B = |\ell _B|$.

The tree starts from the root ball $B_{\rm root}=B_{\Rk}(\zk)$. Whenever a ball sees a neck center and remains close to a vee ($h_B<\delta$ small), it branches, and children at the next scale cover its approximate ridge. Otherwise that branch stops: a regular terminal ball contains two smooth, nearly flat sheets, while a neck-like terminal ball isolates a possible joining region. Parent--child comparison carries one choice of sign down from the root, while the no-hole property ensures that no part of the ridge is lost between generations. The resulting same-sign pieces glue to form $U_+$ and $U_-$. See \Cref{defn:fb-tree-dyadic-balls-centered-vee-excess,defn:fb-tree-branching-and-terminal-balls,lem:fb-tree-properties,lem:fb-tree-regular-terminal-balls-split-into-two-signed-sheets,lem:fb-tree-orientation-compatibility}.

The tree also lets us pass from \eqref{eq:overview-selected-carleson} on neck-centered balls to the regular terminal balls needed for the first flux estimate. Indeed, \eqref{eq:overview-selected-carleson} controls the heights generation by generation; predecessor comparison sums the heights along a branch to control its tilt; and the descendant estimate sums these errors over all terminal leaves. For every flux-admissible (see \eqref{eq:flux-admissible}) $p>1$, the form we need below is
\begin{equation*}\tag{2.9$'$}
\label{eq:overview-terminal-tree-additive} \sum_{\substack{B\ {\rm regular\ terminal}\\
B\ {\rm in\ the\ working\ region}}} r_B^{n-1}\bigl(h_B^{2p}+|\nabla(\ell_B-\ell_{B_{\rm root}})|^{2p}\bigr) \le C_{p,\alpha}\Rk^{n-1}\epk^{2p}.
\end{equation*}
The generation and descendant bounds in \Cref{lem:fb-tree-inner-counting-package} yield this transfer in the proof of \Cref{prop:fb-tree-sided-higher-power-transverse-estimate}. On a regular terminal sheet, the Bernoulli condition makes the normal-flux error quadratic in precisely these height and tilt errors:
\[
1\mp \nabla \ell_{B_{\rm root}}\cdot\nabla u =\tfrac12|\nabla u\mp \nabla \ell_{B_{\rm root}} |^2 \lesssim h_B^2+|a_B-\nabla \ell_{B_{\rm root}} |^2.
\]
Consequently, \eqref{eq:overview-terminal-tree-additive} and the quadratic identity give the regular part of the following flux bound. The boundary-area estimate in \Cref{lem:fb-tree-terminal-neck-bookkeeping}, together with \Cref{lem:lipschitz-bound}, gives the same bound for the artificial boundaries (i.e., the parts of $\partial U_\pm$ surrounding the neck-like terminal balls, which may not fall on $\partial\{u>0\}$):
\begin{equation}\label{eq:overview-sided-flux}
\int_{\Gamma_+}\bigl|\partial_{\nu}(u-\ell_{B_{\rm root}})\bigr|^p \,d\cH^{n-1} {}+\int_{\Gamma_-}\bigl|\partial_{\nu}(u+\ell_{B_{\rm root}})\bigr|^p \,d\cH^{n-1} \le C_{p,\alpha}\Rk^{n-1}\epk^{2p}.
\end{equation}
This is \Cref{prop:fb-tree-sided-higher-power-transverse-estimate}. In particular, the key power $2p$ in \eqref{eq:overview-sided-flux} is a direct consequence of \eqref{eq:overview-selected-carleson}, through its terminal-tree form \eqref{eq:overview-terminal-tree-additive}, and of the quadratic nature of the regular-sheet flux.

For $w=u-\ell_{B_{\rm root}}$, \eqref{eq:overview-sided-flux} controls the weak Neumann residual of $w$. Combined with the slab geometry and the John-domain Sobolev--Poincar\'e estimate, this gives the sided $W^{1,2}$ control used in the first linearization; see \Cref{lem:direct-sided-local-flux-slab,prop:direct-sided-w12-from-flux}.

Together, the centered vee estimate \eqref{eq:fb-tree-centered-vee-excess}, the root slab \eqref{eq:direct-sided-root-slab}, and the root estimate in \Cref{lem:fb-tree-small-neck-radii-near-flat-active-balls} give an initial $L^2$ affine approximation on each side. We then rescale and subtract these affine functions. The slab and flux estimates give an even harmonic limit on a half-space, and the Neumann compactness lemmas provide the affine improvement used in the iteration; see \Cref{lem:linear-neumann-compactness-half-ball-l2,lem:linear-neumann-compactness-global}.

For this first step we take a fixed flux-admissible $p_1<3/2$. Fixing the small exponent $\cttc>0$, set
\begin{equation}\label{eq:overview-final-flat-scale}
R_\flat:=\epk^\cttc\Rk.
\end{equation}
The compactness iteration in \Cref{prop:asymmetric-excess-decay-almost-critical} constructs affine maps $L_\pm^\flat(x)=a_\pm^\flat\cdot x+b_\pm^\flat$ with quantitative $W^{1,2}$ error at the flat scale.

\subsection{Breaking the dimension-four criticality}

In dimension four, at the benchmark $\alpha=1$ (which corresponds to the setups used in \cites{CFFS25, FS25}) both admissibility conditions in \eqref{eq:flux-admissible} reduce to
\[
2p<3.
\]
Thus this benchmark only gives $p<3/2$. At the formal endpoint $p=3/2$, the two terms in the flux estimate \eqref{eq:direct-sided-phi-flux-bound}, evaluated at $r=\epk\Rk$, both have size $(\epk\Rk)^2$: exactly the size of the oscillation one is trying to improve. The estimate reaches the critical scale but gives no contraction below it, which is needed to finish the proof.

At the flat scale $R_\flat$, \Cref{lem:asymmetric-plane-dichotomy} gives two possibilities for the affine approximations: either they form one vee to higher order, or their positive half-spaces are disjoint. We use tangential translation fields in the parallel configurations and a dilation field centered at the intersection of the two zero hyperplanes in the nonparallel case; see \Cref{lem:tangential-sternberg-zumbrun-inequality,lem:dilation-sternberg-zumbrun-inequality}. In either case, the neck-scale gaps in \Cref{lem:positive-excess-at-the-neck-scale,lem:positive-dilation-gap-at-the-neck-scale}, and the flat-scale estimate \eqref{eq:asymmetric-excess-flat-scale-affine-w12} control the ``total neck charge''. More precisely, for every $\zz\in\cZ\cap B_{\Rk/16}(\zk)$, \Cref{prop:flat-scale-asymmetric-planes-control-neck-radii} gives
\begin{equation}\label{eq:overview-flat-neck-improvement}
\varrho_\zz(R_\flat) \le C\epk^{1+\cttc/6}.
\end{equation}
At the selected centers, the numerator in \eqref{eq:overview-selection-functional} retains its logarithmic lower bound, whereas \eqref{eq:overview-flat-neck-improvement} gives a strict power gain in the denominator. The precise comparison in \Cref{cor:flat-scale-neck-charge-improves-critical-exponent} gives
\begin{equation}\label{eq:overview-alpha-strict}
\alpha_\star\le\frac1{1+\cttc/6}<1.
\end{equation}
By the choice above, this implies $\alpha<1$. We can therefore introduce the second flux exponent
\[
p_2:=1+\frac1{2\alpha}>\frac32.
\]
The parameter check in \Cref{rem:above-critical-flux-exponent} shows that it is flux-admissible; this is the exponent used in the final iteration.

\subsection{Below-flat-scale decay and the final contradiction}

After we have broken criticality we can finish the proof largely following \cite{CFFS25} again:  With $p_2>3/2$, \eqref{eq:below-critical-parameter-gaps} gives strict exponent inequalities that make both local flux errors smaller by a positive power of $\epk$. Fix $\sigma=\cttc/100$.

With this gain in hand, we run the Neumann compactness iteration from $R_\flat$ to $\epk^{1+\sigma}R_\flat$. \Cref{prop:below-critical-frozen-plane-oscillation} shows that the flat-scale slopes stay fixed while the constant terms converge geometrically. Throughout this range, the resulting sided $L^2$ and $W^{1,2}$ errors are $O(\epk^{1+\sigma}R_\flat)$.

At the end of this iteration, the flat-scale dichotomy \Cref{lem:asymmetric-plane-dichotomy} and the centered intercept estimates show that the two affine maps may be replaced by the exact opposite pair
\[
L_{k,\pm}^\flat(x)=\pm e_k\cdot(x-\zk), \qquad e_k\in\mathbb S^3,
\]
without changing the error bounds. The same comparison gives
\[
u\ge |e_k\cdot(x-\zk)|-C\epk^{1+\sigma} R_\flat \quad\text{in }B_{5R_\flat}(\zk).
\]
These conclusions are included in \Cref{prop:below-critical-frozen-plane-oscillation}.

It remains to convert the sided estimates into the combined excess \eqref{eq:overview-combined-excess}. Set $\tau:=\sigma/100$. The frozen-plane estimates in \Cref{prop:below-critical-frozen-plane-oscillation}, together with the radius, boundary-localization, and volume bounds of \Cref{lem:fb-tree-terminal-neck-bookkeeping}, control the vee-excess term. The two-sheet calibration in \Cref{lem:combined-excess-below-flat-scale}, together with the universal vee identity in \Cref{lem:universal-gaussian-weiss-limit}, controls the ridgewise Weiss and Monneau deficits and gives
\begin{equation}\label{eq:overview-final-upper}
\bbE_{\zk}(\epk^\tau R_\flat) \le C\epk^{1+\sigma/2}.
\end{equation}

The upper bound \eqref{eq:overview-final-upper} now contradicts the logarithmic lower bound retained from the original selection. Indeed, \eqref{eq:overview-selected-lower-bound}, $\Rk\simeq R_k$, and $\epk\le\eps_k$ give, after division by $\eps_k$,
\[
1 \le C\biggl(1+\log\bigg(\frac{8R_k}{\epk^\tau R_\flat}\bigg)\biggr) \frac{\epk}{\eps_k}\epk^{\sigma/2} \le C|\log\epk|\,\epk^{\sigma/2} \to0,
\]
which is impossible. This contradiction proves \Cref{thm:main-classification}.

\section{The Bernoulli Problem: Preliminaries}
\label{sec:Bernoulli}
\subsection{Notation}
\label{ssec:notation}

We begin by fixing the notation and conventions used throughout the paper.

\begin{notation}
\label{notation:s1-notation-1}
Throughout the paper, $C>1$ and $c\in(0,1)$ denote generic constants chosen conveniently large and small, respectively. Dependencies are denoted by subscripts or parentheses.

With $B_r(y)$ we denote the ball of radius $r>0$ centered at $y$. When $y=0$, we also write $B_r$ in place of $B_r(0)$. By $B_r(A)$ we denote the $r$-fattening of a set $A\subset \R^n$, namely $B_r(A):=\{x\in \R^n:\dist(x,A)<r\}$, which can also be seen as the Minkowski sum $B_r+A$. We always assume that a modulus of continuity $\omega$ satisfies
\begin{equation}
\label{eq:w} \omega:[0, +\infty)\to [0, +\infty)\ \text{is increasing, with $\omega(t) \ge t$ for all $t\ge 0$}.
\end{equation}
Given $y\in \R^n$ and $e\in \mathbb{S}^{n-1}$, we denote by $V_{y,e}$ a {\em vee}, namely, a function of the form
\begin{equation}
\label{eq:vee_def} \R^n \ni x\mapsto V_{y,e}(x) : = |e\cdot(x-y)|.
\end{equation}
Finally, $\mathcal{H}^{k}$ denotes the $k$-dimensional Hausdorff measure.
\end{notation}

{
For the reader's convenience, we collect below the principal manuscript-specific notation. The entries are grouped according to their role in the proof.

\footnotesize \setlength{\tabcolsep}{4pt} \renewcommand{\arraystretch}{1.08}

\par\medskip
\noindent
\begin{tabular}{@{}p{.26\textwidth}p{.68\textwidth}@{}}
\textbf{Symbol} & \textbf{Meaning/Location}\\
\hline
\multicolumn{2}{@{}l}{\emph{Basic and neck geometry}}\\[1pt]
$u,\{u>0\},\FB(u),\nu$ & Bernoulli solution, positivity set, and free boundary; on $\FB(u)$, $\nu=\nabla u$ is the inward unit normal. See \eqref{eq:intro-bernoulli} and \Cref{defn:bernoulli-jacobi-field}.\\[2pt]
$V_{y,e}$ & Centered unit-slope vee, $V_{y,e}(x)=|e\cdot(x-y)|$; see \eqref{eq:vee_def}.\\[2pt]
$\nAC$ & Critical dimension for nonflat $1$-homogeneous stable Bernoulli cones; see \Cref{defn:critical-dimension}.\\[2pt]
$\eta_0,\rmin$ & Hessian-energy threshold and uniform lower bound for threshold radii; see \eqref{eq:eta0} and \eqref{eq:rmin_def}.\\[2pt]
$\rb(y),\cZ,\zz,\cZ_R$ & Threshold radius (called neck radius at a neck center), the selected set of neck centers and its elements, and the centers with neck radius at most $R$; see \eqref{eq:def_rB}, \Cref{defn:neck-centers}, and \eqref{eq:ZR}.\\
\hline
\end{tabular}

\par\smallskip
\noindent
\begin{tabular}{@{}p{.26\textwidth}p{.68\textwidth}@{}}
\multicolumn{2}{@{}l}{\emph{Defects, energies, and excesses}}\\[1pt]
$\nabla^e u,\cA_e(u)$ & Gradient transverse to $e$ and the associated Sternberg--Zumbrun defect; see \Cref{lem:tangential-sternberg-zumbrun-inequality}.\\[2pt]
$\widetilde\varrho_x(R),\bE_x(R)$ & Transverse energy and symmetric $L^2$-gradient excess; see \eqref{eq:tilde_varrho_def_intro} and \eqref{eq:Ez_def_intro}.\\[2pt]
$\varrho_x(R)$ & Square root of the additive, scale-normalized neck-radius charge; see \eqref{eq:varrho_def_intro}.\\[2pt]
${\bf W}(u,r),\alpha_n$ & Boundary-adjusted Weiss energy on a ball and its half-space density; see \eqref{eq:W-def} and \eqref{eq:alpha_n}.\\[2pt]
$P(x),G_r,H_{r,x},D_{r,x},W_{r,x}$ & Bernoulli half-energy density, Gaussian kernel, Gaussian energy and mass, and the resulting Gaussian Weiss energy; see \Cref{defn:gaussian-weiss-quantities}.\\[2pt]
$W_\infty,\wW_{r,x}$ & Universal Gaussian Weiss value of a unit vee and the deficit $W_\infty-W_{r,x}$; see \Cref{lem:universal-gaussian-weiss-limit} and \Cref{lem:gaussian-weiss-deficit-bound}.\\[2pt]
$\Theta_{R,y}^e,M_{R,y}^e, \bbE_y(R;e),\bbE_y(R)$ & Ridgewise Weiss deficit, Monneau-type $L^2$ vee error, directional combined excess, and its minimum in $e$; see \Cref{defn:monneau-excess} and \eqref{eq:def-fixed-direction-Monneau}.\\
\hline
\end{tabular}

\par\smallskip
\noindent
\begin{tabular}{@{}p{.26\textwidth}p{.68\textwidth}@{}}
\multicolumn{2}{@{}l}{\emph{Selection of power, center, and scale}}\\[1pt]
$\beta_\circ$ & Fixed exponent governing the covering estimate from the selection onward; see \eqref{eq:beta_circ_def}.\\[2pt]
$F_u^\alpha(R),\alpha_\star$ & Selection functional and its critical exponent; see \eqref{eq:selection-functional} and \Cref{defn:selection-functionals}.\\[2pt]
$\alpha$ & Exponent chosen strictly above $\alpha_\star$, with a prescribed choice when $\alpha_\star=1$; see \eqref{eq:single-selection-alpha-choice}.\\[2pt]
$(R_k,\zz_k,\eps_k)$ & Initially selected radius, neck center, and combined excess; see \eqref{eq:eps_k}.\\[2pt]
$(\tilde\zeta_k,\Rk,\zk,\epk)$ & Refined scale factor, radius, neck center, and small parameter; see \Cref{lem:refined-center-and-scale-selection}.\\[2pt]
$\bN(\zeta,B_R(\zz)), \mathcal A^\zeta_{\zz,R}$ & Normalized covering size and the neck centers being covered; see \eqref{eq:Ndef}. \\
\hline
\end{tabular}

\par\smallskip
\noindent
\begin{tabular}{@{}p{.26\textwidth}p{.68\textwidth}@{}}
\multicolumn{2}{@{}l}{\emph{Free-boundary tree and signed geometry}}\\[1pt]
$B_{\rm root},r_l,r_B,x_B,\lsup{\lambda}{B},\cQ^l$ & Root ball, dyadic radius, radius and center of an individual ball, concentric dilation, and the ambient dyadic free-boundary family; see \Cref{defn:fb-tree-dyadic-balls-centered-vee-excess}.\\[2pt]
$h_B,V_B$ & Centered vee error of $B$ and a minimizing centered vee; see \eqref{eq:fb-tree-centered-vee-excess} and \Cref{defn:fb-tree-dyadic-balls-centered-vee-excess}.\\[2pt]
$\delta;\ \cD^l,\cB^l,\cR^l,\cN^l;\ \cB,\cR,\cN$ & Fixed tree parameter; active, branching, regular-terminal, and neck-terminal balls at generation $l$; and the corresponding all-generation families. See \Cref{defn:fb-tree-branching-and-terminal-balls}.\\[2pt]
$\operatorname{Ch}(B),PB,P^jB$ & Local children, chosen predecessor, and iterated predecessor in the tree; see \eqref{eq:fb-tree-local-children} and \Cref{defn:fb-tree-branching-and-terminal-balls}.\\[2pt]
$\ell_B,a_B,e_{\rm root}$ & Oriented unit-slope affine representative of $V_B$, its slope, and the oriented root direction; see \eqref{eq:fb-tree-oriented-affine-representatives} and \eqref{eq:fb-tree-root-direction}.\\[2pt]
$c_\ast,B^\pm,U_\pm,\mathcal X_{\rm neck}$ & Fixed signed-core cutoff, signed local pieces, the two assembled signed domains, and the terminal-neck region; see \eqref{eq:fb-tree-branching-signed-pieces} and \Cref{defn:fb-tree-signed-pieces}.\\[2pt]
$\cD_{\rm in}^j,\cN_{\rm in}, \Gamma_\pm,\Gamma_\pm^{\rm reg},\Gamma_\pm^{\rm neck}$ & Inner active family, relevant terminal neck balls, the localized reduced boundary of $U_\pm$, and its regular and neck parts; see \eqref{eq:fb-tree-inner-family}, \eqref{eq:fb-tree-inner-neck-family}, and \eqref{eq:fb-tree-signed-boundary-pieces}.\\
\hline
\end{tabular}

\par\smallskip
\noindent
\begin{tabular}{@{}p{.26\textwidth}p{.68\textwidth}@{}}
\multicolumn{2}{@{}l}{\emph{Flux exponents and final scales}}\\[1pt]
$\gamma_\alpha,\ p\text{ flux-admissible}$ & Terminal-neck gain and the permitted range of flux exponents; see \eqref{eq:gamma-alpha} and \eqref{eq:flux-admissible}.\\[2pt]
$p_1,\cttc,R_\flat$ & First flux exponent, flat-scale exponent, and flat scale; see \eqref{eq:asym-decay-first-flux-exponent}, \eqref{eq:asym-decay-flat-scale-chi-smallness}, and \eqref{eq:flat-scale-definition}.\\[2pt]
$a_\pm^\flat,b_\pm^\flat,L_\pm^\flat,\theta_\flat$ & Flat-scale affine coefficients and maps, together with the improved flat-plane comparison size; see \eqref{eq:asymmetric-excess-flat-scale-normalization}, \Cref{prop:asymmetric-excess-decay-almost-critical}, and \eqref{eq:asymmetric-plane-dichotomy-theta-flat}.\\[2pt]
$p_2,\sigma,A_\#$ & Dimension-four flux exponent, small final exponent, and lower endpoint of the second iteration; see \eqref{eq:below-critical-second-flux-exponent}, \eqref{eq:sigma-def}, and \eqref{eq:below-critical-lower-endpoint}.\\[2pt]
$e_k,L_{k,\pm}^\flat$ & Frozen direction and centered opposite affine pair used below the flat scale; see \Cref{prop:below-critical-frozen-plane-oscillation}.\\[2pt]
$\tau,\rho_k$ & Gaussian-calibration exponent and radius, $\tau=\sigma/100$ and $\rho_k=\epk^\tau R_\flat$; see \Cref{lem:combined-excess-below-flat-scale}.\\
\hline
\end{tabular}
\par\medskip
}

\subsection{The notions of solution}
\label{ssec:notion}

We recall the classical, stationary, and stable notions of solution used throughout the paper.

Let $\Omega\subset\R^n$ be open, and let $\cE(u;\Omega)$ be the Alt--Caffarelli functional defined in \eqref{eq:Alt-Caffarelli-functional}. Critical points of $\cE$ solve the one-phase Bernoulli problem.

We are interested in \emph{classical solutions} of the Bernoulli problem: functions $u: \Omega \to \R_+ := [0,\infty)$ such that
\begin{equation}\label{eq:Bernoulli-main}
\text{$\{u > 0\}$ is locally a smooth domain in $\Omega$}\qquad\text{and}\qquad
\begin{cases}
\Delta u=0 & \text{ in } \Omega \cap \set{u>0},\\
|\nabla u|=1 & \text{ on } \Omega \cap \partial\set{u>0}.
\end{cases}
\end{equation}
The set $\partial\{u > 0\}$ is called the \emph{free boundary} and will also be denoted $\FB(u)$. In particular, for a classical solution $\{u > 0\}$ is locally the subgraph of a smooth function around each free boundary point (up to a rotation).

Given $\xi\in C_c^\infty(\Omega;\R^n)$, let $D\Subset\Omega$ be any bounded open set containing $\supp\xi$. Classical solutions $u$ are \emph{stationary critical points} of $\cE$, that is, they satisfy
\begin{equation}
\label{eq:stationary} \frac{d}{dt}\bigg|_{t = 0} \cE(u\circ\Psi_t;D) = 0 \quad \mbox{ for every $\Psi_t(x) := x+t \xi (x)$ with $\xi \in C^\infty_c(\Omega; \R^n)$.}
\end{equation}
Stationary critical points $u$ are called \emph{stable} if they have nonnegative second (inner) variations; that is, they satisfy
\begin{equation}\label{eq:stability-var}
\dfrac{d^2}{dt^2}\bigg|_{t=0} \cE(u\circ\Psi_t;D)\geq 0, \quad \mbox{ for every $\Psi_t(x) := x+t \xi (x)$ with $\xi \in C^\infty_c(\Omega; \R^n)$.}
\end{equation}

A \emph{solution} will always refer to the one-phase Bernoulli (or Alt--Caffarelli) problem. Moreover, we will distinguish among the following notions:

\begin{defn}[Notions of solution]
\label{defn:notions-of-solution}
Let $n \ge 2$ and let $\Omega\subset\R^n$ be open. In relation to the one-phase Bernoulli problem, we say that a nonnegative $u\in H^1_{\rm loc}(\Omega)$ is:
\begin{itemize}
\item a \emph{stationary solution} (or simply \emph{stationary}) in $\Omega$ if it satisfies \eqref{eq:stationary};
\item a \emph{classical solution} or a \emph{classical critical point} in $\Omega$ if it satisfies \eqref{eq:Bernoulli-main} (in particular, it is stationary);
\item a \emph{stable solution} in $\Omega$ if it is stationary and satisfies \eqref{eq:stability-var};
\item a \emph{classical stable solution} or \emph{classical stable critical point} in $\Omega$ if it satisfies \eqref{eq:Bernoulli-main} and \eqref{eq:stability-var}.
\end{itemize}
If $\Omega=\R^n$, we call the corresponding solution \emph{global}.
\end{defn}

\subsection{Some results}

Throughout the work we will use various results from \cite{CFFS25} repeatedly. For the reader's convenience, we collect here the basic regularity, perimeter, and vee-comparison estimates used below.

\begin{lem}[Epsilon-regularity, {\cite{CFFS25}*{Lemma~3.7}}]
\label{lem:epsilon-regularity}
Let $n\geq 2$. There exists $\eps_\circ=\eps_\circ(n) > 0$ such that the following holds. Let $0<\eps\le\eps_\circ$, and let $u$ be a classical solution to the Bernoulli problem in $B_1\subset \R^n$. If
\begin{equation}\label{eq:eps-reg-classical-flat}
\norm[L^\infty(B_1 \cap \set{u>0})]{u-x_n}\leq \eps,
\end{equation}
then, for any $k\in \N$, there exists $C_{n,k} = C_{n, k}(n, k) > 0$, such that
\[
\norm[C^k(B_{1/2} \cap \set{u>0})]{u-x_n}\leq C_{n,k}\eps\qquad \text{and}\qquad \text{$\FB(u)\cap B_{1/2}$ is a $C^k$ graph, with $C^k$-norm bounded by $C_{n,k}\eps$.}
\]
Moreover, $u$ is analytic in $B_{1/2} \cap \set{u>0}$.
\end{lem}

\begin{lem}[Hessian Epsilon-regularity, {\cite{CFFS25}*{Lemma~3.11}}]
\label{lem:epsilon-regularity-for-the-hessian}
Let $n\ge 2$. There exists $\eta_*=\eta_*(n)>0$ such that, for all $0<\eta\le \eta_*$, the following holds. Let $u$ be a classical solution to the Bernoulli problem in $B_1\subset \R^n$. Then
\begin{equation}\label{eq:eps_reg_Hess}
\int_{B_1 \cap \set{u>0}}|D^2u|^n\,dx \leq \eta^n \quad \implies \quad \norm[L^\infty(B_{1/2} \cap \set{u>0})]{D^2u} \leq C_n \eta,
\end{equation}
for some $C_n = C_n(n)$.
\end{lem}

\begin{lem}[Lipschitz bound, {\cite{CFFS25}*{Lemma~3.2}}]
\label{lem:lipschitz-bound}
Let $n \ge 2$ and let $u$ be a classical solution to the Bernoulli problem in $\R^n$. Then $|\nabla u|\le 1$ in~$\R^n$.
\end{lem}

\begin{lem}[Perimeter bound, {\cite{CFFS25}*{Lemma~3.3}}]
\label{lem:perimeter-bound}
Let $n\ge 2$, and let $u$ be a global classical solution to the Bernoulli problem in $\R^n$. Then, we have that for any $R> 0$ and $y\in \{u=0\}\cap B_R$,
\[
\cH^{n-1} \left( \FB(u)\cap B_\varrho(y)\right) \le C \varrho^{n-1}\qquad\text{for all}\quad \varrho \in (0, R/2),
\]
for some $C = C(n)$.
\end{lem}

\begin{cor}[{\cite{CFFS25}*{Lemma~3.15}}]
\label{cor:vee-plus-hessian-implies-regularity}
Let $n\ge 2$. Given $C_1\ge 1$ there exists $\eps_1 = \eps_1(n, C_1)>0$ such that the following holds. Let $\varrho>0$, $0<\eps\le\eps_1$, and let $u$ be a global classical solution to the Bernoulli problem in $\R^n$. Suppose that
\[
|D^2u| \le C_1\varrho^{-1} \quad\text{in}\quad B_{2\varrho}\cap \{u>0\}\qquad \text{and}\qquad \big |u- V_{0,e_n}\big|\le \eps\varrho\quad \mbox{in}\quad B_{2\varrho},
\]
where $e_n$ is the $n$-th vector in the canonical basis. Then
\[
\varrho^2\|D^2 u\|_{L^\infty(\{u > 0\}\cap B_{\varrho})}\leq C \eps \varrho
\]
for some $C = C(n)$. Moreover,
\[
\{u>0\} = \{x_n > g^{(+)}(x_1, \dots , x_{n-1})\} \cup \{x_n < g^{(-)}(x_1, \dots , x_{n-1})\} \qquad\text{in}\quad B_\varrho,
\]
where $g^{(\pm)} : D_\varrho \to \mathbb{R}$ with $D_\varrho$ being the lower dimensional ball $\{ x_1^2 + \cdots + x_{n-1}^2 < \varrho^2 \}$ in $\mathbb{R}^{n-1}$, $g^{(-)}< g^{(+)}$, and
\[
\|g^{(\pm)}\|_{L^\infty(D_\varrho)} + \varrho^2 \|D^2 g^{(\pm)}\|_{L^\infty(D_\varrho)} \leq C \eps \varrho.
\]
\end{cor}

\begin{lem}[{\cite{CFFS25}*{Lemma~3.14}}]
\label{lem:consequences-of-closeness-to-a-vee}
Let $n\ge 2$. There exists $\eps_1=\eps_1(n)>0$ such that the following holds. Let $u$ be a global classical solution to the Bernoulli problem in $\R^n$. Suppose that for some $y_1\in \R^n$, $r_1>0$, $\eps>0$, and $\bar e\in \mathbb S^{n-1}$, we have
\begin{equation}\label{closenessep100}
\| u -V_{y_1, \bar e} \|_{L^\infty(B_{r_1}(y_1))} \le \eps r_1.
\end{equation}
Then
\begin{equation}\label{closenessepzero}
\{u=0\}\cap B_{r_1}(y_1) \subset \{x\in B_{r_1}(y_1)\ : \ |\bar e\cdot(x-y_1)|\le \eps r_1\}.
\end{equation}
Moreover:

\begin{enumerate}[(a)]
\item For all $y_2\in \{u=0\}$ and $r_2>0$ such that $B_{r_2}(y_2)\subset B_{r_1}(y_1)$, we have
\[
\| u -V_{y_2, \bar e} \|_{L^\infty(B_{r_2}(y_2))} \le 2\eps r_1.
\]
\item If $\eps\le\eps_1(n)$ then, for some $C= C(n)$,
\[
\dist(x, \{u = 0\}) \le C \eps r_1,\qquad\text{for all}\quad x\in \{z\in B_{r_1/2}(y_1):|\bar e\cdot(z-y_1)|\le \eps r_1\}.
\]
\end{enumerate}
\end{lem}

\section{Blow-down of global stable solutions}
\label{sec:blowdown}

The goal of this section is to prove that, in $\R^n$ for $n < \nAC$ (see \Cref{defn:critical-dimension} below), non-flat global stable solutions look like a vee at large scales. We first isolate half-space regularity for limits of classical stable solutions, then classify their homogeneous limits. The almost-homogeneity and blow-down statements will follow directly from this classification.

\subsection{Stable cones and homogeneous limits}

We introduce the critical dimension and classify the homogeneous stable limits that arise in the blow-down analysis.

\begin{defn}[Critical dimension]
\label{defn:critical-dimension}
From now on, we define
\[
\nAC := \min\left\{ n\ge 2 :
\begin{array}{l}
\text{there exists a global $1$-homogeneous stable solution $v$ in $\R^n$, with $0\in\FB(v)$,}\\
\text{classical away from the origin, and not equal to $(e\cdot x)_+$ for any $e\in\mathbb S^{n-1}$.}
\end{array}
\right\}.
\]
\end{defn}

From the (now classical) literature on the one-phase problem, the results of Jerison--Savin and De Silva--Jerison imply the following range of possible values for $\nAC$:

\begin{prop}[{\cites{Jerison-Savin,DJ09}}]
\label{prop:critical-dimension-bigger-than-four}
One has
\[
5\le \nAC\le 7.
\]
\end{prop}

For the compactness arguments below, we will also use the following quantitative nondegeneracy estimate.

\begin{lem}[Nondegeneracy, {\cite{CFFS25}*{Lemma~4.2}}]
\label{lem:nondegeneracy-of-stable-solutions}
Let $n\geq 2$, and let $u$ be a global classical stable solution to the Bernoulli problem in $\R^n$. Then, for all $y\in \partial{\{u > 0\}}$ and $r>0$, there exists $c = c(n)$ such that
\begin{equation}\label{eq:density-B}
\fint_{\partial B_r(y)}u \,d\cH^{n-1} \geq c r \quad \mbox{ and }\quad \mathcal{H}^{n-1}(\partial\{u > 0\}\cap B_r(y)) \ge c r^{n-1}.
\end{equation}
\end{lem}

We introduce an important monotone quantity:

\begin{defn}[Weiss energy and monotonicity]
\label{defn:weiss-energy-and-monotonicity}
For $u\in H^1(B_r)$ and $0\in \FB(u)$, the Weiss boundary-adjusted energy (see \cite{Weiss1998}) is given by
\begin{equation}
\label{eq:W-def} {\bf W}(u,r) =\frac{1}{r^n} \int_{B_r} (|\nabla u|^2+\mathbbm{1}_{\set{u>0}}) \,dx -\frac{1}{r^{n+1}} \int_{\partial B_r} u^2 \,d\cH^{n-1} = {\bf W}(u_r, 1),
\end{equation}
where $u_r$ denotes the natural dilation of $u$, namely
\[
u_r(x):=\frac{u(rx)}{r}, \quad \text{ for } r>0.
\]
By the Weiss monotonicity formula (see \cite{Weiss1998}*{Theorem~3.1}), if $u\in H^1(B_R)$ is a stationary solution to the Bernoulli problem, then
\[
r\mapsto {\bf W}(u,r)\quad\text{ is non-decreasing on $(0,R)$}
\]
and
\begin{align}\label{eq:W-monotone}
\partial_r{\bf W}(u,r) &= \frac{2}{r^{n+2}} \int_{\partial B_r} (u-x\cdot \nabla u)^2 \,d\cH^{n-1} = \frac{2}{r}\int_{\partial B_1} (u_r - x\cdot \nabla u_r)^2 d\mathcal{H}^{n-1} \ge 0\quad \text{ for a.e.}\quad r\in(0,R)
\end{align}
(see also \cite{Vel23}*{Section 9}).
\end{defn}

We first evaluate this energy on $1$-homogeneous solutions.
\begin{lem}
\label{lem:homogeneous-weiss-upper-bound}
Let $v$ be a Lipschitz $1$-homogeneous stationary solution to the Bernoulli problem in $\R^n$, with $0\in\FB(v)$. Then
\[
{\bf W}(v,1)=|\{v>0\}\cap B_1|.
\]
\end{lem}

\begin{proof}
Let $\Omega:=\{v>0\}$. For almost every $\varepsilon>0$, integrate by parts in $\{v>\varepsilon\}\cap B_1$ against $v-\varepsilon$, and then let $\varepsilon\downarrow0$. The contribution from $\{v=\varepsilon\}$ vanishes, while $1$-homogeneity gives $\partial_\nu v=v$ on $\partial B_1$. Therefore
\[
\int_{B_1}|\nabla v|^2\,dx =\int_{\partial B_1\cap\Omega}v\,\partial_\nu v\,d\cH^{n-1} =\int_{\partial B_1}v^2\,d\cH^{n-1}.
\]
Substituting this equality into \eqref{eq:W-def} proves the identity.
\end{proof}

Set
\begin{equation}\label{eq:alpha_n}
\alpha_n:=\frac12|B_1| =\frac{\cH^{n-1}(\mathbb S^{n-1})}{2n}.
\end{equation}
By \Cref{lem:homogeneous-weiss-upper-bound}, for every $e\in\mathbb S^{n-1}$,
\[
{\bf W}((e\cdot x)_+,1)=\alpha_n, \qquad {\bf W}(|e\cdot x|,1)=2\alpha_n.
\]

To prove \Cref{lem:almost-homogeneous-solutions} we will need the following compactness result for sequences of stable solutions.

\begin{lem}[Compactness]
\label{lem:compactness}
Let $n \geq 2$, and let $v_k \in C^{0,1}_{\rm loc}(B_k)$ be a sequence of classical stable solutions to the Bernoulli problem in $B_k \subset \R^n$, with $0 \in \FB(v_k)$ for all $k \in \N$. Then the following hold:
\begin{enumerate}
\item Up to a subsequence, $v_k$ converges to some function $v_\infty$ satisfying $|\nabla v_\infty| \leq 1$ in $\R^n$, with strong convergence in $(H^1_{\rm loc} \cap C^{0,\alpha}_{\rm loc})(\R^n)$ for all $\alpha \in (0,1)$.

\item \label{it:compact-2} The sets $\overline{\{v_k>0\}}$, $\{v_k=0\}$, and the free boundaries $\FB(v_k)$, converge locally in the Hausdorff distance in $\R^n$ to their corresponding sets for $v_\infty$ (up to a subsequence). Specifically,
\[
\overline{\{v_k>0\}} \to \overline{\{v_\infty>0\}}, \quad \{v_k=0\} \to \{v_\infty=0\}, \quad \text{and} \quad \FB(v_k) \to \FB(v_\infty),\quad\text{locally}.
\]
Moreover,
\[
\mathbbm 1_{\{v_k>0\}}\to \mathbbm 1_{\{v_\infty>0\}} \qquad\text{in }L^1_{\rm loc}(\mathbb R^n).
\]

\item The limit function $v_\infty$ is a stable solution in the sense of \Cref{defn:notions-of-solution}.
\end{enumerate}
\end{lem}

\begin{proof}
This is \cite{CFFS25}*{Lemma 4.5}; its proof also gives the stated phase-indicator convergence.
\end{proof}

In particular, if $u$ is global and $u_{r_k}\to u_\infty$ along a blow-down sequence, the strong $H^1_{\rm loc}$, phase-indicator, and locally uniform convergence give
\[
{\bf W}(u_\infty,\rho) =\lim_{k\to\infty}{\bf W}(u,r_k\rho) ={\bf W}(u,\infty) \quad\text{for}\quad \rho>0.
\]
Thus the equality case of \eqref{eq:W-monotone} makes $u_\infty$ $1$-homogeneous.

\begin{lem}
\label{lem:half-space-regularity-classical-stable-limits}
For each $n\ge2$, there is a dimensional constant $\eps_{\rm lim}>0$ with the following property. Let $u_j$ be classical stable solutions in $B_2\subset\R^n$, with $0\in\FB(u_j)$, and suppose that
\[
u_j\to u \quad\text{locally uniformly and strongly in }H^1_{\rm loc}(B_2).
\]
If, for some $e\in\mathbb S^{n-1}$,
\[
\fint_{B_2}\bigl|u-(e\cdot x)_+\bigr|\,dx\le\eps_{\rm lim},
\]
then $u$ is a classical stable solution in $B_{1/4}$.
\end{lem}

\begin{proof}
It is a direct consequence of \cite{CFFS25}*{Lemma C.1} and \Cref{lem:epsilon-regularity}.
\end{proof}

For $m\ge2$, let $\mathscr C_m$ be the class of nonnegative global $1$-homogeneous functions $v$ in $\mathbb R^m$, with $0\in\FB(v)$, for which there are radii $R_j\to\infty$ and classical stable solutions $v_j$ in $B_{R_j}$, with $0\in\FB(v_j)$, such that
\[
v_j\to v \quad\text{locally uniformly and strongly in }H^1_{\rm loc}(\mathbb R^m).
\]

\begin{prop}[Classification of homogeneous limits]
\label{prop:classification-homogeneous-classical-stable-limits}
Let $2\le m<\nAC$ and $v\in\mathscr C_m$. Then, for some $e\in\mathbb S^{m-1}$,
\[
\text{either}\qquad v=(e\cdot x)_+, \qquad\text{or}\qquad v=|e\cdot x|.
\]
\end{prop}

\begin{proof}
Fix $2\le m<\nAC$. By \Cref{lem:compactness} and a diagonal argument, $\mathscr C_m$ is sequentially compact and closed under free-boundary blow-ups. Its elements and their blow-ups are stable, and the latter are $1$-homogeneous by the equality case of \eqref{eq:W-monotone}.

Moreover, \Cref{lem:half-space-regularity-classical-stable-limits} and \Cref{defn:critical-dimension} give a positive $L^1(B_2)$-gap between the non-half-space elements of $\mathscr C_m$ (for instance, the vees) and the unit half-spaces. Since ${\bf W}(\cdot,1)$ is continuous under the convergence in \Cref{lem:compactness}, the minimum
\[
\Lambda_m:=\min\left\{{\bf W}(v,1):v\in\mathscr C_m, \ v\text{ is not a half-space}\right\},
\]
is attained. We claim that every minimizer in this definition is a unit vee.

Fix such a minimizer $v$. Its spine is the linear subspace
\[
S(v):=\{q\in\mathbb R^m:v(x+q)=v(x)\text{ for every }x\in\mathbb R^m\}.
\]
Write $\ell=m-\dim S(v)$ for the quotient dimension. Suppose first that $\ell\ge2$. Write $\mathbb R^m=S(v)^\perp\times S(v)$ and $v(z,s)=\bar v(z)$. Stationarity and stability pass to $\bar v$. Thus $\bar v$ is a global $1$-homogeneous stable solution in $\mathbb R^\ell$. If $\bar v$ is classical near every nonzero free-boundary point, then it is classical away from the origin. Since $\bar v$ has trivial spine, it is not a half-space when $\ell\ge2$, contradicting $\ell<\nAC$.

There is therefore a nonzero free-boundary point $\bar y$ at which $\bar v$ is not classical. Set $y=(\bar y,0)\in S(v)^\perp$, and let $w$ be a blow-up of $v$ at $y$. It cannot be a half-space: otherwise, at a sufficiently small fixed scale, the rescaling of $v$ about $y$ would satisfy \Cref{lem:half-space-regularity-classical-stable-limits}, making $v$ classical near $y$. Since $v$ is $1$-homogeneous,
\[
r^{-1}v(y+r\,\cdot)=v(\,\cdot+y/r)\to v \quad\text{locally in }H^1
\]
as $r\to\infty$. Hence, the translated form of \eqref{eq:W-monotone} and minimality of $v$ in the definition of $\Lambda_m$ give (denoting, for $v\in\mathscr C_m$ and $y\in\FB(v)$, ${\bf W}_y(v,r):={\bf W}(v(y+\cdot),r)$)
\[
\Lambda_m \le {\bf W}(w,1) =\lim_{r\downarrow0}{\bf W}_y(v,r) \le\lim_{r\to\infty}{\bf W}_y(v,r) ={\bf W}(v,1) =\Lambda_m.
\]
That is, ${\bf W}_y(v,r)$ is constant. The equality case of \eqref{eq:W-monotone}, translated to the center $y$, makes $v$ homogeneous about $y$ as well as the origin. Subtracting the two Euler identities gives $y\cdot\nabla v=0$. Integrating along lines parallel to $y$, using that $v$ is Lipschitz, gives $v(x+ty)=v(x)$ for every $x,t$. Thus $y\in S(v)$, a contradiction. This excludes $\ell\ge2$, so $\ell=1$.

After a rotation, $v$ is a one-dimensional function of $x_m$. $1$-homogeneity and stationarity imply that it is either a unit half-space or has the form $a|x_m|$ for some $a>0$. The first alternative is excluded by the definition of $\Lambda_m$. In the second one, \Cref{lem:compactness,lem:nondegeneracy-of-stable-solutions,lem:lipschitz-bound} give $0<a\le1$. The slicing argument in the proofs of \cite[Lemmas~4.4 and C.1]{CFFS25} now rules out $a<1$ (such an argument works in any dimension). Therefore $v=|x_m|$, proving the claim, and
\[
\Lambda_m=2\alpha_m.
\]

Finally, let $u\in\mathscr C_m$. If $u$ is a half-space, stationarity fixes its slope to be one. Otherwise, minimality and \Cref{lem:homogeneous-weiss-upper-bound} give
\[
2\alpha_m=\Lambda_m \le {\bf W}(u,1) =|\{u>0\}\cap B_1| \le |B_1|=2\alpha_m.
\]
Hence $u$ is a minimum in the definition of $\Lambda_m$, and the claim gives $u=|e\cdot x|$ for some $e\in\mathbb S^{m-1}$.
\end{proof}

\begin{remark}[Weiss-energy range]
\label{rem:weiss-energy-range}
Let $2\le n<\nAC$, and let $u$ be a global classical stable solution with $0\in\FB(u)$. Then
\begin{equation}\label{eq:weiss-energy-range}
\alpha_n\le {\bf W}(u,1)\le2\alpha_n.
\end{equation}
Indeed, every blow-up of $u$ at the classical free-boundary point $0$ is a unit half-space, whereas every blow-down belongs to $\mathscr C_n$ and is therefore a unit half-space or a unit vee by \Cref{prop:classification-homogeneous-classical-stable-limits}. The inequalities follow from \eqref{eq:W-monotone} and \eqref{eq:alpha_n}.
\end{remark}

\subsection{Almost homogeneity and blow-downs}

We now turn \Cref{prop:classification-homogeneous-classical-stable-limits} into quantitative almost-homogeneity and blow-down statements for global stable solutions.

\begin{lem}[Almost homogeneous solutions]
\label{lem:almost-homogeneous-solutions}
Let $2\le n < \nAC$. For any $\eps\in(0,\frac{\alpha_n}{2})$, there exists $\delta\in(0,\frac{\alpha_n}{2})$ such that the following hold.

Let $u$ be a classical stable solution to the Bernoulli problem in $\R^n$ such that $0\in \FB(u)$ and
\begin{equation}\label{eq:H1-compact-W}
{\bf W}(u,2)-{\bf W}(u,1)<\delta.
\end{equation}
Then, either
\begin{equation}\label{eq:H1-compact-half}
\normbig[L^\infty(B_1\cap\{u>0\})]{u-e\cdot x}<\eps \quad \text{ for some $e\in \bS^{n-1}$ and } \quad {\bf W}(u,2)<\alpha_n+\eps,
\end{equation}
or
\begin{equation}\label{eq:H1-compact-abs}
\normbig[L^\infty(B_1)]{u-|e\cdot x|}<\eps \quad \text{ for some $e\in \bS^{n-1}$ and } \quad {\bf W}(u,1)>2\alpha_n-\eps.
\end{equation}
\end{lem}

\begin{proof}
Suppose the conclusion fails. Then for some $\eps_0>0$ there are global classical stable solutions $u_j$, with $0\in\FB(u_j)$, such that
\[
{\bf W}(u_j,2)-{\bf W}(u_j,1)\to0,
\]
but neither \eqref{eq:H1-compact-half} nor \eqref{eq:H1-compact-abs} holds with $\eps=\eps_0$. By \Cref{lem:compactness}, after passing to a subsequence,
\[
u_j\to u_\infty \quad\text{locally uniformly and strongly in }H^1_{\rm loc}(\mathbb R^n), \qquad 0\in\FB(u_\infty),
\]
and, by the locally uniform, strong $H^1$, and phase-indicator convergence,
\[
{\bf W}(u_\infty,2)-{\bf W}(u_\infty,1)=0.
\]
The equality case of \eqref{eq:W-monotone} gives $x\cdot\nabla u_\infty=u_\infty$ in $B_2\setminus\overline{B_1}$. The function $x\cdot\nabla u_\infty-u_\infty$ is harmonic in each positive component, and every such component meeting $B_1$ reaches this annulus by the maximum principle. Unique continuation therefore makes $u_\infty$ $1$-homogeneous throughout $B_2$.

Let $\bar u$ be its global $1$-homogeneous extension. Choose $r_k\downarrow0$ and a diagonal subsequence of the $u_j$ whose uniform and squared $H^1$ errors on $B_2$ are $o(r_k)$ and $o(r_k^n)$, and whose phase-indicator error is $o(r_k^n)$. Then $r_k^{-1}u_j(r_k\,\cdot)$ converges to $\bar u$ on expanding balls, so $\bar u\in\mathscr C_n$. Therefore \Cref{prop:classification-homogeneous-classical-stable-limits} applies and gives, for some $e\in\mathbb S^{n-1}$,
\[
\bar u=(e\cdot x)_+ \qquad\text{or}\qquad \bar u=|e\cdot x|.
\]

In the half-space case, Hausdorff convergence of the positive phases gives
\[
\sup_{B_1\cap\{u_j>0\}}(-e\cdot x)_+\to0.
\]
Together with the convergences supplied by \Cref{lem:compactness}, this gives \eqref{eq:H1-compact-half} for large $j$. In the vee case, the same convergences give $\|u_j-|e\cdot x|\|_{L^\infty(B_1)}\to0$ and ${\bf W}(u_j,1)\to2\alpha_n$, hence \eqref{eq:H1-compact-abs}. Both alternatives contradict the choice of the sequence.
\end{proof}

As a consequence of \Cref{lem:almost-homogeneous-solutions}, if we can lower bound the Hessian of a solution at one point, then the solution cannot be energetically close to a half-space.

\begin{lem}
\label{lem:lower-bound-for-the-weiss-energy}
Let $2\le n< \nAC$. Let $\eps_\circ=\eps_\circ(n)$ and $C_{n,2}$ be the constants from \Cref{lem:epsilon-regularity}, after decreasing $\eps_\circ$ so that $\eps_\circ<\alpha_n/2$, and set $C_\circ := \max\{1,C_{n,2}\eps_\circ\}$. Let $\delta_\circ\in (0,\frac{\alpha_n}{2})$ be chosen from \Cref{lem:almost-homogeneous-solutions} with $\eps=\eps_\circ$.

Let $u$ be a global classical stable solution to the Bernoulli problem in $\R^n$, with $0\in \FB(u)$. If
\begin{equation}\label{eq:lem-W-lower-Hess}
\norm[L^\infty(B_{1/2} \cap \set{u>0})]{D^2u} \geq 2C_\circ,
\end{equation}
then
\begin{equation}\label{eq:lem-W-lower}
{\bf W}(u,2) \geq \alpha_n+\delta_\circ.
\end{equation}
\end{lem}

\begin{proof}
Recalling \eqref{eq:weiss-energy-range}, we always have ${\bf W}(u,1)\ge \alpha_n$. If \eqref{eq:lem-W-lower} failed, then
\[
{\bf W}(u,2)-{\bf W}(u,1) <(\alpha_n+\delta_\circ)-\alpha_n =\delta_\circ.
\]
Applying \Cref{lem:almost-homogeneous-solutions} with $\eps=\eps_\circ$, we obtain either \eqref{eq:H1-compact-half} or \eqref{eq:H1-compact-abs}.

The half-space branch \eqref{eq:H1-compact-half} is impossible. Indeed, by \Cref{lem:epsilon-regularity} the $L^\infty$-closeness to a half-space on $B_1\cap\{u>0\}$ implies
\[
\norm[L^\infty(B_{1/2}\cap\{u>0\})]{D^2u}\le C_\circ,
\]
contradicting the assumed lower bound \eqref{eq:lem-W-lower-Hess}.

Hence only the symmetric branch \eqref{eq:H1-compact-abs} remains, and in that case
\[
{\bf W}(u,2)\ge {\bf W}(u,1)>2\alpha_n-\eps_\circ.
\]
Since $\eps_\circ<\alpha_n/2$ and $\delta_\circ<\alpha_n/2$, we have
\[
2\alpha_n-\eps_\circ>\alpha_n+\delta_\circ,
\]
which contradicts the negation of \eqref{eq:lem-W-lower}. Therefore \eqref{eq:lem-W-lower} holds.
\end{proof}

The almost-homogeneity and energy-gap conclusions in \Cref{lem:almost-homogeneous-solutions,lem:lower-bound-for-the-weiss-energy} lead to the following quantitative vee asymptotics.
\begin{prop}[Blow-down of non-flat solutions]
\label{prop:blow-down-of-non-flat-solutions}
Let $2\le n<\nAC$. There is a dimensional modulus of continuity $\omega$ of the form \eqref{eq:w} with the following property. Let $u$ be a global classical stable solution to the Bernoulli problem in $\mathbb R^n$, with $0\in\FB(u)$, and suppose that
\begin{equation}\label{eq:lem-W-Hess}
\norm[L^\infty(B_1 \cap \set{u>0})]{D^2 u} \geq 1.
\end{equation}
Then, for every $R>0$, one can choose $e_R\in\mathbb S^{n-1}$ with
\begin{equation}
\label{eq:omega_mod}
\begin{aligned}
\frac1R\big\|u-V_{0,e_R}\big\|_{L^\infty(B_R)} +\left(\frac1{R^n}\int_{B_R} \big|\nabla u-\nabla V_{0,e_R}\big|^2\,dx\right)^{1/2} +\frac1{R^n}\int_{B_R} \big|\mathbbm 1_{\{u>0\}}-\mathbbm 1_{\{V_{0,e_R}>0\}}\big|\,dx \le \omega(R^{-1}).
\end{aligned}
\end{equation}
\end{prop}

\begin{proof}
We first obtain the $L^\infty$ estimate by the dyadic Weiss iteration in the proof of \cite[Proposition~4.1]{CFFS25} using \Cref{lem:almost-homogeneous-solutions,lem:lower-bound-for-the-weiss-energy}. Indeed, choose the dichotomy parameter below the fixed gap in \Cref{lem:lower-bound-for-the-weiss-energy}. Every rescaling $u_s$ used in the iteration has $s\ge2C_\circ$, and \eqref{eq:lem-W-Hess} gives
\[
\|D^2u_s\|_{L^\infty(B_{1/2}\cap\{u_s>0\})}\ge2C_\circ.
\]
Thus \eqref{eq:lem-W-lower} excludes the half-space alternative, while \eqref{eq:weiss-energy-range} prevents indefinitely many bad dyadic increments. As in the cited proof in \cite{CFFS25}, we obtain a dimensional modulus $\omega_0$ such that, for every $R>0$, some $e\in\mathbb S^{n-1}$ satisfies
\[
\|u-V_{0,e}\|_{L^\infty(B_{2R})} \le 2R\,\omega_0((2R)^{-1}).
\]
Set $\tau:=2\omega_0((2R)^{-1})$ and $V:=V_{0,e}$. Since $u\ge0$,
\[
\{u=0\}\cap B_R\subset B_R\cap\{V\le\tau R\},
\]
and therefore the normalized phase-indicator error in $B_R$ is at most $C\tau$.

For the gradient term, take $\eta\in C_c^\infty(B_{2R})$, equal to $1$ on $B_R$, with $|\nabla\eta|\le C/R$, and set $w:=u-V$. In the sense of distributions,
\[
\Delta u=\mathcal H^{n-1}\!\lfloor_{\FB(u)}, \qquad \Delta V=2\mathcal H^{n-1}\!\lfloor_{e^\perp}.
\]
Thus \Cref{lem:perimeter-bound} gives $|\Delta w|(B_{2R})\le CR^{n-1}$. Testing with $\eta^2w$, by smooth approximation, and absorbing the cutoff term yields
\[
\int_{B_R}|\nabla u-\nabla V|^2\,dx \le C(\tau+\tau^2)R^n.
\]
Consequently \eqref{eq:omega_mod} holds with
\[
\omega(t):=C\bigl(\sqrt{\omega_0(t/2)}+\omega_0(t/2)\bigr),
\]
after a universal enlargement. This is again a modulus of the form \eqref{eq:w}.
\end{proof}

\section{Necks: definition and properties}
\label{sec:neck_set}

In this section, we recall the construction from \cite{CFFS25} of the neck centers, used in subsequent sections, and their initial properties.

\begin{remark}[Dimension independence of neck set construction]
\label{rmk:dim3-dimn}
All the results in this section come from \cite{CFFS25}, and we restate them for the reader's convenience. Although they are stated there for $n=3$, their proofs generalize verbatim to arbitrary $n$.
\end{remark}
\subsection{Reduction}

We first reduce \Cref{thm:main-classification} to a normalized non-flat global solution with controlled Hessian.

If $u$ is a classical solution and $x_\circ \in \FB(u)$, we write $D^2u(x_\circ)$ to denote the limit of $D^2u$ from the positivity set:
\[
D^2u(x_\circ):=\lim_{\{u>0\}\ni x\to x_\circ} D^2 u(x).
\]

\begin{lem}[Reduction, {\cite{CFFS25}*{Lemma~5.1}}]
\label{lem:reduction}
Let $n \ge 2$, and suppose there exists a global classical stable solution $v$ to the Bernoulli problem in $\R^n$ such that $|D^2 v|\not\equiv 0$ in $\{v > 0\}$. Then, there exists a global classical stable solution $u$ such that $0\in \FB(u)$, $|D^2 u|\le 1$ in $\overline{\{u>0\}}$, and $|D^2 u|(0) = 1$.
\end{lem}

\subsection{Standing global assumptions}
\label{ssec:fixinghyp}

We now fix the normalized global solution and the Hessian-energy threshold used throughout the contradiction argument.

\begin{assumption}[Standing global assumptions]
\label{ass:standing-global-assumptions}
We argue by contradiction and assume that a non-flat global classical stable solution exists in some dimension $n<\nAC$. From this point on, unless stated otherwise, we fix the normalized global solution $u\in {\rm Lip}(\R^n)$ provided by \Cref{lem:reduction} (see also \Cref{lem:lipschitz-bound}), so that
\begin{equation}
\label{eq:assump_glob_u} \text{$0\in \FB(u)$, \quad $|\nabla u|\le 1$, \quad $|D^2 u(0)| = 1$, \qquad and \qquad $|D^2 u|\le 1\quad\text{in}\quad \{u > 0\}$.}
\end{equation}
We reserve the letter $u$ for this standing global solution unless stated otherwise.

We also fix, for $\eta_*(n)$ the constant from \Cref{lem:epsilon-regularity-for-the-hessian} in dimension $n$,
\begin{equation}
\label{eq:eta0} \eta_0 : = \min\big\{\eta_*(n),|B_1|^{1/n}\big\}.
\end{equation}
\end{assumption}

\subsection{Definition of neck centers}
\label{ssec:neck-set-def}

We construct the neck centers by selecting a separated family of free-boundary points at their threshold Hessian-energy scales.

\begin{defn}[Neck centers]
\label{defn:neck-centers}
Given $u$ and $\eta_0$ as above, we define the set of \emph{neck centers} $\cZ$ as follows.
\begin{itemize}
\item First, for any $y\in \FB(u)$, we define its \emph{threshold radius} as
\begin{equation}
\label{eq:def_rB} \rb(y) := \inf\bigg\{r > 0 : \int_{B_r(y)\cap \{u > 0\}}|D^2 u|^n \,dx \ge \eta_0^n \bigg\}.
\end{equation}
Observe that, since we are assuming $|D^2 u|$ to be globally and universally bounded, we know that
\begin{equation}\label{eq:rmin_def}
\rb(x) \geq \rmin = \rmin ( \eta_0):= c \, \eta_0 > 0\qquad\text{for all}\quad x\in \FB(u).
\end{equation}

\item Then, for any $k\in \N_0$, we define
\begin{equation}\label{eq:Xktilde}
\tilde \cZ_k := \left\{ x \in \FB(u) : \rb(x) \in [\rmin 2^k, \rmin 2^{k+1})\right\}.
\end{equation}

\item Given $\lambda > 0$ and $\mathcal Y \subset \FB(u)$, we denote
\begin{equation}
\label{eq:cBlambdadef} \cB_\lambda(\mathcal Y) := \bigcup_{y\in \mathcal Y} B_{\lambda \rb(y)}(y)
\end{equation}
(not to be confused with the notation $B_r(A) = A+B_r$ in \Cref{ssec:notation}). Thanks to Vitali's covering lemma, we can consider a countable subset of centers $\cZ_0\subset \tilde \cZ_0$ such that
\[
B_{\rb(z_1)}(z_1)\cap B_{\rb(z_2)}(z_2) = \varnothing\qquad\text{for all}\quad z_1, z_2\in \cZ_0,\quad z_1\neq z_2,
\]
and
\[
\tilde \cZ_0\subset \cB_1(\tilde \cZ_0)\subset \cB_4(\cZ_0).
\]

\item Then, for $k\ge 1$, we recursively define
\begin{equation}
\label{eq:Xkprime} \cZ_k' := \{x \in \tilde \cZ_k : B_{4\rb(x)}(x) \cap \cZ_{<k} = \varnothing\},
\end{equation}
where we have denoted $\cZ_{<k} := \bigcup_{i = 0}^{k-1} \cZ_i$. We take $\cZ_k$ to be the centers of a Vitali subcovering of $\cB_1(\cZ_k')$, namely, $\cZ_k\subset \cZ_k'$ is a countable subset such that
\[
B_{\rb(z_1)}(z_1)\cap B_{\rb(z_2)}(z_2) = \varnothing\qquad\text{for all}\quad z_1, z_2\in \cZ_k,\quad z_1\neq z_2,
\]
and
\[
\cZ_k'\subset \cB_1(\cZ_k')\subset \cB_4(\cZ_k).
\]

\item Finally, we define
\[
\cZ := \bigcup_{k\ge 0} \cZ_k.
\]
We call the points in $\cZ$ {\em neck centers} and denote the points in $\cZ$ by $\zz$, $\zz_k$, etc. The threshold radii of neck centers are simply called \emph{neck radii}.
\end{itemize}
\end{defn}

The standing normalization, the choice of $\eta_0$, and \Cref{lem:epsilon-regularity-for-the-hessian} give $\rb(0)<\infty$, as in \cite[Lemma~5.2]{CFFS25}. Thus the first nonempty dyadic class in the construction of \Cref{defn:neck-centers} produces a neck center, and $\cZ\ne\varnothing$. From now on, this set is fixed. For $R>0$, we write
\begin{equation}\label{eq:ZR}
\cZ_R := \{\zz\in \cZ : \rb(\zz) \le R\}.
\end{equation}

\subsection{Basic properties of the neck centers and neck radii}

We record the covering, Hessian-decay, blow-down, and scale-comparison properties of the neck centers needed below, coming from \cite{CFFS25}, stated here in dimension $n <\nAC$.

\begin{lem}[Global Hessian decay, {\cite{CFFS25}*{Lemma~5.6}, \Cref{rmk:dim3-dimn}}]
\label{lem:global-hessian-decay}
We have:
\[
|D^2 u(x)| \leq C \min\biggl\{\frac{1}{\dist(x,\cZ)},1\biggr\},\qquad\text{for all}\quad x\in \{u > 0\},
\]
for some $C$ universal.
\end{lem}

\begin{lem}[Blow-down around a neck center, {\cite{CFFS25}*{Lemma~5.7}, \Cref{rmk:dim3-dimn}}]
\label{lem:blow-down-around-a-neck-center}
For any $\eps>0$ there exists $M=M(\eps)\ge 1$ such that the following holds.

For every $\zz \in \cZ$ and every $R\geq M \rb(\zz)$, we have
\begin{equation}\label{epscloseness11}
\min_{e\in \mathbb S^{n-1}}\big\| u - V_{\zz, e} \big\|_{L^\infty(B_{R}(\zz))} \leq \eps R.
\end{equation}

More precisely, choosing $M_\star:=|B_1|^{\frac{1}{n}}\eta_0^{-1}\ge 1$ and $\omega$ as in \eqref{eq:omega_mod}, it is enough to choose $M$ so that
\[
\omega\big(M_\star/M\big)\le \eps.
\]
\end{lem}

\begin{lem}[Local comparability of neck radii, {\cite{CFFS25}*{Corollary~5.4 and Lemma~5.9}, \Cref{rmk:dim3-dimn}}]
\label{lem:local-comparability-of-neck-radii}
For any $M\ge 1$ there exists $C_M = C_M(M)$ such that the following holds. Given $\zz\in \cZ$ and $\varrho = M \rb(\zz)$, for all $\zz' \in \cZ \cap B_{\varrho} (\zz)$ we have
\[
\rb(\zz')\ge \frac{\rb(\zz)}{C_M}\qquad\text{and}\qquad \|D^2 u \|_{L^\infty(\{u>0\}\cap B_\varrho(\zz))} \le \frac{C_M}{\rb(\zz)}.
\]
\end{lem}

\begin{lem}[Upper bound for nearby neck radii, {\cite{CFFS25}*{Lemma~5.10}, \Cref{rmk:dim3-dimn}}]
\label{lem:upper-bound-for-nearby-neck-radii}
There exists $M_\circ>0$ universal such that if $\zz\in \cZ$ and $R\ge M_\circ\rb(\zz)$, then $\rb(\zz')\le \frac{R}{8}$ for all $\zz'\in \cZ\cap B_{3R/4}(\zz)$.
\end{lem}

We finish the preliminary toolkit with a one-sided estimate that converts averaged vee closeness into pointwise trapping.
\begin{lem}[{\cite{CFFS25}*{Lemma~6.5}, \Cref{rmk:dim3-dimn}}]
\label{lem:one-sided-lower-bound}
Let $y\in \R^n$, $e\in \mathbb{S}^{n-1}$, and $R> 0$. Then, the following implication holds for all $\eps\in (0,1)$:
\[
\frac{1}{R}\fint_{\{x\in B_R(y)\ :\ e\cdot(x-y)>R/8\}} (u(x)- e\cdot(x-y))_-\,dx\le \eps\qquad\implies\qquad u(x) + C\eps R \ge e\cdot(x-y) \quad \forall \,x\in B_{3R/4}(y),
\]
where $(t)_-= \max(-t,0)$ denotes the negative part of $t\in \R$ and $C$ depends only on $n$.

In particular:
\[
\frac{1}{R}\fint_{B_R(y)} |u- V_{y, e}|\,dx \le \eps\qquad\implies\qquad u+C\eps R \ge V_{y, e}\qquad\text{in}\quad B_{3R/4}(y).
\]
\end{lem}

\section{Stability inequality}
\label{sec:stability-inequality}

\subsection{Jacobi fields and stability inequalities}

We formulate the stability inequality in terms of Jacobi fields and derive the Sternberg--Zumbrun estimates used later; \cite{SZ98}.

\begin{defn}[Stability inequality]
\label{defn:stability-inequality}
For a classical stable solution $v$ to the Bernoulli problem in $B_1$, the stability inequality reads
\begin{equation}
\label{eq:stab_ineq_H} \int_{\partial\{v > 0\}} H \xi^2\, d\mathcal{H}^{n-1} \le \int_{\{v > 0\}} |\nabla\xi|^2\, dx,\qquad\text{for all}\quad \xi\in C^\infty_c(B_1), \ \xi\ge 0,
\end{equation}
(see \cite[Lemma 1]{CJK04}), where $H$ denotes the mean curvature of the free boundary, $H=-\partial^2_{\nu\nu}v$, and $\nu$ is its inward unit normal.

Putting $\xi=\bc \eta$ in \eqref{eq:stab_ineq_H} and integrating by parts,
\begin{equation}
\label{eq:stab_ineq_H_2} \int_{\set{v>0}} \bc\Delta\bc\, \eta^2 \,dx +\int_{\partial\{v > 0\}} \bc(\bc_\nu+H\bc) \eta^2 \, d\mathcal{H}^{n-1} \le \int_{\{v > 0\}} \bc^2 |\nabla\eta|^2 \, dx, \qquad\text{for all}\quad \bc,\eta\in C^\infty_c(B_1), \ \bc,\eta\ge 0.
\end{equation}
\end{defn}

To exploit this quadratic form, we next identify the Jacobi fields and the finite-family inequality they satisfy.
\begin{defn}[Jacobi fields for the Bernoulli problem]
\label{defn:bernoulli-jacobi-field}
Let $v$ be a classical solution to the Bernoulli problem in an open set $U\subset\R^n$. A function $\phi\in C^2(\{v > 0\}\cap U)\cap C^1(\overline{\{v > 0 \}}\cap U)$ is a \emph{Jacobi field} for $v$ in $U$ if
\[
\Delta\phi=0\quad\text{in }\{v > 0\}\cap U, \qquad \phi_\nu+H\phi=0\quad\text{on }\FB(v) \cap U,
\]
where $\nu=\nabla v$ is the inward unit normal and $H$ has the sign convention fixed in \eqref{eq:stab_ineq_H}.
\end{defn}

\begin{lem}[Sternberg--Zumbrun inequality for finite Jacobi families]
\label{lem:finite-jacobi-sternberg-zumbrun}
Let $n\ge2$, let $v$ be a classical stable solution to the Bernoulli problem in $B_1\subset\R^n$. Let $N\ge1$, and let $\phi_1,\ldots,\phi_N$ be $N$ Jacobi fields for $v$ in $B_1$, in the sense of \Cref{defn:bernoulli-jacobi-field}. Define

\[
\Phi:=(\phi_1,\ldots,\phi_N), \qquad \bc_\Phi^2:=|\Phi|^2 =\sum_{\alpha=1}^N\phi_\alpha^2, \qquad \cK_\Phi^2:= \sum_{\alpha=1}^N|\nabla\phi_\alpha|^2 -|\nabla\bc_\Phi|^2 =\bc_\Phi^2\left|\nabla\left(\frac{\Phi}{\bc_\Phi}\right)\right|^2,
\]
and set $\cK_\Phi^2 = 0$ when $\Phi = 0$. Then, for every $\eta\in C_c^\infty(B_1)$,
\begin{equation}\label{eq:finite-jacobi-sternberg-zumbrun}
\int_{\{v > 0\}\cap B_1}\cK_\Phi^2\eta^2\,dx \le \int_{\{v > 0\}\cap B_1}\bc_\Phi^2|\nabla\eta|^2\,dx .
\end{equation}
\end{lem}

\begin{proof}
Set
\[
\bc_\eps:=\sqrt{\bc_\Phi^2+\eps^2}.
\]
Since all $\phi_\alpha$ are harmonic in $\{v > 0\}\cap B_1$,
\begin{equation}\label{h1uo2g4oug1}
\bc_\eps\Delta\bc_\eps =\frac12\Delta(\bc_\eps^2)-|\nabla\bc_\eps|^2 = \sum_{\alpha=1}^N|\nabla\phi_\alpha|^2-|\nabla\bc_\eps|^2 \qquad\text{in }\{v > 0\}\cap B_1 .
\end{equation}
On the free boundary, the Jacobi boundary conditions give
\[
2\bc_\eps(\bc_\eps)_\nu =\partial_\nu\left(\bc_\Phi^2+\eps^2\right) =2\sum_{\alpha=1}^N\phi_\alpha(\phi_\alpha)_\nu =-2H\bc_\Phi^2.
\]
Therefore
\[
\bc_\eps\bigl((\bc_\eps)_\nu+H\bc_\eps\bigr) =H\eps^2 \qquad\text{on }\FB(v).
\]

Multiplying \eqref{h1uo2g4oug1} by a cutoff $\eta^2$ and recalling \eqref{eq:stab_ineq_H_2} we obtain
\[
\int_{\{v > 0\}\cap B_1} \left( \sum_{\alpha=1}^N|\nabla\phi_\alpha|^2-|\nabla\bc_\eps|^2 \right)\eta^2\,dx \le \int_{\{v > 0\}\cap B_1}\bc_\eps^2|\nabla\eta|^2\,dx -\eps^2\int_{\FB(v)}H\eta^2\,d\cH^{n-1}.
\]
Since $v$ is classical and $\eta$ is compactly supported, the last term vanishes as $\eps\downarrow0$.

By Cauchy-Schwarz in $\R^N$, $|\nabla\bc_\eps|^2\le\sum_{\alpha=1}^N|\nabla\phi_\alpha|^2,$ and $|\nabla\bc_\eps|^2\to|\nabla\bc_\Phi|^2$ a.e. on $\{\bc_\Phi>0\}$. On the zero set $\{\bc_\Phi=0\}$, one has $\nabla\bc_\eps=0$, and $\nabla\Phi=0$ at a.e. point. Hence the integrand converges a.e. to $\sum_{\alpha=1}^N|\nabla\phi_\alpha|^2-|\nabla\bc_\Phi|^2$. Dominated convergence gives
\[
\int_{\{v > 0\}\cap B_1} \left( \sum_{\alpha=1}^N|\nabla\phi_\alpha|^2-|\nabla\bc_\Phi|^2 \right)\eta^2\,dx \le \int_{\{v > 0\}\cap B_1}\bc_\Phi^2|\nabla\eta|^2\,dx .
\]
On $\{\bc_\Phi>0\}$, writing $\Phi=\bc_\Phi (\Phi/\bc_\Phi)$ and using $|\Phi/\bc_\Phi|=1$, we have
\[
\sum_{\alpha=1}^N|\nabla\phi_\alpha|^2 =|\nabla\Phi|^2 =|\nabla\bc_\Phi|^2+\bc_\Phi^2|\nabla (\Phi/\bc_\Phi)|^2 .
\]
This proves \eqref{eq:finite-jacobi-sternberg-zumbrun}.
\end{proof}

We will use two symmetry-generated specializations of \Cref{lem:finite-jacobi-sternberg-zumbrun}, corresponding to translations and dilations.
\begin{lem}[Tangential Sternberg-Zumbrun inequality]
\label{lem:tangential-sternberg-zumbrun-inequality}
Let $n \ge 2$, and let $v$ be a classical stable solution to the Bernoulli problem in $B_1\subset \R^n$. Let $e\in \mathbb{S}^{n-1}$, and define
\[
\nabla^e v := \nabla v - (\nabla v\cdot e)\,e.
\]
Then
\[
\int_{B_1\cap \{v > 0\} }\cA_e^2(v)\,\eta^2\,dx \le \int_{B_1}|\nabla^e v|^2|\nabla \eta|^2 \,dx,\qquad\text{for any}\quad \eta\in C^\infty_c(B_1),
\]
where
\[
\cA_e^2(v) := \sum_{i = 1}^n |\partial_i \nabla^ev|^2 -\big|\nabla|\nabla^e v|\big|^2 = |D\nabla^ev|^2 -\big|\nabla|\nabla^e v|\big|^2,
\]
if $|\nabla^e v|(x) \neq 0$, and $\cA_e^2(v) = 0$ otherwise.
\end{lem}

\begin{proof}
Let
\[
\Phi:=\nabla^e v =\nabla v-(\nabla v\cdot e)e.
\]
Translation invariance shows that each $\partial_jv$ is a Jacobi field for $v$. Hence every component of $\Phi$ is a Jacobi field, and the result follows from \Cref{lem:finite-jacobi-sternberg-zumbrun}.
\end{proof}

\begin{lem}[Dilation Sternberg--Zumbrun inequality]
\label{lem:dilation-sternberg-zumbrun-inequality}
Let $n\ge2$, and let $v$ be a classical stable solution to the Bernoulli problem in $B_1\subset\R^n$. Fix $p\in\R^n\setminus\{0\},$ and let
\[
Z_p(x):=\frac{x-p}{|p|}, \qquad A_p:=Z_p\cdot\nabla v-\frac{v}{|p|},\qquad J_p:=\bigl(A_p,\partial_2v,\ldots,\partial_{n-1}v\bigr)\in\R^{n-1}.
\]
Define
\[
\bc_p^2:= \left|J_p\right|^2 =A_p^2+\sum_{i=2}^{n-1}(\partial_iv)^2 \quad\text{in }B_1\cap\{v>0\},\qquad \cK_p^2(v):= \bc_p^2|\nabla (J_p/\bc_p)|^2 =|\nabla A_p|^2+\sum_{i=2}^{n-1}|\nabla\partial_i v|^2 -|\nabla\bc_p|^2,
\]
and set $\cK_p^2(v)=0$ on $\{\bc_p=0\}$. Then $J_p$ is a finite family of Jacobi fields, so \eqref{eq:finite-jacobi-sternberg-zumbrun} applies with $\Phi=J_p$.
\end{lem}

\begin{proof}
We verify that $J_p$ is a finite family of Jacobi fields for $v$ in $B_1$. As in \Cref{lem:tangential-sternberg-zumbrun-inequality}, each $\partial_i v$ is a Jacobi field. Dilation invariance shows that $t^{-1}v\bigl(p+t(x-p)\bigr)$ is again a solution in the corresponding rescaled domain. Differentiating at $t=1$ gives
\[
\left.\frac d{dt}\right|_{t=1}t^{-1}v\bigl(p+t(x-p)\bigr) =(x-p)\cdot\nabla v-v=|p|A_p,
\]
so $A_p$ is a Jacobi field as well. Thus \Cref{lem:finite-jacobi-sternberg-zumbrun} applies to $J_p$.
\end{proof}

\subsection{Transverse excess and neck-scale consequences}

We introduce the transverse excesses and derive their neck-scale rigidity, positivity, and packing consequences.

\begin{defn}[Transverse energy and symmetric excess]
\label{defn:transverse-energy-and-symmetric-excess}
Returning to the standing global solution $u$, for $x\in \R^n$ and $R>0$ define the transverse energy as
\begin{equation}
\label{eq:tilde_varrho_def_intro} \widetilde\varrho_x(R) :=\min_{e\in \mathbb{S}^{n-1}} \left(R^{2-n}\int_{B_R(x)\cap\{u>0\}}\cA_e^2(u)\,dx\right)^{1/2}.
\end{equation}
We also define the symmetric excess as
\begin{equation}
\label{eq:Ez_def_intro} \bE_x(R) := \min_{e\in \mathbb{S}^{n-1}}\left( \frac{1}{R^n}\int_{B_R(x)} |\nabla^e u|^2\, dx\right)^{1/2}.
\end{equation}
\end{defn}

The tangential stability inequality gives the following comparison.

\begin{lem}[Transverse energy controls the symmetric excess]
\label{lem:transverse-energy-controls-the-symmetric-excess}
For any $x\in \R^n$ and $R >0$, we have
\[
\widetilde\varrho_x(R) \le C \bE_x(2R),
\]
for some $C$ universal.
\end{lem}

\begin{proof}
This follows by applying \Cref{lem:tangential-sternberg-zumbrun-inequality} after rescaling, with a localized test function $\eta$ and a direction $e\in \bS^{n-1}$ attaining the minimum in \eqref{eq:Ez_def_intro} at scale $2R$.
\end{proof}

To turn small transverse defect into rigidity, we first record the compactness available at neck scale.
\begin{lem}[Neck-scale compactness]
\label{lem:neck-scale-compactness}
Let $v_j$ be classical stable Bernoulli solutions in $B_4\subset\R^n$, with $0\in\FB(v_j)$. Assume that, for fixed constants $C_0<\infty$ and $\eta_0>0$,
\[
\|D^2v_j\|_{L^\infty(B_4\cap\{v_j>0\})} \le C_0\qquad\text{and}\qquad \int_{B_1\cap\{v_j>0\}} |D^2v_j|^n\,dx=\eta_0^n,
\]
and that $|\nabla v_j|\le 1$ uniformly. Then, up to a subsequence, $v_j\to v_\infty$ locally uniformly in $B_4$, where $v_\infty$ is a classical stable solution. Moreover,
\[
\int_{B_1\cap\{v_\infty>0\}} |D^2v_\infty|^n\,dx=\eta_0^n.
\]
If, in addition, for some $e_j\to e_\infty\in\bS^{n-1}$, $\int_{B_2\cap\{v_j>0\}}\cA_{e_j}^2(v_j)\,dx\to 0,$ then
\[
\cA_{e_\infty}^2(v_\infty)\equiv 0 \quad\text{a.e. in }B_2\cap\{v_\infty>0\}.
\]
\end{lem}

\begin{proof}
By \Cref{lem:compactness}, $v_\infty$ is a stable solution, and it is either classical stable or a vee thanks to the curvature bound on the free boundary and the 1-Lipschitz bound, arguing as in \cite[Lemma 5.1]{CFFS25}. Moreover, thanks to the higher regularity provided by the curvature bounds, we also have
\[
D^2v_j\,\mathbf 1_{\{v_j>0\}} \to D^2v_\infty\,\mathbf 1_{\{v_\infty>0\}} \qquad\text{strongly in }L^n(B_1),
\]
so that in particular $\int_{B_1\cap\{v_\infty>0\}} |D^2v_\infty|^n\,dx=\eta_0^n$, and $v_\infty$ cannot be a vee.

For the final assertion, set
\[
F_j:=\nabla^{e_j}v_j, \qquad F_\infty:=\nabla^{e_\infty}v_\infty,
\]
and extend $\cA_{e_j}^2(v_j)$ and $\cA_{e_\infty}^2(v_\infty)$ by zero outside their positive phases. If $x\in B_2\cap\{v_\infty>0\}$, local uniform convergence and interior estimates give $v_j\to v_\infty$ in $C^2$, and hence $F_j\to F_\infty$ in $C^1$, around $x$. On $\{F\ne0\}$,
\[
\cA_e^2(v) =\left|\left(I-\frac{F\otimes F}{|F|^2}\right)DF\right|^2, \qquad F=\nabla^e v,
\]
whereas $\cA_e^2(v)=0$ by definition on $\{F=0\}$. Thus, at every point of $B_2\cap\{v_\infty>0\}$ where $F_\infty\ne0$, $\cA_{e_j}^2(v_j)\to\cA_{e_\infty}^2(v_\infty).$ At points where $F_\infty=0$, and outside the positive phase of $v_\infty$, the left-hand side below is zero. Thus,
\[
0\le \cA_{e_\infty}^2(v_\infty)\mathbf 1_{\{v_\infty>0\}} \le \liminf_{j\to\infty} \cA_{e_j}^2(v_j)\mathbf 1_{\{v_j>0\}} \qquad\text{a.e. in }B_2.
\]
Fatou's lemma therefore gives
\[
\int_{B_2\cap\{v_\infty>0\}}\cA_{e_\infty}^2(v_\infty)\,dx \le \liminf_{j\to\infty} \int_{B_2\cap\{v_j>0\}}\cA_{e_j}^2(v_j)\,dx =0,
\]
which proves the final claim.
\end{proof}

The limiting zero-defect condition forces each positive component to depend on at most two variables.
\begin{lem}
\label{lem:zero-defect-two-dimensional-reduction}
Let $n\ge 2$, let $U\subset\R^n$ be open, let $v$ be a classical solution to the Bernoulli problem in $U$, and let $e\in\bS^{n-1}$. Assume that
\[
\cA_e^2(v)\equiv 0 \qquad\text{in }U\cap\{v>0\}.
\]
Then, for every connected component $\Omega$ of $U\cap\{v>0\}$, either $\nabla^e v\equiv0$ in $\Omega$, or there is $a_\Omega\in e^\perp\cap\bS^{n-1}$ such that
\[
\partial_\tau v\equiv0\quad\text{in }\Omega \qquad\text{for every }\tau\in(\operatorname{span}\{e,a_\Omega\})^\perp.
\]
In particular, on each such component $v$ depends on at most two Euclidean variables.
\end{lem}

\begin{proof}
Fix a connected component $\Omega$ of $U\cap\{v>0\}$, and set
\[
F:=\nabla^e v=\nabla v-(\nabla v\cdot e)e .
\]
The vector field $F$ is smooth and takes values in $e^\perp$. Let us assume that $F\not\equiv0$, and let $O$ be a connected component of the nonempty open set $\{F\neq0\}\cap\Omega$. In $O$, we have
\[
|DF|^2-\big|\nabla |F|\big|^2 = |F|^2\left|D\left(\frac{F}{|F|}\right)\right|^2.
\]
Since $\cA_e^2(v)\equiv0$, it follows that $D\left({F}/{|F|}\right)\equiv0$ in $O$. Thus there is a fixed vector $a\in e^\perp\cap\bS^{n-1}$ such that $F=|F|a$ in $O$. For every $\tau\in(\operatorname{span}\{e,a\})^\perp$, we have
\[
\partial_\tau v=\tau\cdot\nabla v=\tau\cdot F=0 \qquad\text{in }O .
\]
By unique continuation, $\partial_\tau v\equiv0$ in $\Omega$, for every such $\tau$.
\end{proof}

To rule out the resulting two-dimensional limits, we use the following rigidity theorem for global classical stable solutions in $\mathbb R^2$.
\begin{prop}[{\cite{FR19}*{Theorem~1.9} or \cite{kamburov2022nondegeneracy}*{Corollary~1.3}}]
\label{prop:low-dimensional-stable-solutions-flat}
Let $w$ be a global classical stable Bernoulli solution in $\mathbb R^2$. Then
\[
D^2w\equiv0\qquad\text{in }\{w>0\}.
\]
\end{prop}

As a consequence of \Cref{lem:neck-scale-compactness,lem:zero-defect-two-dimensional-reduction,prop:low-dimensional-stable-solutions-flat}, we obtain:

\begin{lem}[Positive excess at the neck scale]
\label{lem:positive-excess-at-the-neck-scale}
There exists $c > 0$ universal such that
\[
\widetilde\varrho_\zz(2\rb(\zz))\ge c > 0,\qquad\text{for all}\quad \zz\in \cZ.
\]
\end{lem}

\begin{proof}
Suppose that no universal lower bound exists. Then there are global classical stable Bernoulli solutions $u^{(j)}$ satisfying the standing normalization \eqref{eq:assump_glob_u} and neck centers $\zz_j\in\cZ(u^{(j)})$ with
\[
\widetilde\varrho^{\,u^{(j)}}_{\zz_j}\bigl(2\rb^{\,u^{(j)}}(\zz_j)\bigr)\to 0.
\]
Let $r_j:=\rb^{\,u^{(j)}}(\zz_j)$ be the corresponding threshold radius, and rescale
\[
v_j(x):=\frac{u^{(j)}(\zz_j+r_jx)}{r_j}.
\]
By \eqref{eq:def_rB},
\[
\int_{B_1\cap\{v_j>0\}}|D^2v_j|^n\,dx=\eta_0^n,
\]
and \Cref{lem:local-comparability-of-neck-radii}, with $M=4$, gives
\[
\|D^2v_j\|_{L^\infty(B_4\cap\{v_j>0\})}\le C.
\]

Choose $e_j\in\bS^{n-1}$ realizing the minimum in \eqref{eq:tilde_varrho_def_intro} for $\widetilde\varrho^{\,u^{(j)}}_{\zz_j}(2r_j)$. The scaling of $\cA_e$ gives
\[
\int_{B_2\cap\{v_j>0\}}\cA_{e_j}^2(v_j)\,dx =2^{n-2}\widetilde\varrho^{\,u^{(j)}}_{\zz_j}(2r_j)^2\to 0.
\]
Passing to a subsequence, $e_j\to e_\infty$. Moreover, \Cref{lem:local-comparability-of-neck-radii} gives a uniform Hessian bound for the rescaled functions on every fixed ball. A diagonal application of \Cref{lem:compactness} and \Cref{lem:neck-scale-compactness} therefore gives a global classical stable limit $v_\infty: \R^n \to [0,\infty)$, with $v_j\to v_\infty$ locally uniformly in $\R^n$ and
\[
\int_{B_1\cap\{v_\infty>0\}}|D^2v_\infty|^n\,dx=\eta_0^n \qquad\text{and}\qquad \cA_{e_\infty}^2(v_\infty)\equiv 0 \qquad\text{a.e. in }B_{2}\cap\{v_\infty>0\}.
\]

On the set where $|\nabla^{e_\infty}v_\infty|>0$, the defect is continuous; on its complement it vanishes by definition. Thus its a.e. vanishing in $B_2$ is pointwise. Let $\bar v_\infty$ agree with $v_\infty$ on the connected components of its positivity set that meet $B_1$, and vanish elsewhere. Using a local cutoff to isolate each of these components in the stability inequality shows that $\bar v_\infty$ is still a global stable solution. By \Cref{lem:zero-defect-two-dimensional-reduction} and unique continuation, each such component is cylindrical over a two-dimensional Bernoulli solution. The corresponding planar solution is stable. Hence \Cref{prop:low-dimensional-stable-solutions-flat} gives
\[
\int_{B_1\cap\{v_\infty>0\}}|D^2v_\infty|^n\,dx=0,
\]
contradicting the retained Hessian mass $\eta_0^n$.
\end{proof}

Since every neck carries transverse defect, it is natural to aggregate the neck radii (raised to a suitable power) into a scale-normalized neck charge.
\begin{defn}[Neck charge]
\label{defn:neck-radius-symmetric-excess}
For $x\in\R^n$ and $R>0$, define
\begin{equation}\label{eq:varrho_def_intro}
\varrho_x(R)^2 := \sum_{\zz'\in\cZ_{R/2}\cap B_{R/2}(x)} \left(\frac{\rb(\zz')}{R}\right)^{n-2}.
\end{equation}
Equivalently, $\varrho_x(R)$ is the square root of the additive neck-radius charge in $B_{R/2}(x)$ from neck centers whose threshold radii are at most $R/2$.
\end{defn}

\Cref{lem:positive-excess-at-the-neck-scale} allows this discrete charge to be estimated through the continuous symmetric excess.
\begin{prop}
\label{prop:symmetric-excess-controls-neck-radii}
There exists $C\ge 1$ universal such that, for every $x\in\R^n$ and $R>0$,
\begin{equation}\label{boundneckra1}
\varrho_x(R)\le C\widetilde\varrho_x(2R)\le C\bE_x(4R).
\end{equation}
\end{prop}

\begin{proof}
Fix $z\in\R^n$, $R>0$, and let
\[
\tilde \cZ_{R/2}(z):=\cZ_{R/2}\cap B_{R/2}(z).
\]
The balls $\{B_{2\rb(y)}(y):y\in\tilde \cZ_{R/2}(z)\}$ have bounded overlap by \Cref{defn:neck-centers}.

Choose $e_R\in\bS^{n-1}$ minimizing the integral in \eqref{eq:tilde_varrho_def_intro} for $\widetilde\varrho_z(2R)$. If $y\in\tilde \cZ_{R/2}(z)$, then $B_{2\rb(y)}(y)\subset B_{2R}(z)$, and \Cref{lem:positive-excess-at-the-neck-scale} gives
\[
c^2\rb(y)^{n-2} \le C\min_{e\in\bS^{n-1}} \int_{B_{2\rb(y)}(y)\cap\{u>0\}}\cA_e^2(u)\,dx \le C\int_{B_{2\rb(y)}(y)\cap\{u>0\}}\cA_{e_R}^2(u)\,dx.
\]
Summing over the bounded-overlap family gives
\[
\sum_{y\in\tilde \cZ_{R/2}(z)}\rb(y)^{n-2} \le C\int_{B_{2R}(z)\cap\{u>0\}}\cA_{e_R}^2(u)\,dx =C(2R)^{n-2}\widetilde\varrho_z(2R)^2,
\]
which yields the first inequality. The second follows from \Cref{lem:transverse-energy-controls-the-symmetric-excess}.
\end{proof}

\section{A Weiss-type excess}
\label{sec:weiss-type-excess}

Recall that $u$ is the standing global classical stable Bernoulli solution fixed in \Cref{ass:standing-global-assumptions}. All constants are universal unless another dependence is displayed.

\subsection{Monotonicity formulas}

We introduce the Gaussian Weiss quantities and establish the fixed- and moving-center monotonicity estimates needed below.

\begin{defn}[Gaussian Weiss quantities]
\label{defn:gaussian-weiss-quantities}
We use the Bernoulli half-energy density
\[
P(x):=\frac{|\nabla u|^2}{2}+\frac12\mathbf 1_{\{u>0\}},
\]
and the Gaussian kernel
\[
G_r(x):=\frac{1}{r^n}e^{-\frac{|x|^2}{r^2}}.
\]
For $r>0$ and $x_\circ\in\R^n$, define
\[
H_{r,x_\circ}:=\int P(x)G_r(x-x_\circ)\,dx,\qquad D_{r,x_\circ}:=\int u^2(x)G_r(x-x_\circ)\,dx,
\]
and
\[
W_{r,x_\circ}:=H_{r,x_\circ}-r^{-2}D_{r,x_\circ}.
\]
This Gaussian Weiss quantity is globally defined and is a weighted version of the boundary-adjusted Weiss energy ${\bf W}(u,r)$ in \eqref{eq:W-def}.

For any function $w$ for which the following Gaussian integrals are finite (in particular, for any globally Lipschitz $w$), we write
\[
P_w(x):=\frac{|\nabla w|^2}{2}+\frac12\mathbf 1_{\{w>0\}},
\]
and $H_{r,x_\circ}(w)$, $D_{r,x_\circ}(w)$, $W_{r,x_\circ}(w)$ for the same quantities computed with $w$ in place of $u$. Thus $H_{r,x_\circ}=H_{r,x_\circ}(u)$, $D_{r,x_\circ}=D_{r,x_\circ}(u)$, and $W_{r,x_\circ}=W_{r,x_\circ}(u)$. When $x_\circ=0$, we write $H_r,D_r,W_r$.
\end{defn}

\begin{lem}[Gaussian Weiss monotonicity]
\label{lem:fixed-center-gaussian-weiss-monotonicity}
For every $x_\circ\in\R^n$ and $r > 0$, we have
\begin{equation}\label{eq:Weiss_Gaussian}
\partial_r W_{r,x_\circ} = \frac{2}{r^3} \int \bigl(u-(x-x_\circ)\cdot \nabla u\bigr)^2 G_r(x-x_\circ)\, dx.
\end{equation}
In particular, $r\mapsto W_{r,x_\circ}$ is nondecreasing.
\end{lem}

\begin{proof}
Let $x_\circ=0$ and $G=G_r$. We first differentiate the Gaussian $L^2$-term. Since
\[
\partial_rG_r=-\frac1r\bigl(nG_r+x\cdot\nabla G_r\bigr),
\]
we have, integrating by parts, and denoting $Y:=x\cdot\nabla u$,
\begin{equation}
\label{eq:fixed-center-D-prime} D_r'=-\frac1r\int u^2\bigl(nG+x\cdot\nabla G\bigr)\,dx = \frac2r\int uYG\,dx.
\end{equation}

We next differentiate $H_r$:
\[
H_r'=-\frac1r\int P\bigl(nG+x\cdot\nabla G\bigr)\,dx.
\]
The stationarity identity \eqref{eq:stationary} can be written in terms of $P$ as
\[
\int_{\R^n} P {\rm div} \xi \, dx = \int_{\R^n} \nabla u\, D\xi\, \nabla u \, dx,
\]
for $\xi \in C^\infty_c(\R^n)$. Taking $\xi(x)=xG_r(x)$ (up to a cutoff localization argument) we obtain
\[
\int P\bigl(nG+x\cdot\nabla G\bigr)\,dx =\int |\nabla u|^2G\,dx+\int Y\,\nabla u\cdot\nabla G\,dx.
\]
Because $\nabla G=-2xG/r^2$,
\[
H_r'=-\frac1r\int |\nabla u|^2G\,dx +\frac2{r^3}\int Y^2G\,dx.
\]
Since $u$ is harmonic in $\{u > 0\}$,
\[
\int |\nabla u|^2G\,dx =-\int u\,\nabla u\cdot\nabla G\,dx =\frac2{r^2}\int uYG\,dx,
\]
and hence
\begin{equation}
\label{eq:fixed-center-H-prime} H_r'=\frac2{r^3}\int (Y^2-uY)G\,dx.
\end{equation}

Combining \eqref{eq:fixed-center-D-prime} and \eqref{eq:fixed-center-H-prime},
\[
\begin{aligned}
\partial_rW_r &=H_r'+2r^{-3}D_r-r^{-2}D_r' =\frac2{r^3}\int (Y^2-uY)G\,dx +\frac2{r^3}\int u^2G\,dx -\frac2{r^3}\int uYG\,dx =\frac2{r^3}\int (u-Y)^2G\,dx.
\end{aligned}
\]
This concludes the proof.
\end{proof}

\begin{lem}
\label{lem:universal-gaussian-weiss-limit}
For every $x_\circ\in\R^n$, $r > 0$, and $e\in \mathbb{S}^{n-1}$,
\[
\lim_{R\to\infty}W_{R,x_\circ}=W_\infty, \qquad W_\infty:=\frac{\pi^{n/2}}2 = W_{r,0}(V_{0,e}).
\]
\end{lem}

\begin{proof}
By Assumption \ref{ass:standing-global-assumptions}, we have $|D^2u(0)|=1$. We apply Proposition \ref{prop:blow-down-of-non-flat-solutions} using that $|\nabla u|\le 1$ and the exponential decay of the Gaussian to control the tails: For fixed $L>1$, choose $e_{R,L}$ from \Cref{prop:blow-down-of-non-flat-solutions} at scale $LR$. Then
\[
\bigl|W_{R,x_\circ}(u)-W_{1,0}(V_{0,e_{R,L}})\bigr| \le C_L\omega((LR)^{-1})+C\frac{|x_\circ|}{R} +C\!\int_{\R^n\setminus B_L}(1+|x|^2)e^{-c|x|^2}\,dx .
\]
Letting first $R\to\infty$ and then $L\to\infty$ proves the limit, while
\[
W_{r,0}(V_{0,e}) =\int_{\R^n}\left(1-\frac{(e\cdot x)^2}{r^2}\right)G_r(x)\,dx =\frac{\pi^{n/2}}2,
\]
where the last equality follows from $\int_{\mathbb R^n}G_r=\pi^{n/2}$ and $\int_{\mathbb R^n}(e\cdot x)^2G_r(x)\,dx =\frac12r^2\pi^{n/2}$.
\end{proof}

Subtracting from the universal limit in \Cref{lem:universal-gaussian-weiss-limit} turns monotonicity into a quantitative measure of the failure of homogeneity.
\begin{lem}[Gaussian Weiss deficit]
\label{lem:gaussian-weiss-deficit-bound}
Let $W_\infty$ from \Cref{lem:universal-gaussian-weiss-limit} and define the fixed-center Gaussian Weiss deficit:
\[
\wW_{r,x_\circ}:=W_\infty-W_{r,x_\circ}.
\]
Then $\wW_{r,x_\circ}$ is nonnegative and nonincreasing in $r$. Moreover, for every $r>0$,
\begin{equation}\label{eq:gaussian-weiss-deficit-weighted-comparison}
c\int_{\R^n} \frac{\bigl(u(x)-(x-x_\circ)\cdot\nabla u(x)\bigr)^2} {|x-x_\circ|^{n+2}+r^{n+2}}\,dx \le \wW_{r,x_\circ} \le C\int_{\R^n} \frac{\bigl(u(x)-(x-x_\circ)\cdot\nabla u(x)\bigr)^2} {|x-x_\circ|^{n+2}+r^{n+2}}\,dx,
\end{equation}
and
\begin{equation}\label{eq:gaussian-weiss-deficit-gaussian-window}
\frac1{r^2}\int_{\R^n} \bigl(u(x)-(x-x_\circ)\cdot\nabla u(x)\bigr)^2G_r(x-x_\circ)\,dx \le C\wW_{r,x_\circ},
\end{equation}
where $c,C>0$ are universal.
\end{lem}

\begin{proof}
Fix $x_\circ=0$ after translation. By \Cref{lem:fixed-center-gaussian-weiss-monotonicity} and Tonelli's theorem,
\begin{equation}\label{eq:gaussian-weiss-deficit-kernel-representation}
\wW_r =\int_r^\infty W_\rho'\,d\rho =\int_{\R^n}(u-x\cdot\nabla u)^2K_r(x)\,dx, \qquad K_r(x):=2\int_r^\infty\rho^{-n-3}e^{-|x|^2/\rho^2}\,d\rho = |x|^{-n-2} \int_0^{|x|^2/r^2}t^{n/2}e^{-t}\,dt.
\end{equation}
By splitting the last integral below and above one, it is comparable to $\min\{(|x|/r)^{n+2},1\}\asymp \frac{1}{1+(r/|x|)^{n+2}}$. Hence
\[
K_r(x)\simeq\frac1{|x|^{n+2}+r^{n+2}},
\]
which proves \eqref{eq:gaussian-weiss-deficit-weighted-comparison}. Finally, with $t=|x|/r$,
\[
r^{-2}G_r(x)=r^{-n-2}e^{-t^2} \le \frac{Cr^{-n-2}}{1+t^{n+2}} \le CK_r(x).
\]
Inserting this in \eqref{eq:gaussian-weiss-deficit-kernel-representation} proves \eqref{eq:gaussian-weiss-deficit-gaussian-window}.
\end{proof}

We next compare, at different scales, this defect when the center moves transversely to a candidate vee ridge.
\begin{lem}
\label{lem:shifted-center-gaussian-weiss-almost-monotonicity}
There exists a universal $C$ such that for every $y\in\R^n$, $R>0$, $e\in\mathbb S^{n-1}$, and $\eta\in(0,1)$,
\[
\sup_{x_\circ\in (y+e^\perp)\cap B_R(y)} \wW_{R,x_\circ} \le 2\sup_{x_\circ\in (y+e^\perp)\cap B_{\eta R}(y)} \wW_{\eta R,x_\circ} + C |\log \eta|\, \wW_{\sqrt{2}\eta R,y}.
\]
\end{lem}

\begin{proof}
By translation we may assume $y=0$. Fix $\lambda\in[0,1)$ and $\theta\in e^\perp\cap\mathbb S^{n-1}$, and write
\[
\Theta(\rho):=\wW_{\rho,\lambda \rho\theta} =W_\infty-W_{\rho,\lambda\rho\theta}.
\]
For $\rho>0$, define
\[
\bar G_\rho(x):=G_\rho(x-\lambda \rho \theta).
\]
Thus $\Theta'(\rho)=-\frac{d}{d\rho}W_{\rho,\lambda\rho\theta}$, where the derivative on the right is total along $\rho\mapsto(\rho,\lambda\rho\theta)$.

The moving-center computation is the fixed-center first variation with $a(\rho)=\lambda\rho\theta$ and
\[
\frac{d}{d\rho}\bar G_\rho =-\frac1\rho\bigl(n\bar G_\rho+x\cdot\nabla \bar G_\rho\bigr).
\]
Differentiating $D_{\rho,a(\rho)}$, applying the stationarity condition with $\xi(x)=x\bar G_\rho(x)$, and using the harmonicity of $u$, as in the proof of \Cref{lem:fixed-center-gaussian-weiss-monotonicity} gives
\begin{equation}
\label{eq:shifted-gaussian-total-derivative} \frac{d}{d\rho}W_{\rho,\lambda \rho\theta} =\frac{2}{\rho^3}\int \bigl(u-x\cdot \nabla u\bigr) \bigl(u-(x-\lambda \rho\theta)\cdot \nabla u\bigr) \bar G_\rho(x)\,dx.
\end{equation}
Set
\[
A(\rho):= \frac{2}{\rho^2}\int \bigl(u-x\cdot \nabla u\bigr)^2 \bar G_\rho(x)\,dx, \qquad B(\rho):= \frac{2}{\rho^2}\int \bigl(u-(x-\lambda \rho\theta)\cdot \nabla u\bigr)^2 \bar G_\rho(x)\,dx.
\]
Applying Cauchy--Schwarz to \eqref{eq:shifted-gaussian-total-derivative},
\[
\left|\rho\frac{d}{d\rho}W_{\rho,\lambda \rho\theta}\right| \le A(\rho)^{1/2}B(\rho)^{1/2}.
\]

Because $\lambda\le 1$,
\[
\bar G_\rho(x) =\rho^{-n}e^{-\frac{|x-\lambda \rho\theta|^2}{\rho^2}} \le e\,\rho^{-n}e^{-\frac{|x|^2}{2\rho^2}} = C\,G_{\sqrt 2\,\rho}(x)
\]
for a universal constant $C$. Therefore, from \Cref{lem:fixed-center-gaussian-weiss-monotonicity},
\[
A(\rho) \le C\rho\partial_r W_{r,0}\big|_{r=\sqrt2\rho}.
\]
For the second factor, \eqref{eq:gaussian-weiss-deficit-gaussian-window} at the frozen center $\lambda\rho\theta$ yields
\[
B(\rho)\le C\,\wW_{\rho,\lambda \rho\theta}= C\,\Theta(\rho).
\]
Hence
\[
\left|\rho \Theta'(\rho)\right| =\left|\rho\frac{d}{d\rho}W_{\rho,\lambda \rho\theta}\right| \le C\bigl(\rho\partial_r W_{r,0}\big|_{r=\sqrt2\rho}\bigr)^{1/2} \Theta(\rho)^{1/2}\quad\Rightarrow \quad \rho\left|\frac{d}{d\rho}\sqrt{\Theta(\rho)}\right|^2 \le C\,\partial_r W_{r,0}\big|_{r=\sqrt2\rho}.
\]
Integrating from $\eta R$ to $R$, using Cauchy's inequality with weight $d\rho/\rho$, gives
\[
\bigl(\sqrt{\Theta(R)} -\sqrt{\Theta(\eta R)}\bigr)^2 \le C|\log\eta| \bigl(W_{\sqrt2 R,0}-W_{\sqrt2\eta R,0}\bigr) \le C|\log\eta|\,\wW_{\sqrt2\eta R,0}.
\]

Finally, the triangle inequality yields
\[
\Theta(R)\le 2\Theta(\eta R)+C|\log\eta|\,\wW_{\sqrt2\eta R,0}.
\]
Taking the supremum over $\lambda\in[0,1)$ and $\theta\in e^\perp\cap\mathbb S^{n-1}$ proves the result.
\end{proof}

\subsection{The combined Weiss--vee excess}

We combine the ridgewise Weiss deficit with a Monneau-type vee error and prove decay, logarithmic propagation, and recentering estimates for the resulting excess.

\begin{defn}
\label{defn:monneau-excess}
Let $\wW$ be as in \Cref{lem:gaussian-weiss-deficit-bound}, and define, for $y\in\R^n$, $R>0$, and $e\in\mathbb S^{n-1}$,
\[
\Theta_{R,y}^e:= \sup_{x\in (y+e^\perp)\cap B_R(y)} \wW_{R,x}, \qquad M_{R,y}^e:= R^{-n-2}\int_{B_R(y)}(u-V_{y,e})^2\,dx,
\]
and
\begin{equation}\label{eq:def-fixed-direction-Monneau}
\bbE_y(R;e):=\sqrt{\Theta_{R,y}^e+M_{R,y}^e}, \qquad \bbE_y(R):=\min_{e\in \mathbb S^{n-1}}\bbE_y(R;e).
\end{equation}
We call $\Theta_{R,y}^e$ the \emph{ridgewise Weiss deficit} and $M_{R,y}^e$ the \emph{Monneau-type $L^2$ vee excess}. We call $\bbE_y(R;e)$, and its minimized version $\bbE_y(R)$, the \emph{combined Weiss--vee excess}, or simply the \emph{combined excess}.
\end{defn}

We next record the two consequences used later: control of the transverse energy and decay at scales large compared with a neck radius.
\begin{lem}
\label{lem:monneau-excess-controls-transverse-energy}
There exists a universal $C$ such that
\[
\bE_y(R) \le C \bbE_y(R)
\]
for all $y\in \R^n$ and $R > 0$.
\end{lem}

\begin{proof}
Choose $e\in \mathbb S^{n-1}$ attaining the minimum in \eqref{eq:def-fixed-direction-Monneau}. In particular,
\[
\sup_{x_\circ\in (y+e^\perp)\cap B_R(y)} \wW_{R,x_\circ} \le \bbE_y(R)^2.
\]
and, by \eqref{eq:gaussian-weiss-deficit-gaussian-window}, for every $x_\circ\in (y+e^\perp)\cap B_R(y)$,
\[
\frac{1}{R^2}\int \bigl(u-(x-x_\circ)\cdot\nabla u\bigr)^2 G_R(x-x_\circ)\,dx \le C\wW_{R,x_\circ}\le C\bbE_y(R)^2.
\]

Fix such an $x_\circ$. For $x\in B_R(y)$ we have $|x-y|\le R$ and $|x-x_\circ|\le 2R$, so
\begin{equation}
\label{eq:monneau-kernel-compare} R^{-n}\chi_{B_R(y)} \le C\min\bigl\{G_R(\,\cdot\,-y),G_R(\,\cdot\,-x_\circ)\bigr\}
\end{equation}
for a universal constant $C$.

Since $x_\circ-y\in e^\perp$, we have $(x_\circ-y)\cdot \nabla u=(x_\circ-y)\cdot \nabla^e u.$ By the triangle inequality, applied once at $x_\circ$ and once at $y$, together with \eqref{eq:monneau-kernel-compare},
\begin{equation}\label{eq:monneau-directional-gradient-control}
\frac{1}{R^{n+2}}\int_{B_R(y)} \bigl((x_\circ-y)\cdot\nabla^e u\bigr)^2\,dx \le C\bbE_y(R)^2.
\end{equation}

Choose an orthonormal basis $\tau_1,\ldots,\tau_{n-1}$ of $e^\perp$ and apply \eqref{eq:monneau-directional-gradient-control} with $x_\circ-y=(R/2)\tau_i$. Summing over $i$, we get
\[
\frac1{R^n}\int_{B_R(y)} |\nabla^e u|^2\,dx\le C\bbE_y(R)^2,
\]
as we wanted.
\end{proof}

\begin{lem}
\label{lem:decay-of-monneau-excess}
We have
\[
\bbE_\zz(R) \le \omega_*(\rb(\zz)/R)\quad\text{for all}\quad \zz\in \cZ,\quad R > 0,
\]
for some universal modulus of continuity $\omega_*$.
\end{lem}

\begin{proof}

Set $M_\star:=|B_1|^{1/n}\eta_0^{-1}$, $\rho:=M_\star\rb(\zz)$, and $\bar u(x):=\frac{u(\zz+\rho x)}{\rho}$. By the definition of $\rb(\zz)$ and scale invariance,
\[
\int_{B_{1/M_\star}\cap\{\bar u>0\}}|D^2\bar u|^n =\eta_0^n=|B_{1/M_\star}|.
\]
Thus $\|D^2\bar u\|_{L^\infty(B_1\cap\{\bar u>0\})}\ge1$. The result now follows from \Cref{prop:blow-down-of-non-flat-solutions}, after scaling back, using its three estimates and the Gaussian-tail argument from the proof of \Cref{lem:universal-gaussian-weiss-limit}.
\end{proof}

We complement the asymptotic decay in \Cref{lem:decay-of-monneau-excess} with a quantitative comparison between two fixed scales.
\begin{prop}
\label{prop:fixed-direction-monneau-doubling}
There exists a universal constant $C$ such that for every $y\in\R^n$, $R>0$, $e\in\mathbb S^{n-1}$, and $0<\eta\le\frac12$,
\begin{equation}\label{eq:fixed-direction-monneau-doubling}
\bbE_y(R;e)\le C|\log\eta|\,\bbE_y(\eta R;e),\qquad \bbE_y(R)\le C|\log\eta|\,\bbE_y(\eta R).
\end{equation}
\end{prop}

\begin{proof}
For the sake of readability, we write
\[
\Theta_{R}:=\Theta_{R,y}^e,\qquad M_R:=M_{R,y}^e
\]
throughout the proof. By \Cref{lem:shifted-center-gaussian-weiss-almost-monotonicity},
\[
\Theta_R\le 2\Theta_{\eta R} +C|\log\eta|\,\wW_{\sqrt2\eta R,y}.
\]
Since $y\in (y+e^\perp)\cap B_{\eta R}(y)$, the monotonicity conclusion in \Cref{lem:gaussian-weiss-deficit-bound} gives $\wW_{\sqrt2\eta R,y}\le \wW_{\eta R,y}\le \Theta_{\eta R},$ and
\begin{equation}\label{eq:theta-scale-propagation}
\Theta_R\le C|\log\eta|\,\Theta_{\eta R}.
\end{equation}

We now claim that
\begin{equation}\label{eq:boundary-monneau-trace}
\mathfrak m_{\rho,y}^e:= \rho^{-n-1}\int_{\partial B_\rho(y)}(u-V_{y,e})^2\,d\cH^{n-1}\le C\bigl(M_{\rho}+\Theta_{\rho}\bigr).
\end{equation}

By translation assume $y=0$, and put $f=u-V_{0,e}$. For $\omega\in\mathbb S^{n-1}$, define
\[
F(s,\omega):=s^{-1}f(s\omega).
\]
Since $V_{0,e}$ is $1$-homogeneous,
\[
\partial_sF(s,\omega) =s^{-2}\bigl(s\,\partial_su(s\omega)-u(s\omega)\bigr).
\]
A one-dimensional averaging estimate on $[\rho/2,\rho]$ gives
\[
F(\rho,\omega)^2 \le C\rho^{-1}\int_{\rho/2}^{\rho}F(s,\omega)^2\,ds +C\rho\int_{\rho/2}^{\rho} \frac{\bigl(s\,\partial_su(s\omega)-u(s\omega)\bigr)^2}{s^4}\,ds.
\]
After integrating in $\omega$, the first term is bounded by $CM_{\rho,0}^e$. The second term is bounded by
\[
C\rho^{-n-2}\int_{B_\rho} \bigl(u-x\cdot\nabla u\bigr)^2\,dx.
\]
On $B_\rho$, the Gaussian kernel $G_\rho$ is bounded below by $c\rho^{-n}$. Therefore \eqref{eq:gaussian-weiss-deficit-gaussian-window} gives \eqref{eq:boundary-monneau-trace}:
\[
\rho^{-n-2}\int_{B_\rho}\bigl(u-x\cdot\nabla u\bigr)^2\,dx \le C\wW_{\rho,0} \le C\Theta_{\rho,0}^e.
\]

We now propagate the vee-excess term directly. For $s\in[\eta R,R]$,
\[
F(s,\omega)^2 \le 2F(\eta R,\omega)^2 +2|\log\eta|\int_{\eta R}^R \left(\rho\partial_\rho F(\rho,\omega)\right)^2\frac{d\rho}{\rho}.
\]
Multiplying by $s^{n+1}$, integrating in $s$ and $\omega$, and adding the inner ball $B_{\eta R}$, gives
\[
M_R \le C\eta^{n+2}M_{\eta R} +C\mathfrak m_{\eta R,0}^e +C|\log\eta| \int_{B_R\setminus B_{\eta R}} \frac{(u-x\cdot\nabla u)^2}{|x|^{n+2}}\,dx .
\]
The last integral is controlled by the centered Gaussian deficit. Indeed, \eqref{eq:gaussian-weiss-deficit-weighted-comparison} and $|x|\ge\eta R$ on the annulus give
\[
\int_{B_R\setminus B_{\eta R}} \frac{(u-x\cdot\nabla u)^2}{|x|^{n+2}}\,dx \le C\int_{\R^n} \frac{(u-x\cdot\nabla u)^2}{|x|^{n+2}+(\eta R)^{n+2}}\,dx \le C\wW_{\eta R,0} \le C\Theta_{\eta R}.
\]
Using \eqref{eq:boundary-monneau-trace} at the scale $\eta R$, and $|\log\eta|\ge\log2$, we obtain
\begin{equation}\label{eq:bulk-scale-propagation}
M_R\le C|\log\eta|\bigl(\Theta_{\eta R}+M_{\eta R}\bigr).
\end{equation}

Combining \eqref{eq:theta-scale-propagation} and \eqref{eq:bulk-scale-propagation}, we get the desired bound.
\end{proof}

Because the preferred center may change with the scale, we also need a recentering estimate:
\begin{lem}
\label{lem:large-scale-monneau-recentering}
Let $x\in\FB(u)$, $y\in\R^n$, $e\in\mathbb S^{n-1}$, and $R>0$. Assume $B_{8R}(x)\subset B_{32R}(y).$ Then
\begin{equation}\label{eq:large-scale-monneau-recentering}
\bbE_x(8R)\le C \bbE_y(32R;e),
\end{equation}
for some $C$ universal.
\end{lem}

\begin{proof}
We know $\bbE_x(8R)\le C$ because $|\nabla u|\le 1$ and $u(x) = 0$. Thus, we can always assume $\bbE_y(32R;e)\le\delta_0$ for some $\delta_0 > 0$ small universal. After translating and scaling, assume $R=1$, $y=0$, and let $\delta=\bbE_0(32;e).$ All constants below are universal. Since $M_{32,0}^e\le\delta^2$, H\"older's inequality gives the $L^1$ hypothesis in \Cref{lem:one-sided-lower-bound} at scale $32$, with right-hand side $C\delta$. Thus, after decreasing $\delta_0$,
\begin{equation}\label{eq:large-scale-recentering-slab}
u\ge V_{0,e}-C\delta \qquad\text{in }B_{24}.
\end{equation}
Because $x\in\FB(u)$ and $B_8(x)\subset B_{32}$, this implies $|e\cdot x|\le C\delta .$ Consequently
\begin{equation}\label{eq:large-scale-recentering-vee-shift}
|V_{x,e}-V_{0,e}|\le C\delta \qquad\text{in }\mathbb R^n .
\end{equation}
Using $B_8(x)\subset B_{32}$, \Cref{defn:monneau-excess}, and \eqref{eq:large-scale-recentering-vee-shift}, we obtain
\begin{equation}\label{eq:large-scale-recentering-bulk}
M_{8,x}^e \le C M_{32,0}^e+C\delta^2 \le C\delta^2 .
\end{equation}

We next compare the ridgewise Weiss deficits. Let
\[
s\in(x+e^\perp)\cap B_8(x),\qquad s':=s-(e\cdot s)e .
\]
Then $s'\in e^\perp$, $|s-s'|\le C\delta$, and $s'\in B_{32}\cap e^\perp$. Hence, by \Cref{defn:monneau-excess},
\[
\wW_{32,s'}\le \delta^2 .
\]
For every $z\in\R^n$, using $|\nabla u|\le1$ from \eqref{eq:assump_glob_u},
\[
|u(z)-(z-s)\cdot\nabla u(z)|^2 =|u(z)-(z-s')\cdot\nabla u(z)+(s-s')\cdot\nabla u(z)|^2 \le 2\bigl|u(z)-(z-s')\cdot\nabla u(z)\bigr|^2+C\delta^2 .
\]
The weighted comparison \eqref{eq:gaussian-weiss-deficit-weighted-comparison} gives
\[
\wW_{8,s} \le C\int_{\R^n} \frac{\bigl(u(z)-(z-s)\cdot\nabla u(z)\bigr)^2} {|z-s|^{n+2}+8^{n+2}}\,dz .
\]
Since $|s-s'|\le C\delta\le1$, the triangle inequality and the larger scale give
\[
|z-s|^{n+2}+8^{n+2} \ge c\bigl(|z-s'|^{n+2}+32^{n+2}\bigr).
\]
The extra $C\delta^2$ term contributes at most
\[
C\delta^2\int_{\R^n} \frac{dz}{|z-s|^{n+2}+8^{n+2}}\le C\delta^2 .
\]
Using \eqref{eq:gaussian-weiss-deficit-weighted-comparison} again at scale $32$ and center $s'$, we get
\[
\wW_{8,s}\le C\wW_{32,s'}+C\delta^2\le C\delta^2 .
\]
Taking the supremum over $s\in(x+e^\perp)\cap B_8(x)$ together with \eqref{eq:large-scale-recentering-bulk}, this yields the desired result.
\end{proof}

\section{Selection of power, center, and scale}
\label{sec:neck-radii-from-symmetric-excess}

\subsection{Initial center and scale}

We now set up the contradiction argument that will yield the desired classification result. From this point on, we work in dimensions $3\le n < \nAC$; the proof of the main theorem will eventually specialize to $n=4$ in the last section, \Cref{sec:final-contradiction-argument}. Throughout the rest of the paper, fix an exponent
\begin{equation}
\label{eq:beta_circ_def} \beta_\circ\in\left(0,1/2\right).
\end{equation}
Since $\beta_\circ$ remains fixed, dependence on it is suppressed in the notation and in subsequent dependency lists. The reader may fix $\beta_\circ = 1/4$, but we decide to track it to better identify the numerology in the exponents.

For $\zz,\zz'\in\cZ$ and $R\ge\max\{\rb(\zz),\rb(\zz')\}$, with $\zz'\in B_{3R/2}(\zz)$, we shall use the estimate
\begin{equation}\label{eq:rbeps}
\left(\frac{\rb(\zz')}{R}\right)^{\frac{n-2}{2}} \le C \varrho_\zz(4R)\le C \bE_\zz(16R)\le C\bbE_\zz(16R).
\end{equation}
The first inequality follows from \eqref{eq:varrho_def_intro}, the second from \Cref{prop:symmetric-excess-controls-neck-radii} at scale $4R$, and the third from \Cref{lem:monneau-excess-controls-transverse-energy}.

We shall also use the uniform bounds
\begin{equation}\label{eq:selection-uniform-bounds}
\bbE_\zz(8R)\le C \qquad\text{and}\qquad \varrho_\zz(2R)\le C \qquad\text{for}\quad R>0,\ \zz\in\cZ_R.
\end{equation}
The first follows from $|\nabla u|\le1$ and $u(\zz)=0$; the second follows from \Cref{prop:symmetric-excess-controls-neck-radii} at scale $2R$, \Cref{lem:monneau-excess-controls-transverse-energy}, and the first bound.

\begin{defn}[Selection functionals]
\label{defn:selection-functionals}
For every $\alpha>0$, define
\begin{equation}
\label{eq:selection-functional} F_u^\alpha(R) := \sup_{\zz\in \cZ_R} \frac{ \bbE_\zz(8R)}{ \varrho_\zz(2R)^{\alpha} }.
\end{equation}
Define also
\[
\alpha_\star:= \inf\left\{\alpha>0: \limsup_{R\to\infty}F_u^\alpha(R)=+\infty\right\}.
\]
\end{defn}
From
\[
\varrho_\zz(2R)\ge \left(\frac{\rb(\zz)}{2R}\right)^{\frac{n-2}{2}} \ge \left(\frac{\rmin}{2R}\right)^{\frac{n-2}{2}},
\]
and \eqref{eq:selection-uniform-bounds}, we have that $F_u^\alpha(R)$ is finite for every $R>0$.

\begin{remark}[The critical exponent is at most one]
\label{rem:critical-selection-exponent-at-most-one}
The exponent $\alpha_\star$ satisfies $\alpha_\star\le1$. Indeed, fix $\zz_0\in\cZ$. For every $\alpha>1$, \eqref{eq:rbeps} gives $\varrho_{\zz_0}(2R)\le C\bbE_{\zz_0}(8R)$. Hence, using \Cref{lem:decay-of-monneau-excess},
\[
F_u^\alpha(R)\ge c\,\bbE_{\zz_0}(8R)^{1-\alpha}\to+\infty \qquad\text{as }R\to\infty,
\]
along the large scales for which $\zz_0\in\cZ_R$.
\end{remark}

Let us define alpha $\alpha$ to be the following exponent:
\begin{equation}\label{eq:single-selection-alpha-choice}
\alpha:=
\begin{cases}
\dfrac{1+\max\{\alpha_\star,\frac12\}}2, &\text{if }\alpha_\star<1,\\[2mm]
1+\dfrac1{100}, &\text{if }\alpha_\star=1.
\end{cases}
\end{equation}
In either case,
\begin{equation}\label{eq:single-selection-alpha-range}
\frac34\le\alpha,\qquad \alpha_\star<\alpha\le1+\frac1{100}.
\end{equation}

The next lemma provides suitable centers and scales where we can start our argument:

\begin{lem}[Selection of initial center and scale]
\label{lem:selection-of-initial-center-and-scale}
There exist sequences $R_k>0$ and $\zz_k\in\cZ_{R_k}$, with $R_k\to\infty$, such that
\begin{equation}
\label{eq:eps_k} \left(\frac{\rb(\zz_k)}{R_k}\right)^{\frac{n-2}{2}} \le C\varrho_{\zz_k}(2R_k)\to0, \qquad \eps_k:=\bbE_{\zz_k}(8R_k)\to0,
\end{equation}
and
\begin{equation}\label{eq:selection-functional-maximality}
\bbE_\zz(8R) \le 2\frac{\varrho_{\zz}(2R)^{\alpha}} {\varrho_{\zz_k}(2R_k)^{\alpha}}\,\eps_k \qquad\text{for all } \zz\in\cZ_R,\ R\le R_k.
\end{equation}
\end{lem}

\begin{proof}
With the section exponent $\alpha>\alpha_\star$ fixed above, consider the nondecreasing envelope of $F_u^\alpha$, namely
\[
\widetilde F_u(R) : = \sup_{ R'\le R} F^\alpha_u(R'),
\]
which is finite for each $R$. Indeed, if $R'\le R$ and $\zz\in\cZ_{R'}$, then the summand corresponding to $\zz$ gives
\[
\varrho_\zz(2R') \ge \left(\frac{\rmin}{2R}\right)^{\frac{n-2}{2}},
\]
while the numerator is uniformly bounded by \eqref{eq:selection-uniform-bounds}.

Set
\[
S:=\left\{\beta>0:\limsup_{R\to\infty}F_u^\beta(R)=+\infty\right\}.
\]
This set is upward closed. Indeed, if $\alpha>\beta$, then \eqref{eq:selection-uniform-bounds} gives
\[
F_u^\alpha(R)\ge C^{\beta-\alpha}F_u^\beta(R).
\]
Since $\alpha>\alpha_\star=\inf S$, there is some $\beta\in S$ with $\beta<\alpha$, and hence $\alpha\in S$. Therefore $\widetilde F_u(R)\to\infty$ as $R\to\infty$. Choose a monotone increasing sequence $R_k\to\infty$ such that, for each $k$, the value $F_u^\alpha(R_k)$ almost realizes the running maximum $\widetilde F_u(R_k)$. Thus there exists $\zz_k\in \cZ_{R_k}$ satisfying
\begin{equation}
\label{eq:Ftildeineq} \tfrac 1 2 \widetilde F_u(R_k) \le \frac{\bbE_{\zz_k}(8R_k)}{\varrho_{\zz_k}(2R_k)^{ \alpha}}\le \widetilde F_u(R_k).
\end{equation}

By \eqref{eq:selection-uniform-bounds}, the numerator in \eqref{eq:Ftildeineq} is uniformly bounded. Thus the divergence of $\widetilde F_u(R_k)$ can only come from the denominator, so $\varrho_{\zz_k}(2R_k)\to0$. Then the numerator must also converge to zero, because by \Cref{lem:decay-of-monneau-excess} and \eqref{eq:varrho_def_intro},
\[
\bbE_{\zz_k}(8R_k) \le \omega_*\!\left(\frac{\rb(\zz_k)}{8R_k}\right) \le \omega_*\!\left(C\varrho_{\zz_k}(2R_k)^{\frac{2}{n-2}}\right)\to0.
\]
This proves \eqref{eq:eps_k}.

Finally, by the definition of $\widetilde F_u$,
\[
\frac{\bbE_\zz(8R)}{ \varrho_\zz(2R)^{ \alpha} } \le \widetilde F_u(R_k) \le 2 \frac{\bbE_{\zz_k}(8R_k)}{\varrho_{\zz_k}(2R_k)^{ \alpha} } = \frac{2\eps_k}{\varrho_{\zz_k}(2R_k)^{\alpha} } \qquad \mbox{for all $\zz\in \cZ_R$, $R\le R_k$,}
\]
which is exactly \eqref{eq:selection-functional-maximality}.
\end{proof}

\subsection{Refined selection and logarithmic lower bounds}

We refine the selected center and scale to control smaller necks while retaining a logarithmic lower bound for the combined excess.

Given $\zeta\in(0,1)$ and a ball $B_R(\zz)$, define
\begin{equation}
\label{eq:Ndef} \bN\big(\zeta,B_R(\zz)\big) :=(\zeta R)^{-n} \left|\bigcup_{\zz'\in\mathcal A^\zeta_{\zz,R}} B_{\zeta R}(\zz')\right|, \qquad \mathcal A^\zeta_{\zz,R}:=\cZ_{\zeta R}\cap B_R(\zz).
\end{equation}
This is the normalized covering size of $\cZ_{\zeta R}\cap B_R(\zz)$ at scale $\zeta R$.

\begin{lem}[Refined center and scale selection]
\label{lem:refined-center-and-scale-selection}
Let $R_k$ and $\zz_k$ be the sequences given by \Cref{lem:selection-of-initial-center-and-scale}. There exists $\tilde \zeta_k\in(0,1]$ such that, setting
\[
\Rk:=\tilde\zeta_k R_k, \qquad \epk:=\tilde\zeta_k^{\alpha\beta_\circ}\eps_k,
\]
there is a center $\zk\in\cZ_{\Rk}\cap B_{R_k}(\zz_k)$ such that $\Rk\to\infty$, $\epk\to0$, and the following properties hold:
\begin{align}\label{eq:selected-monneau-at-refined-scale}
\bbE_{\zk}(8\Rk)&\le C\epk, \\ \label{strongalternativeA}
\bN\big(\zeta, B_{\Rk}(\zk)\big) &\le C \zeta^{2-n-2\beta_\circ} \qquad\text{for all } \zeta\in(0,1), \\ \label{eq:rbepk}
\rb(\zz)&\le C\Rk\epk^{\frac{2}{n-2}} \qquad\text{for all } \zz\in\cZ\cap B_{3\Rk/2}(\zk), \\ \label{eq:countingeta}
\#\big((\cZ_{\eta\Rk}\setminus\cZ_{\eta\Rk/2})\cap B_{\Rk}(\zk)\big) \eta^{n-2} &\le C\epk^2 \qquad\qquad\text{for all } \eta>0,
\end{align}
and, for every $\zeta\in(0,\frac12)$ and every $\cF\subset\cZ_{\zeta\Rk}\cap B_{\Rk}(\zk)$ such that the balls $\{B_{\zeta\Rk}(\zz')\}_{\zz'\in\cF}$ are pairwise disjoint,
\begin{equation}\label{keyeqn-alpha}
\sum_{\zz'\in \cF} \bbE_{\zz'}(8\zeta\Rk)^{2/\alpha} \le C_\alpha \zeta^{2-n}\epk^{2/\alpha}.
\end{equation}
Here $C$ is universal, $C_\alpha$ depends only on $\alpha$, and $\bN$ is given by \eqref{eq:Ndef}.
\end{lem}

\begin{proof}
We divide the proof into four steps.

\medskip
\noindent\emph{Step 1.}
We call $\zeta\in(0,\frac14]$ admissible if there exists $\zz\in\cZ_{\zeta R_k}$ such that
\[
B_{4\zeta R_k}(\zz)\subset B_{R_k}(\zz_k), \qquad \frac{\varrho_{\zz}(4\zeta R_k)}{\varrho_{\zz_k}(2R_k)} \le 8\zeta^{\beta_\circ}.
\]
Define $\zeta_k$ as the infimum of the admissible numbers. The admissible class is nonempty by \eqref{eq:eps_k}: for large $k$, $\rb(\zz_k)\le R_k/4$, and then $\zeta=\frac14$, $\zz=\zz_k$, is admissible (using also that $\varrho_{\zz_k}(R_k)\le 2^{\frac{n-2}{2}}\varrho_{\zz_k}(2R_k)\le 8 \cdot 4^{-\beta_\circ}\varrho_{\zz_k}(2R_k)$ since $n \le 6$ and $\beta_\circ<1/2$). Since $\rb(\zz)\ge\rmin>0$ for all $\zz\in\cZ$, we also have $\zeta_k\ge c/R_k>0$.

Choose $\tilde\zeta_k\in[\zeta_k,\min\{2\zeta_k,\frac14\}]$ and $\zk\in\cZ_{\Rk}$, with $\Rk:=\tilde\zeta_k R_k$, such that the preceding admissibility conditions hold with $\zeta=\tilde\zeta_k$ and $\zz=\zk$. Then $B_{4\Rk}(\zk)\subset B_{R_k}(\zz_k)$, and in particular
\begin{equation}\label{key3}
B_{\Rk}(\zk)\subset B_{R_k}(\zz_k),\qquad \frac{\varrho_{\zk}(4\Rk)}{\varrho_{\zz_k}(2R_k)} \le 8\tilde\zeta_k^{\beta_\circ}.
\end{equation}
Moreover $\varrho_{\zk}(4\Rk)\le 8\varrho_{\zz_k}(2R_k)\to0$. Applying \eqref{eq:rbeps} with center $\zk$, scale $\Rk$, and $\zz'=\zk$ gives $\Rk\to\infty$. We also have $\epk=\tilde\zeta_k^{\alpha\beta_\circ}\eps_k\le \eps_k\to0.$

The minimality of $\zeta_k$ gives the following fact. If $t\in(0,\frac14)$, $\zz\in\cZ_{t\Rk}$, and $B_{4t\Rk}(\zz)\subset B_{4\Rk}(\zk)$, then
\begin{equation}
\label{eq:rhoI} \frac{\varrho_{\zz}(4t\Rk)}{\varrho_{\zk}(4\Rk)} >t^{\beta_\circ}.
\end{equation}
Indeed, if this failed then $\zeta=t\tilde\zeta_k<\zeta_k$ would be admissible in the definition of $\zeta_k$, since
\[
B_{4t\Rk}(\zz)\subset B_{4\Rk}(\zk)\subset B_{R_k}(\zz_k)
\]
and the admissibility of $\zk$, \eqref{key3}, would give
\[
\frac{\varrho_{\zz}(4t\Rk)}{\varrho_{\zz_k}(2R_k)} \le 8 (t\tilde\zeta_k)^{\beta_\circ},
\]
contradicting the definition of $\zeta_k$.

\medskip
\noindent\emph{Step 2.} The estimate \eqref{eq:selection-functional-maximality}, applied with $R=\Rk$ and center $\zk$, gives
$\bbE_{\zk}(8\Rk)\le C\epk$, because $\varrho_{\zk}(2\Rk)\le C\varrho_{\zk}(4\Rk)$. This proves \eqref{eq:selected-monneau-at-refined-scale}. Hence, by \eqref{eq:rbeps} with $\zz'=\zk$ and \eqref{eq:fixed-direction-monneau-doubling} (to pass from scale $16\Rk$ to $8\Rk$),
\[
\varrho_{\zk}(4\Rk)\le C\epk\qquad\text{and}\qquad \left(\frac{\rb(\zk)}{\Rk}\right)^{\frac{n-2}{2}} \le C\epk.
\]
Since $\epk\to0$, after increasing $k$ we have $2\Rk\ge M_\circ\rb(\zk)$. Applying \Cref{lem:upper-bound-for-nearby-neck-radii} with center $\zk$ and radius $2\Rk$, we get $\rb(\zz)\le \Rk/4$ for every $\zz\in\cZ\cap B_{3\Rk/2}(\zk)$. Thus \eqref{eq:rbeps}, applied with base center $\zk$, scale $R=\Rk$, and nearby center $\zz'=\zz$, is applicable for every such $\zz$, and yields
\[
\left(\frac{\rb(\zz)}{\Rk}\right)^{\frac{n-2}{2}} \le C\varrho_{\zk}(4\Rk) \le C\epk.
\]
This proves \eqref{eq:rbepk}.

For the dyadic scale-slice estimate, set
\[
S_\eta:=(\cZ_{\Rk\eta}\setminus\cZ_{\Rk\eta/2})\cap B_{\Rk}(\zk).
\]
By \eqref{eq:rbepk}, we may assume $\eta\le \frac12$. For such $\eta$, $S_\eta$ is contained in the index set defining $\varrho_{\zk}(4\Rk)^2$, and
\[
\begin{aligned}
c\,\#S_\eta\,\eta^{n-2} &\le C\sum_{\zz\in S_\eta} \left(\frac{\rb(\zz)}{\Rk}\right)^{n-2}\le C\varrho_{\zk}^2(4\Rk).
\end{aligned}
\]
Together with $\varrho_{\zk}(4\Rk)\le C\epk$ this gives \eqref{eq:countingeta}.

\medskip
\noindent\emph{Step 3.} If
$\cF\subset\cZ_{\zeta\Rk}\cap B_{\Rk}(\zk)$, $\zeta\in(0,\frac12)$, and the balls $B_{\zeta\Rk}(\zz')$, $\zz'\in\cF$, are pairwise disjoint, then
\begin{equation}
\label{eq:rho-family-sum-refined} \zeta^{n-2}\sum_{\zz'\in\cF}\varrho_{\zz'}^2(2\zeta\Rk) \le C\varrho_{\zk}^2(4\Rk).
\end{equation}
Indeed, the index sets in the left-hand sum are disjoint and their union is contained in $\cZ_{2\Rk}\cap B_{2\Rk}(\zk)$. Expanding \eqref{eq:varrho_def_intro}, the change from the denominator $2\Rk$ to $4\Rk$ costs exactly $2^{n-2}$, proving \eqref{eq:rho-family-sum-refined}.

If $\zeta\in[\frac12,1)$, then $\bigcup_{\zz'\in\cA^\zeta_{\zk,\Rk}}B_{\zeta\Rk}(\zz')\subset B_{2\Rk}(\zk)$. Hence, by \eqref{eq:Ndef},
\[
\bN(\zeta,B_{\Rk}(\zk)) \le \frac{|B_{2\Rk}|}{(\zeta\Rk)^n} \le C \le C\zeta^{2-n-2\beta_\circ},
\]
which proves \eqref{strongalternativeA} in this range. Assume from now on that $\zeta\in(0,\frac12)$. Split
\[
\cA^\zeta_{\zk,\Rk} = \big(\cZ_{\zeta\Rk/2}\cap B_{\Rk}(\zk)\big) \cup \big((\cZ_{\zeta\Rk}\setminus\cZ_{\zeta\Rk/2})\cap B_{\Rk}(\zk)\big) =:\cA_-^\zeta\cup\cA_+^\zeta .
\]

For the upper annulus, \eqref{eq:countingeta} with $\eta=\zeta$ gives
\[
\#\cA_+^\zeta \le C\zeta^{2-n}\epk^2 \le C\zeta^{2-n-2\beta_\circ}.
\]
Thus the part of $\bN(\zeta,B_{\Rk}(\zk))$ generated by $\cA_+^\zeta$ is acceptable.

For $\cA_-^\zeta$, choose a Vitali subfamily $\{B_{\zeta\Rk}(\zz')\}_{\zz'\in\mathcal F}$ of disjoint balls such that
\[
\bigcup_{\zz'\in\cA_-^\zeta}B_{\zeta\Rk}(\zz') \subset \bigcup_{\zz'\in\mathcal F}B_{3\zeta\Rk}(\zz').
\]
For every $\zz'\in\mathcal F$, we have $\zz'\in\cZ_{\zeta\Rk/2}$, and
\[
B_{2\zeta\Rk}(\zz')\subset B_{2\Rk}(\zk)\subset B_{4\Rk}(\zk) \subset B_{R_k}(\zz_k).
\]
Taking $t=\zeta/2$ in \eqref{eq:rhoI} yields
\[
\varrho_{\zz'}(2\zeta\Rk) \ge c\,\zeta^{\beta_\circ}\varrho_{\zk}(4\Rk).
\]
On the other hand, applying \eqref{eq:rho-family-sum-refined} to the disjoint family $\mathcal F$ gives
\[
\zeta^{n-2} \sum_{\zz'\in\mathcal F}\varrho_{\zz'}^2(2\zeta\Rk) \le C\varrho_{\zk}^2(4\Rk).
\]
Therefore $\#\mathcal F\le C\zeta^{2-n-2\beta_\circ}$, and
\[
\left|\bigcup_{\zz'\in\cA_-^\zeta}B_{\zeta\Rk}(\zz')\right| \le C(\zeta\Rk)^n\zeta^{2-n-2\beta_\circ}.
\]
Combining the estimates for $\cA_-^\zeta$ and $\cA_+^\zeta$, and dividing by $(\zeta\Rk)^n$, proves \eqref{strongalternativeA}.

\medskip
\noindent\emph{Step 4.} Let $\cF\subset\cZ_{\zeta\Rk}\cap B_{\Rk}(\zk)$ be as in the statement, with
$\zeta\in(0,\frac12)$. Applying \eqref{eq:selection-functional-maximality} with $R=\zeta\Rk$, and using \eqref{key3} and $\epk=\tilde\zeta_k^{\alpha\beta_\circ}\eps_k$, gives
\[
\bbE_{\zz'}(8\zeta\Rk) \le C\epk \left( \frac{\varrho_{\zz'}(2\zeta\Rk)} {\varrho_{\zk}(4\Rk)} \right)^\alpha \qquad\text{for every }\zz'\in\cF .
\]
Taking the power $2/\alpha$, summing, and using \eqref{eq:rho-family-sum-refined}, we obtain
\[
\sum_{\zz'\in\cF}\bbE_{\zz'}(8\zeta\Rk)^{2/\alpha} \le C_\alpha \frac{\epk^{2/\alpha}}{\varrho_{\zk}^2(4\Rk)} \sum_{\zz'\in\cF}\varrho_{\zz'}^2(2\zeta\Rk) \le C_\alpha\zeta^{2-n}\epk^{2/\alpha}.
\]
This proves \eqref{keyeqn-alpha}.
\end{proof}

The optimized denominator from \Cref{lem:selection-of-initial-center-and-scale} and the scale propagation in \Cref{prop:fixed-direction-monneau-doubling} now provide the lower bound needed to start the tree argument below.
\begin{prop}
\label{prop:optimized-denominator-logarithmic-lower-bounds}
Let $(R_k,\zz_k)$ and $(\Rk,\zk)$, with $\eps_k$ and $\epk$, be given by \Cref{lem:selection-of-initial-center-and-scale} and \Cref{lem:refined-center-and-scale-selection}. Then, for all $k$ sufficiently large,
\begin{equation}\label{eq:optimized-denominator-monneau-lower}
\bbE_{\zk}(R)\ge \frac{\eps_k}{C\bigl(1+\log(8R_k/R)\bigr)} \qquad\text{for all }0<R\le8R_k,
\end{equation}
and
\begin{equation}\label{eq:optimized-denominator-selected-scale-comparable}
\Rk\ge cR_k,
\end{equation}
for some $C = C(\alpha)$ and $c = c(\alpha)>0$.
\end{prop}

\begin{proof}
Set $t_k:=\Rk/R_k=\tilde\zeta_k$. We first prove the combined-excess lower bound. Fix $0<R\le 8R_k$, and choose $e_R\in\mathbb S^{n-1}$ with
\[
\bbE_{\zk}(R;e_R)\le 2\bbE_{\zk}(R).
\]
Since $R/(32R_k)\le1/4$, \Cref{prop:fixed-direction-monneau-doubling} gives
\begin{equation}\label{eq:optimized-denominator-proof-upward}
\bbE_{\zk}(32R_k;e_R) \le C\bigl(1+\log(8R_k/R)\bigr)\bbE_{\zk}(R).
\end{equation}
Apply \Cref{lem:large-scale-monneau-recentering} with $y=\zk$, $x=\zz_k$, $e=e_R$, and scale parameter $R_k$. The geometric hypothesis holds because $\zk\in B_{R_k}(\zz_k)$, hence
\[
B_{8R_k}(\zz_k)\subset B_{9R_k}(\zk)\subset B_{32R_k}(\zk).
\]
Thus
\[
\eps_k=\bbE_{\zz_k}(8R_k) \le C\bbE_{\zk}(32R_k;e_R).
\]
Combining this with \eqref{eq:optimized-denominator-proof-upward} yields
\[
\bbE_{\zk}(R)\ge \frac{\eps_k}{C\bigl(1+\log(8R_k/R)\bigr)},
\]
which proves \eqref{eq:optimized-denominator-monneau-lower}.

It remains to compare the two selected scales. By \eqref{eq:selected-monneau-at-refined-scale},
\[
\bbE_{\zk}(8\Rk) \le C t_k^{\alpha\beta_\circ}\eps_k =C\epk .
\]
On the other hand, \eqref{eq:optimized-denominator-monneau-lower} at $R=8\Rk$ gives
\[
\bbE_{\zk}(8\Rk) \ge \frac{\eps_k}{C\bigl(1+\log(R_k/\Rk)\bigr)}.
\]
Therefore
\[
\epk\ge \frac{\eps_k}{C\bigl(1+\log(R_k/\Rk)\bigr)}.
\]
Since $\epk=t_k^{\alpha\beta_\circ}\eps_k$, this also gives
\[
t_k^{\alpha\beta_\circ} \ge \frac{1}{C\bigl(1+\log(1/t_k)\bigr)}.
\]
Since $(1+\log(1/t))t^{\alpha\beta_\circ}\to0$ as $t\downarrow0$, the preceding inequality forces $t_k\ge c>0$ after increasing $k$. This is \eqref{eq:optimized-denominator-selected-scale-comparable}.
\end{proof}

\section{Free-boundary ball tree}
\label{sec:fb-tree}

\subsection{Construction and affine propagation}

We start with the definition.

\begin{defn}[Free-boundary tree dyadic balls and centered vee excess]
\label{defn:fb-tree-dyadic-balls-centered-vee-excess}
Let
\[
B_{\rm root}:=B_{\Rk}(\zk),\qquad r_l:=2^{-l}\Rk\quad\text{for}\quad l\ge0.
\]
For a ball $B=B_{r_B}(x_B)$ and $\lambda>0$, we write $r_B$ for its radius and $x_B$ for its center, and set
\[
\lsup{\lambda}{B}:=B_{\lambda r_B}(x_B).
\]
For every $l\ge1$, choose a maximal $r_l/4$-separated family of points $x_{l,j}\in\FB(u)\cap B_{\rm root}$ indexed by $j\in J_l$, and let
\[
\cQ^l:=\{B_{r_l}(x_{l,j})\}_{j\in J_l}.
\]
In particular, we have $\FB(u)\cap B_{\rm root} \subset \bigcup_{B\in\cQ^l}\lsup{1/4}{B}$, the balls $\{\lsup{1/16}{B}:B\in\cQ^l\}$ are pairwise disjoint, and
\begin{equation}\label{eq:fb-tree-surface-cardinality}
\#\cQ^l\le C_n 2^{l(n-1)},
\end{equation}
which follows from the disjointness, the lower density estimate \eqref{eq:density-B}, and the perimeter bound \Cref{lem:perimeter-bound}. Set $\cQ^0:=\{B_{\rm root}\}$.

For $B\in\cQ^l$, $l\ge0$, define
\begin{equation}\label{eq:fb-tree-centered-vee-excess}
h_B^2:= \inf_{e\in\mathbb S^{n-1}} \frac{1}{r_B^{n+2}} \int_{\lsup{16}{B}} \big|u-V_{x_B,e}\big|^2\,dx.
\end{equation}
Fix one minimizing vee $V_B$. It has exactly two unit-slope affine representatives, which differ by sign. Their polarity will be fixed simultaneously along the tree in \Cref{lem:choice}.
\end{defn}

We first record the pointwise consequence of centered vee closeness that will be used throughout the tree construction.
\begin{cor}
\label{cor:centered-vee-lower-slab}
Let $x_0\in\FB(u)$, $r>0$, $e\in\mathbb S^{n-1}$, and $\eta\in(0,1)$ satisfy
\[
\frac1r\fint_{B_{16r}(x_0)} \bigl|u(x)-V_{x_0,e}(x)\bigr|\,dx\le \eta.
\]
Then, in $B_{12r}(x_0)$,
\[
u+C\eta r\ge V_{x_0,e}, \qquad \FB(u)\subset\{V_{x_0,e}\le C\eta r\}, \qquad B_{12r}(x_0)\cap \{\pm e\cdot(x-x_0)>C\eta r\} \subset\{u>0\},
\]
for some $C$ dimensional.
\end{cor}

\begin{proof}
After enlarging the dimensional constant in \Cref{lem:one-sided-lower-bound}, we obtain
\[
u+C\eta r\ge V_{x_0,e} \qquad\text{in }B_{12r}(x_0),
\]
which is the lower bound. If $p\in\FB(u)\cap B_{12r}(x_0)$, then $u(p)=0$, so the same inequality gives
\[
V_{x_0,e}(p)\le C\eta r .
\]
This proves the slab containment. Finally, if $\pm e\cdot(x-x_0)>C\eta r$, then $V_{x_0,e}(x)>C\eta r$; after enlarging $C$ once more in the lower bound, we obtain $u(x)>0$.
\end{proof}

With \Cref{cor:centered-vee-lower-slab} in hand, we can define the branching and stopping rules of the dyadic tree.
\begin{defn}[Free-boundary tree branching and terminal balls]
\label{defn:fb-tree-branching-and-terminal-balls}
Fix a tree parameter $0<\delta\ll1$, to be chosen below the dimensional thresholds in the statements that follow (with $k$ increased accordingly). We define \emph{active balls} recursively by the following construction. First, denote
\[
\cD^0:=\{B_{\rm root}\}.
\]
Once $\cD^l$ is defined, $B\in\cD^l$ is called {\em branching} if
\begin{equation}\label{eq:fb-tree-branching-condition}
\lsup{4}{B}\cap\cZ\ne\varnothing, \qquad h_B<\delta.
\end{equation}
Otherwise $B$ is {\em terminal}. Write
\[
\cB^l:=\{B\in\cD^l:B\text{ is branching}\}.
\]
For $l\ge0$, put
\begin{equation}\label{eq:fb-tree-active-next-generation}
\cD^{l+1}:= \{B'\in\cQ^{l+1}:x_{B'}\in\lsup{5/4}{B} \text{ for some }B\in\cB^l\}.
\end{equation}
For $B\in\cB^l$, define its local children by
\begin{equation}\label{eq:fb-tree-local-children}
\operatorname{Ch}(B):= \{B'\in\cD^{l+1}:x_{B'}\in\lsup{5/4}{B}\}.
\end{equation}
For every $B'\in\cD^{l+1}$, choose one parent $PB'\in\cB^l$ with $x_{B'}\in\lsup{5/4}{PB'}$. This predecessor will be used only to propagate the polarity from the root and to record ancestry; write $P^0B=B$ and $P^{j+1}B=P(P^jB)$ whenever these iterates are defined. The local child family $\operatorname{Ch}(B)$ is used for covers. In particular,
\begin{equation}\label{eq:fb-tree-active-cover-near-branching}
\FB(u)\cap B_{\rm root}\cap\lsup{9/8}{B} \subset \bigcup_{B'\in\operatorname{Ch}(B)}\lsup{1/4}{B'}.
\end{equation}
Finally put $\cB:=\bigcup_l\cB^l$.

A terminal ball is {\em regular} if
\[
\lsup{4}{B}\cap\cZ=\varnothing,
\]
and {\em neck-like} otherwise. We denote the two classes by $\cR$ and $\cN$, respectively, and write
\[
\cR^l:=\cR\cap\cD^l,\qquad \cN^l:=\cN\cap\cD^l .
\]
For $k$ sufficiently large, the root $B_{\rm root}=B_{\Rk}(\zk)$ is branching: $\zk\in\cZ\cap\lsup{4}{B_{\rm root}}$, and \Cref{lem:blow-down-around-a-neck-center} gives $h_{B_{\rm root}}<\delta$ since $\rb(\zk)/\Rk\to0$.
\end{defn}

\begin{lem}
\label{lem:fb-tree-properties}
There is a dimensional constant $C_{\rm inh}$ such that every active ball satisfies
\begin{equation}\label{eq:fb-tree-active-child-flatness}
h_B\le C_{\rm inh}\delta .
\end{equation}
\end{lem}
\begin{proof}
The root already satisfies $h_{B_{\rm root}}<\delta$. Let $B\in\cD^{l+1}$, and set $A=PB$. The parent $A$ is branching, hence $h_A<\delta$, and $x_B\in\FB(u)\cap\lsup{5/4}{A}$. Write $V_A=V_{x_A,e_A}$. By Cauchy--Schwarz, the hypothesis of \Cref{cor:centered-vee-lower-slab} holds on $A$ with $\eta=Ch_A$. Its slab conclusion gives
\[
V_A(x_B)\le C h_A r_A.
\]
Therefore the centered vee
\[
\widetilde V_B:=V_{x_B,e_A}
\]
is $Ch_A r_A$-close in $L^\infty$ to $V_A$. Since $\lsup{16}{B}\subset\lsup{16}{A}$ and $r_A=2r_B$, \eqref{eq:fb-tree-centered-vee-excess}, with $\widetilde V_B$ as competitor, yields
\[
h_B^2 \le \frac{C}{r_B^{n+2}} \left( \int_{\lsup{16}{B}}|u-V_A|^2 +r_B^n h_A^2r_A^2 \right) \le C h_A^2 \le C\delta^2.
\]
This proves \eqref{eq:fb-tree-active-child-flatness}.
\end{proof}

To propagate the affine representatives coherently through the tree, we need the following comparison of unit-slope vees.
\begin{lem}
\label{lem:unit-slope-vee-comparison}
Let $\ell_1,\ell_2$ be affine functions in $\R^n$ with $|\nabla\ell_1|=|\nabla\ell_2|=1$. Then, for every ball $B_r(x_0)$,
\[
\min_{\sigma=\pm1} \Bigl( r|\nabla\ell_1-\sigma\nabla\ell_2| +\|\ell_1-\sigma\ell_2\|_{L^\infty(B_r(x_0))} \Bigr) \le C r\left( r^{-n-2}\int_{B_r(x_0)} \bigl||\ell_1|-|\ell_2|\bigr|^2\,dx \right)^{1/2},
\]
where $C$ is a dimensional constant.
\end{lem}

To prove this estimate, one fixes the sign on a subset of $B_r(x_0)$ of measure at least $|B_r|/2$ and applies \Cref{lem:affine-comparison-proportional}. The details are included in \Cref{app:proofs-standard-auxiliary-lemmas}.

At this point, we have all the ingredients to propagate an orientation/polarity for the dyadic tree through child-parent links.

\begin{lem}[Choice of representatives]
\label{lem:choice}
We choose the sign for an affine representative of $V_{B_{\rm root}}$ and denote it by $\ell_{B_{\rm root}}$, $V_{B_{\rm root}} = |\ell_{B_{\rm root}}|$. After decreasing $\delta$, this choice propagates uniquely along the predecessor map. More precisely, for every active ball $B$, we define the unit-slope affine function $\ell_B$ and its oriented direction $a_B$ by
\begin{equation}\label{eq:fb-tree-oriented-affine-representatives}
V_B=|\ell_B|,\qquad \ell_B(x_B)=0,\qquad a_B:=\nabla\ell_B\in\mathbb S^{n-1},
\end{equation}
and, for every $B\in\cD^l$, $l\ge1$,
\begin{equation}\label{eq:fb-tree-predecessor-affine-comparison}
\|\ell_B-\ell_{PB}\|_{L^\infty(\lsup{4}{B})} \le Cr_B(h_B+h_{PB}), \qquad |a_B-a_{PB}|\le C(h_B+h_{PB}).
\end{equation}
\end{lem}

\begin{proof}
We now orient the vees. Fix the chosen representation at the root and suppose that $\ell_A$ has been chosen for a branching parent $A$. For a child $B$, choose either affine representative $\widetilde\ell_B$ of $V_B$. Since $\lsup4B\subset\lsup{16}{A}$, comparison through $u$ and \Cref{lem:unit-slope-vee-comparison} give
\[
\min_{\sigma=\pm1} \Bigl( \|\sigma\widetilde\ell_B-\ell_A\|_{L^\infty(\lsup4B)} +r_B|\sigma\nabla\widetilde\ell_B-a_A| \Bigr) \le Cr_B(h_B+h_A).
\]
Choose the minimizing sign and call the resulting affine function $\ell_B$. By \eqref{eq:fb-tree-active-child-flatness}, the directions of parent and child differ by at most $C\delta$. After decreasing $\delta$, this is less than $1/2$, and in particular their angle is less than $\pi/2$. The opposite representative makes an angle greater than $\pi/2$, so the choice is unambiguous. Iterating from the root defines every $\ell_B$ and proves \eqref{eq:fb-tree-predecessor-affine-comparison}.
\end{proof}

The initial root polarity is arbitrary. What matters is that it propagates consistently through the whole tree; this is what will make the later two signed domains $U_+$ and $U_-$ well defined.

For later use, set
\begin{equation}\label{eq:fb-tree-root-direction}
e_{\rm root}:=a_{B_{\rm root}} = \nabla \ell_{B_{\rm root}},
\end{equation}
which is the oriented root vector fixed by \Cref{lem:fb-tree-properties}.

\subsection{Splitting and signed decomposition}

We first turn accurate vee approximation into a two-sheet splitting and then assemble the coherently oriented sheets into a signed decomposition.

\begin{lem}[Vee splitting under a Hessian bound]
\label{lem:vee-splitting-under-hessian-bound}
For every $C_0\ge1$ there are constants $h_{\rm split}=h_{\rm split}(n,C_0)>0$ and $C_{\rm split}=C_{\rm split}(n,C_0)$ with the following property. Let $\bar u$ be a global classical solution to the Bernoulli problem in $\mathbb R^n$, with $0\in\FB(\bar u)$, such that
\[
\|D^2\bar u\|_{L^\infty(B_{3}\cap\{\bar u>0\})}\le C_0
\]
and, for some $0\le h\le h_{\rm split}$,
\[
\bar u\ge |x_n|-h\qquad\text{in }B_4 .
\]
Then $\{\bar u>0\}\cap B_{2}$ is the disjoint union of two regular components $\Omega^+_{2}$ and $\Omega^-_{2}$, where
\[
\Omega^+_{2}=\{x_n>g^+(x')\}\cap B_{2},\qquad \Omega^-_{2}=\{x_n<g^-(x')\}\cap B_{2},
\]
with $g^\pm\colon D_{2}\to\mathbb R$, $D_{2}:=\{x'\in\mathbb R^{n-1}:|x'|<2\}$, $g^-<g^+$, and
\[
\|g^+\|_{C^2(D_{2})}+\|g^-\|_{C^2(D_{2})}\le C_{\rm split}h .
\]
Moreover, for $0\le j\le2$,
\[
\|D^j(\bar u-x_n)\|_{L^\infty(\Omega^+_{2})} +\|D^j(\bar u+x_n)\|_{L^\infty(\Omega^-_{2})} \le C_{\rm split}h .
\]
\end{lem}

\begin{proof}
We first observe that there is a modulus $\omega$, depending only on $n$ and $C_0$, such that
\begin{equation}\label{eq:vee-splitting-rough-linfty}
\|\bar u-|x_n|\|_{L^\infty(B_{3})} \le \omega(h), \qquad \omega(h)\downarrow0\quad\text{as }h\downarrow0 .
\end{equation}
Indeed, otherwise take $h_i\downarrow0$ and solutions $\bar u_i$ violating \eqref{eq:vee-splitting-rough-linfty}. By Lipschitz compactness and standard Bernoulli compactness, a subsequence converges locally uniformly to a solution $\bar u_\infty$. The lower bound gives $\bar u_\infty\ge |x_n|$. Since $0\in\FB(\bar u_i)$ and $|\nabla\bar u_i|\le1$, we also have $\bar u_\infty(te_n)\le |t|$; hence equality holds on the vertical segment. In $B_{3}\cap\{x_n>0\}$, the function $\bar u_\infty-x_n$ is nonnegative and harmonic, and it vanishes on an interior segment. The strong maximum principle gives $\bar u_\infty=x_n$ in the upper half-ball. The same argument applied to $\bar u_\infty+x_n$ in the lower half-ball gives $\bar u_\infty=-x_n$ there. Therefore $\bar u_i\to |x_n|$ uniformly in $B_{3}$, a contradiction.

Fix a dimensional $s\in(0,1/8)$. If $z\in\{x_n=0\}\cap B_{5/2}$, then \eqref{eq:vee-splitting-rough-linfty} and the Hessian bound allow us to apply \Cref{cor:vee-plus-hessian-implies-regularity} in $B_{2s}(z)$, provided $\omega(h_{\rm split})\ll s$. Hence, in $B_s(z)$, the free boundary is the union of one upper and one lower regular graph. Their vertical ordering identifies the two heights on overlaps, so the local graphs patch into functions $g^-<g^+$. The same balls cover the slab containing the whole free boundary, and therefore exclude any additional sheet. Consequently,
\[
\{\bar u>0\}\cap B_{5/2} =\Omega^+_{5/2}\mathbin{\dot\cup}\Omega^-_{5/2}, \qquad \Omega^+_{5/2}=\{x_n>g^+(x')\}\cap B_{5/2},\qquad \Omega^-_{5/2}=\{x_n<g^-(x')\}\cap B_{5/2},
\]
and the small $C^1$ norm, together with the buffer between $B_2$ and $B_{5/2}$, makes both restrictions to $B_2$ connected. The local graph estimates and the interpolation bound
\[
|\nabla g(x'_0)| \le C\bigl(s^{-1}\|g\|_{L^\infty}+s\|D^2g\|_{L^\infty}\bigr)
\]
give
\begin{equation}\label{eq:vee-splitting-rough-graphs}
\|g^+\|_{C^2(D_{5/2})}+\|g^-\|_{C^2(D_{5/2})} \le C\omega(h).
\end{equation}

At every point $(x',g^\pm(x'))$ of either graph, the one-sided hypothesis gives $|g^\pm(x')|\le h$. If $x\in\Omega^+_{5/2}$, the same lower bound and the Lipschitz estimate from the vertical boundary point give
\[
x_n-h\le\bar u(x)\le x_n+h.
\]
Similarly, $|\bar u+x_n|\le h$ in $\Omega^-_{5/2}$. Fix $r_0>0$ smaller than the buffer between $B_2$ and $B_{5/2}$. In the balls $B_{2r_0}$ centered on the sheet portions meeting $B_2$, extend only the corresponding local component by zero across its graph. This is a classical one-sheet solution, and after rescaling it is $C h/r_0$-close to the appropriate half-space. Applying \Cref{lem:epsilon-regularity} in these balls, and ordinary interior harmonic estimates away from the sheets, gives
\begin{equation}\label{eq:vee-splitting-linear-slab}
\|g^+\|_{C^2(D_{2})}+\|g^-\|_{C^2(D_{2})} \le Ch,
\end{equation}
and the stated estimates for both $\bar u-x_n$ and $\bar u+x_n$.
\end{proof}

The splitting estimate in \Cref{lem:vee-splitting-under-hessian-bound} also converts one-sided vee trapping into quantitative control of the intrinsic neck radius.
\begin{lem}[One-sided vee trapping controls the neck radius]
\label{lem:one-sided-vee-trapping-controls-neck-radius}
There is a dimensional constant $C$ with the following property. Let $\zz\in\cZ$, and suppose that for some $e\in\mathbb S^{n-1}$ and $h>0$,
\[
u\ge V_{\zz,e}-h \qquad\text{in }B_{4\rb(\zz)}(\zz).
\]
Then $\rb(\zz)\le Ch$.
\end{lem}

\begin{proof}
Set $\bar u(x):=\rb(\zz)^{-1}u(\zz+\rb(\zz)x)$. By \Cref{lem:local-comparability-of-neck-radii}, applied with $M=3$,
\[
\|D^2\bar u\|_{L^\infty(B_3\cap\{\bar u>0\})}\le C_0
\]
for a dimensional $C_0$. If $\rb(\zz)>Ah$, then, after a rotation,
\[
\bar u(x)\ge |x_n|-A^{-1}\qquad\text{in }B_4.
\]
Choose $A$ so large that (from \Cref{lem:vee-splitting-under-hessian-bound})
\[
A^{-1}\le h_{\rm split}(n,C_0), \qquad |B_1|^{1/n}C_{\rm split}(n,C_0)A^{-1}<\eta_0.
\]
Since $\bar u$ is classical and $0\in\FB(\bar u)$, \Cref{lem:vee-splitting-under-hessian-bound} gives
\[
\int_{B_1\cap\{\bar u>0\}}|D^2\bar u|^n \le |B_1|\left(\frac{C_{\rm split}}A\right)^n<\eta_0^n.
\]
This contradicts the definition of bad-radius, \Cref{defn:neck-centers}. Hence $\rb(\zz)\le Ah$.
\end{proof}

Applying \Cref{lem:one-sided-vee-trapping-controls-neck-radius} to the inherited vee approximation gives a uniform neck-radius bound on every active ball.
\begin{lem}
\label{lem:fb-tree-small-neck-radii-near-flat-active-balls}
There is a dimensional constant $C_{\rm nr}$ such that, after decreasing $\delta$ and increasing $k$, every active ball $B\in\cD^l$ and every $\zz\in\cZ\cap\lsup{8}{B}$ satisfy
\[
\rb(\zz)\le C_{\rm nr}h_B r_B .
\]
In particular, after increasing $k$,
\[
h_{B_{\rm root}}\le C\epk,\qquad \rb(\zz)\le C\epk\Rk \quad\text{for every }\zz\in\cZ\cap B_{8\Rk}(\zk).
\]
For each fixed $k$, the tree has only finitely many balls.
\end{lem}

\begin{proof}
\medskip
\noindent\emph{Step 1: The conditional estimate.}
Fix $\rho\in(0,1/16]$, let $B\in\cD^l$, and suppose that $\zz\in\cZ\cap\lsup{8}{B}$ satisfies $s:=\rb(\zz)\le\rho r_B$. By \Cref{cor:centered-vee-lower-slab},
\begin{equation}\label{eq:fb-tree-linear-one-sided-from-hB}
u\ge V_B-Ch_Br_B \qquad\text{in }\lsup{12}{B}.
\end{equation}
The slab conclusion of \Cref{cor:centered-vee-lower-slab} at $\zz$ gives
\[
|\ell_B(\zz)|\le Ch_Br_B .
\]
Since $B_{4s}(\zz)\subset\lsup{12}{B}$ and $\lvert V_B-V_{\zz,a_B}\rvert\le|\ell_B(\zz)|$, it follows that
\[
u\ge V_{\zz,a_B}-Ch_Br_B \qquad\text{in }B_{4s}(\zz).
\]
Thus \Cref{lem:one-sided-vee-trapping-controls-neck-radius} yields
\begin{equation}\label{eq:fb-tree-conditional-small-neck-radius}
\rb(\zz)\le C_{\rm nr}h_Br_B
\end{equation}
for a dimensional constant $C_{\rm nr}$.

\medskip
\noindent\emph{Step 2: Propagation down the tree.}
Decrease $\delta$ so that
\[
2C_{\rm nr}\delta\le\rho.
\]
We show that, after increasing $k$,
\[
\rb(\zz)\le C_{\rm nr}h_Br_B \quad\text{for every }B\in\cD^l \text{ and every }\zz\in\cZ\cap\lsup{8}{B}.
\]

First consider the root $B_{\rm root}=B_{\Rk}(\zk)$. By the root initialization in \Cref{defn:fb-tree-branching-and-terminal-balls}, after increasing $k$, $h_{B_{\rm root}}<\delta$.

Let $C_{32/\rho}$ be the constant in \Cref{lem:local-comparability-of-neck-radii} with $M=32/\rho$. Since $\rb(\zk)/\Rk\to0$, we may increase $k$ so that
\[
\rb(\zk)<\frac{\rho}{C_{32/\rho}}\Rk .
\]
If $\zz\in\cZ\cap B_{32\Rk}(\zk)$ and $\rb(\zz)\ge\rho\Rk$, then
\[
|\zz-\zk|\le32\Rk\le \frac{32}{\rho}\rb(\zz).
\]
Applying \Cref{lem:local-comparability-of-neck-radii} with center $\zz$, nearby center $\zk$, and $M=32/\rho$, gives
\[
\rb(\zk)\ge \frac{\rb(\zz)}{C_{32/\rho}} \ge\frac{\rho}{C_{32/\rho}}\Rk,
\]
a contradiction. Hence
\[
\rb(\zz)<\rho\Rk \qquad\text{for every }\zz\in\cZ\cap B_{32\Rk}(\zk).
\]
Step 1 applied to $B_{\rm root}$ gives the desired bound at generation $0$.

Assume now that the desired bound holds for every ball in generation $l$. Let $B'\in\cD^{l+1}$, and let $B:=PB'$. Since $x_{B'}\in\lsup{5/4}{B}$ and $r_{B'}=r_B/2$, we have
\[
\lsup{8}{B'}\subset\lsup{21/4}{B}\subset\lsup{8}{B}.
\]
Thus the induction hypothesis gives, for every $\zz\in\cZ\cap\lsup{8}{B'}$,
\[
\rb(\zz)\le C_{\rm nr}h_Br_B \le 2C_{\rm nr}\delta r_{B'} \le \rho r_{B'} .
\]
Here $B$ is branching, and hence $h_B<\delta$. Applying Step 1 to $B'$ proves the desired estimate at generation $l+1$.

Finally, the vee-excess term in $\bbE_{\zk}(16\Rk)$, the fixed-scale propagation \Cref{prop:fixed-direction-monneau-doubling}, and \eqref{eq:selected-monneau-at-refined-scale} give
\[
h_{B_{\rm root}}\le C\bbE_{\zk}(16\Rk) \le C\bbE_{\zk}(8\Rk) \le C\epk .
\]
Applying \eqref{eq:fb-tree-conditional-small-neck-radius} to $B_{\rm root}$ gives
\[
\rb(\zz)\le C\epk\Rk \qquad\text{for every }\zz\in\cZ\cap B_{8\Rk}(\zk).
\]
Finally, since $\rmin > 0$, the whole tree is finite.
\end{proof}

At a regular terminal ball, \Cref{defn:fb-tree-branching-and-terminal-balls,lem:global-hessian-decay} supply the Hessian control needed to resolve the vee into two signed sheets.
\begin{lem}[Regular terminal balls split into two signed sheets]
\label{lem:fb-tree-regular-terminal-balls-split-into-two-signed-sheets}
If $\delta$ is sufficiently small, there is a dimensional constant $C$ such that the following holds. Let $B\in\cR$. Then $\{u>0\}\cap\lsup{3/2}{B}$ is the disjoint union of two connected components $\lsup{3/2}{B^+}$ and $\lsup{3/2}{B^-}$, characterized by
\[
\{\ell_B>Ch_Br_B\}\cap\lsup{3/2}{B}\subset \lsup{3/2}{B^+}, \qquad \{\ell_B<-Ch_Br_B\}\cap\lsup{3/2}{B}\subset \lsup{3/2}{B^-}.
\]
Moreover the two sheets are $C h_B$-Lipschitz graphs over $\{\ell_B=0\}$ and
\[
|\nabla u-a_B|\le Ch_B\quad\text{in }\lsup{3/2}{B^+},\qquad |\nabla u+a_B|\le Ch_B\quad\text{in }\lsup{3/2}{B^-}.
\]
\end{lem}

\begin{proof}
Normalize so that $B=B_1$, $x_B=0$, and $\ell_B(x)=x_n$, and write $h=h_B$. Since $B$ is regular terminal, \Cref{lem:global-hessian-decay} gives a dimensional Hessian bound in $B_{7/2}\cap\{u>0\}$, while \Cref{cor:centered-vee-lower-slab} gives
\[
u\ge |x_n|-Ch\qquad\text{in }B_8.
\]
Moreover $h\le C_{\rm inh}\delta$ by \eqref{eq:fb-tree-active-child-flatness}. The conclusion is therefore exactly \Cref{lem:vee-splitting-under-hessian-bound}, scaled back to $B$.
\end{proof}

We next rule out holes along branching ridges and make the local orientations compatible across the tree.
\begin{lem}[No-hole and orientation compatibility]
\label{lem:fb-tree-orientation-compatibility}
After decreasing $\delta$ and increasing $k$, the following hold.
\begin{enumerate}[(i)]
\item For every branching ball $B\in\cB$,
\[
B\cap\{V_B=0\}\cap B_{3\Rk/4}(\zk) \subset \bigcup_{B'\in\operatorname{Ch}(B)}\lsup{1/2}{B'}.
\]
\item There is a dimensional constant $C$ such that, if $D,E\in\cD^l$ satisfy $\lsup{8}{D}\cap\lsup{8}{E}\ne\varnothing$, then
\begin{equation}\label{eq:signed-decomposition-same-generation-compatibility}
|a_D-a_E|\le C\delta, \qquad \|\ell_D-\ell_E\|_{L^\infty(\lsup{8}{D}\cap\lsup{8}{E})} \le C\delta r_l .
\end{equation}
\end{enumerate}
\end{lem}

\begin{proof}
We first prove the no-hole projection. Fix $B\in\cB^l$, set $\rho=2^{-6}$, and let $y\in B\cap\{V_B=0\}\cap B_{3\Rk/4}(\zk).$ If $B_{\rho r_B}(y)\cap\FB(u)\ne\varnothing$, choose $p\in B_{\rho r_B}(y)\cap\FB(u)$. Then $p\in B_{\rm root}$ and $p\in\lsup{9/8}{B}$. By \eqref{eq:fb-tree-active-cover-near-branching}, there is $B'\in\operatorname{Ch}(B)$ such that $|p-x_{B'}|\le\frac14r_{l+1}=\frac18r_B,$ and hence $y\in \lsup{1/2}{B'}$. Thus, if $y$ is not contained in the union in (i), then
\[
B_{\rho r_B}(y)\cap\FB(u)=\varnothing.
\]

Rescaling by $r_{B}$, translating by $y$, and rotating so that the vee becomes $|x_n|$, we have that $u$ is harmonic in $B_\rho(0)$, while $h_{B}<\delta$. Together with the Lipschitz bounds for $u$ and the vee, this is a contradiction, for $\delta$ small (depending on $\rho$).

It remains to prove \eqref{eq:signed-decomposition-same-generation-compatibility}. We argue by induction on the generation $l$. At $l=0$, $\cD^0=\{B_{\rm root}\}$, so the claim holds with $D=E=B_{\rm root}$. Assume the claim at level $l-1$, and let $D,E\in\cD^l$ satisfy $\lsup{8}{D}\cap\lsup{8}{E}\ne\varnothing$.

All active balls have $h\le C\delta$ by \Cref{lem:fb-tree-properties}. The overlap of $\lsup{16}{D}$ and $\lsup{16}{E}$ contains a ball of radius $cr_l$. Comparing $|\ell_D|$ and $|\ell_E|$ to $u$ on that ball and using \Cref{lem:unit-slope-vee-comparison} gives
\[
\min_{\sigma=\pm1} \left( |a_D-\sigma a_E| +r_l^{-1} \|\ell_D-\sigma\ell_E\|_{L^\infty(\lsup{8}{D}\cap\lsup{8}{E})} \right) \le C\delta .
\]

To rule out the negative sign, we observe that \eqref{eq:fb-tree-predecessor-affine-comparison}, gives $|a_D-a_E|\le C\delta,$ so choosing $\delta$ small enough we necessarily have $\sigma = +1$.
\end{proof}

This compatibility allows the local sheets to be assembled into globally signed pieces.
\begin{defn}[Free-boundary tree signed pieces]
\label{defn:fb-tree-signed-pieces}
Fix
\[
c_\ast:=\frac1{16}.
\]
For a branching ball $B\in\cB$, define the fixed signed cores
\begin{equation}\label{eq:fb-tree-branching-signed-pieces}
B^\pm:=B\cap\{\pm\ell_B>c_\ast r_B\}.
\end{equation}
For $B\in\cR$, define
\[
B^\pm:=\lsup{3/2}{B^\pm}\cap B,
\]
where $\lsup{3/2}{B^\pm}$ are the two regular sheets from \Cref{lem:fb-tree-regular-terminal-balls-split-into-two-signed-sheets}. The two signed domains are
\[
U_\pm:= B_{\Rk/2}(\zk)\cap \bigcup_{B\in\cB\cup\cR}B^\pm .
\]
Finally, define the terminal neck region
\[
\mathcal X_{\rm neck}:= \overline{B_{\Rk/2}(\zk) \cap \bigcup_{B\in\cN}B}.
\]
\end{defn}

\begin{lem}
\label{lem:fb-tree-signed-decomposition}
After decreasing $\delta$ and increasing $k$, we have
\begin{equation}\label{eq:fb-tree-opposite-sign-separation}
B^+\cap C^-=\varnothing \qquad\text{for all}\quad B,C\in\cB\cup\cR.
\end{equation}
In particular, $U_+\cap U_-=\varnothing$. Moreover, there is $c>0$ such that, for every non-root $B\in\cB\cup\cR$ and each $\sigma\in\{+,-\}$, there is a ball satisfying
\begin{equation}\label{eq:fb-tree-parent-child-overlap}
B_r(x_B^\sigma)\subset (PB)^\sigma\cap B^\sigma, \qquad r\ge c r_B.
\end{equation}
Finally,
\begin{equation}\label{eq:fb-tree-signed-cover}
\bigl(\{u>0\}\cap B_{\Rk/2}(\zk)\bigr)\setminus\mathcal X_{\rm neck} \subset U_+\cup U_- \subset \{u>0\}\cap B_{\Rk/2}(\zk).
\end{equation}
\end{lem}

\begin{proof}
Increase $k$ so that the root ball is branching, and decrease $\delta$ so that \Cref{lem:fb-tree-regular-terminal-balls-split-into-two-signed-sheets} applies whenever $h_B\le C_{\rm inh}\delta$, and so that the tree consequence of \Cref{cor:centered-vee-lower-slab} gives
\[
u\ge V_B-\frac{c_\ast}{4}r_B \qquad\text{in }\lsup{3/2}{B}
\]
for every branching ball $B\in\cB$. Then every branching signed core lies in $\{u>0\}$, and every regular terminal ball satisfies the hypotheses of \Cref{lem:fb-tree-regular-terminal-balls-split-into-two-signed-sheets}.

\medskip
\noindent\emph{Step 1.} Choose $\Lambda=2/c_\ast=2^5$, then a small dimensional $\eta>0$
with $\eta\Lambda\ll c_\ast$.

Let us first show that intersections occur at comparable scales, for signs $\sigma,\tau\in\{+,-\}$:
\begin{equation}\label{eq:fb-tree-intersecting-pieces-comparable-scales}
D^\sigma\cap E^\tau\ne\varnothing \quad\Longrightarrow\quad \Lambda^{-1}r_D\le r_E\le\Lambda r_D.
\end{equation}
Indeed, suppose $r_E\le\Lambda^{-1}r_D$ and choose $x\in D^\sigma\cap E^\tau$. If $D$ is branching, then $\sigma\ell_D(x)>c_\ast r_D$, while $x_E\in\FB(u)$ and \Cref{cor:centered-vee-lower-slab} give
\[
|\ell_D(x_E)|\le C h_D r_D\le C\delta r_D.
\]
Notice also that $|x-x_E|\le r_D/\Lambda$, so $|\ell_D(x)-\ell_D(x_E)|\le r_D/\Lambda$. Since $1/\Lambda=c_\ast/2$ and \eqref{eq:fb-tree-branching-signed-pieces} is strict, this gives $\sigma\ell_D(x_E)>c_\ast r_D/2$, a contradiction for $\delta$ small. If $D$ is regular terminal, the scale separation and the actual intersection give $\lsup{4}{PE}\subset\lsup{4}{D}$. Since $PE$ is branching, this contradicts $\lsup{4}{D}\cap\cZ=\varnothing$. Interchanging $D$ and $E$ gives the opposite inequality and proves \eqref{eq:fb-tree-intersecting-pieces-comparable-scales}.

By the choice of $\Lambda$, $D$ and $E$ differ by at most five generations. If, say, $E$ is deeper than $D$, let $A$ be its ancestor in the generation of $D$, which satisfies
\[
\lsup{8}{A}\cap\lsup{8}{D}\ne\varnothing.
\]
Apply \eqref{eq:signed-decomposition-same-generation-compatibility} to $A,D$ and sum \eqref{eq:fb-tree-predecessor-affine-comparison} along the edges from $A$ to $E$:
\[
|a_D-a_E| + \frac{ \|\ell_D-\ell_E\|_{L^\infty( \lsup{3/2}{D}\cap\lsup{3/2}{E})} }{ \min\{r_D,r_E\} } \le C\delta.
\]
Decreasing $\delta$ so that $C\delta\le\eta$ shows
\begin{equation}\label{eq:fb-tree-comparable-scale-affine-control}
\|\ell_D-\ell_E\|_{L^\infty(\lsup{3/2}{D}\cap\lsup{3/2}{E})} \le \eta\min\{r_D,r_E\}, \qquad |a_D-a_E|\le\eta.
\end{equation}

We also need compatibility with a regular terminal sheet. Let $D\in\cR$, $E\in\cB\cup\cR$, and suppose
\[
x\in D^\tau\cap E^\sigma
\]
for signs $\sigma,\tau\in\{+,-\}$. Then $\tau=\sigma$. Indeed, \eqref{eq:fb-tree-intersecting-pieces-comparable-scales} applies. If $E\in\cB$, then $\sigma\ell_E(x)>c_\ast r_E$, and \eqref{eq:fb-tree-comparable-scale-affine-control} gives
\[
\sigma\ell_D(x) \ge c_\ast r_E-\eta\min\{r_D,r_E\}.
\]
Since the scales are $\Lambda$-comparable and $\eta\Lambda\ll c_\ast$, the right-hand side is much larger than $h_Dr_D$ after decreasing $\delta$. \Cref{lem:fb-tree-regular-terminal-balls-split-into-two-signed-sheets} then puts $x$ in the $\sigma$-sheet of $D$, so $\tau=\sigma$. If $E\in\cR$, then \Cref{lem:fb-tree-regular-terminal-balls-split-into-two-signed-sheets} gives
\[
\nabla u(x)=\tau a_D+O(h_D)=\sigma a_E+O(h_E).
\]
Equation \eqref{eq:fb-tree-comparable-scale-affine-control} gives $a_D=a_E+O(\eta)$, and therefore $\tau=\sigma$ once $\delta$ and $\eta$ are small.

We next prove the parent--child overlap \eqref{eq:fb-tree-parent-child-overlap}. Let $B\in\cB\cup\cR$ be non-root and set $A:=PB$. Since $x_B\in\lsup{5/4}{A}$ and $r_A=2r_B$, \eqref{eq:fb-tree-predecessor-affine-comparison} gives
\[
|a_B-a_A|\le C\delta,\qquad \|\ell_B-\ell_A\|_{L^\infty(\lsup4B)} \le C\delta r_B,
\]
after decreasing $\delta$. Also $x_B\in\FB(u)$, so the slab conclusion of \Cref{cor:centered-vee-lower-slab} for the branching ball $A$ gives
\[
|\ell_A(x_B)|\le C\delta r_A.
\]

It remains to prove the following geometric inclusion. Scale by $r_B$, rotate so that $a_A=e_n$, and first set $\eta=\delta=0$. The parent and child balls become $B_2(0)$ and $B_1(\xi)$, with $|\xi|\le5/2$ and $\xi_n=0$. Since $5/2<2+1$, the two balls have a uniform interior intersection: for a dimensional $c_g>0$, depending only on the fixed value $c_\ast=1/16$,
\[
B_{c_g}(q)\subset B_2(0)\cap B_1(\xi) \cap\{x_n>4c_\ast\} \cap\{x_n-\xi_n>4c_\ast\}.
\]
Compactness of the family $|\xi|\le5/2$, $\xi_n=0$, gives this uniformly. For $\eta$ and $\delta$ small, the same inclusion persists after returning to the actual affine functions. Thus, for each sign $\sigma$, there is a ball of radius $c_g r_B$ contained in
\[
\lsup{1}{A}\cap\lsup{1}{B} \cap\{\sigma\ell_A>2c_\ast r_A\} \cap\{\sigma\ell_B>4c_\ast r_B\}.
\]
If $B$ is branching, this ball is contained in $A^\sigma\cap B^\sigma$. If $B$ is regular terminal, the strict inequality $\sigma\ell_B>4c_\ast r_B$ and $h_B\le C_{\rm inh}\delta$ put the ball in the $\sigma$-regular sheet by \Cref{lem:fb-tree-regular-terminal-balls-split-into-two-signed-sheets}. Thus, taking $x_B^\sigma$ to be its center and $c=c_g$, we obtain
\[
B_{c_g r_B}(x_B^\sigma) \subset A^\sigma\cap B^\sigma =(PB)^\sigma\cap B^\sigma,
\]
which proves \eqref{eq:fb-tree-parent-child-overlap}.

We can now prove \eqref{eq:fb-tree-opposite-sign-separation}. Suppose that $x\in B^+\cap C^-$. If $B$ and $C$ are both branching, \eqref{eq:fb-tree-comparable-scale-affine-control} gives
\[
\ell_C(x) \ge \ell_B(x)-\eta\min\{r_B,r_C\}>0,
\]
contradicting $x\in C^-$. If $B$ or $C$ is regular, the actual-overlap regular compatibility proved above forces the two signs to agree, again a contradiction. This proves \eqref{eq:fb-tree-opposite-sign-separation}, and hence $U_+\cap U_-=\varnothing$.

\medskip
\noindent\emph{Step 2.}
It remains to prove \eqref{eq:fb-tree-signed-cover}. First observe that if
\[
x\in B\cap B_{\Rk/2}(\zk)\cap\{|\ell_B|\le c_\ast r_B\}
\]
for a branching ball $B\in\cB^m$, then the orthogonal projection $y=x-\ell_B(x)a_B$ lies in $B\cap\{V_B=0\}\cap B_{3\Rk/4}(\zk)$. By \Cref{lem:fb-tree-orientation-compatibility}(i), $y\in\lsup{1/2}{B'}$ for some $B'\in\operatorname{Ch}(B)$, and
\[
|x-x_{B'}| \le c_\ast r_B+\frac12r_{B'} =\frac1{16}r_B+\frac14r_B<r_{B'}.
\]
Thus every point in the fixed slab of a branching ball belongs to a local child.

An induction on $m$ gives
\begin{align}\label{eq:fb-tree-finite-level-cover}
\{u>0\}\cap B_{\Rk/2}(\zk) \subset B_{\Rk/2}(\zk)\cap\bigg[{} &\bigcup_{\substack{0\le j<m\\D\in\cB^j}}(D^+\cup D^-)
\cup\bigcup_{\substack{0\le j\le m\\D\in\cR^j}}(D^+\cup D^-)
\cup\bigcup_{\substack{0\le j\le m\\D\in\cN^j}}D
\cup\bigcup_{D\in\cB^m}D \bigg].
\end{align}
For $m=0$, the last union contains the branching root ball. Suppose the inclusion holds at level $m$ and consider a point in a ball $B\in\cB^m$. If $|\ell_B|>c_\ast r_B$, the point lies in $B^+\cup B^-$. Otherwise the projection argument above places it in a child of $B$, which is respectively recorded at level $m+1$ as a branching ball, a neck-like terminal ball, or one of the two regular pieces. This proves \eqref{eq:fb-tree-finite-level-cover}.

Now let
\[
x\in\bigl(\{u>0\}\cap B_{\Rk/2}(\zk)\bigr)\setminus\mathcal X_{\rm neck}.
\]
If $x\notin U_+\cup U_-$, then $x\notin\mathcal X_{\rm neck}$ removes the terminal-neck terms from \eqref{eq:fb-tree-finite-level-cover}; hence the same equation puts $x$ in a branching ball $B_m\in\cB^m$ for every $m$. But $x\in\{u>0\}$, so
\[
d:=\dist(x,\FB(u))>0.
\]
Every $B_m$ is centered on $\FB(u)$ and has radius $r_m\to0$, which is impossible once $r_m<d/2$. Therefore $x\in U_+\cup U_-$. The reverse inclusion in \eqref{eq:fb-tree-signed-cover} follows because every raw signed piece lies in $\{u>0\}$, while \Cref{defn:fb-tree-signed-pieces} supplies the cutoff $B_{\Rk/2}(\zk)$. This proves \eqref{eq:fb-tree-signed-cover}.
\end{proof}

\subsection{Neck scales and localized John geometry}

We next relate the terminal balls of the tree to the intrinsic neck scale and use this relation to obtain localized John geometry for the signed domains.

\begin{lem}[Recentering the vee excess]
\label{lem:fb-tree-recenter-neck-monneau}
There is a dimensional constant $C$ such that, for every $L\ge32$, $\zz\in\cZ$, $r>0$, $x\in\FB(u)$ with $|x-\zz|\le Lr/4$, and $e\in\mathbb S^{n-1}$,
\begin{equation}\label{eq:fb-tree-recenter-neck-monneau-fixed}
\frac1{r^{n+2}}\int_{B_{16r}(x)} \bigl|u(y)-|e\cdot(y-x)|\bigr|^2\,dy \le C L^{n+2}\bbE_{\zz}^2(Lr;e).
\end{equation}
\end{lem}

\begin{proof}
Put $R=Lr$ and $E:=\bbE_\zz(R;e)$. The distance assumption gives $|x-\zz|\le R/4$, and $L\ge32$ gives
\[
B_{16r}(x)\subset B_{3R/4}(\zz).
\]
The vee-excess term in $E$ gives
\begin{equation}\label{eq:recenter-bulk}
R^{-n-2}\int_{B_R(\zz)} |u-V_{\zz,e}|^2\,dy\le C E^2.
\end{equation}

If $E\ge c$, it follows from $u(x)=0$ and $|\nabla u|\le1$, after increasing $C$. We may therefore assume $0<E<c$, with $c$ small. By Cauchy--Schwarz,
\[
\frac1{R/16}\fint_{B_R(\zz)} |u-V_{\zz,e}|\,dy \le C E<1.
\]
\Cref{cor:centered-vee-lower-slab}, applied at $\zz\in\cZ\subset\FB(u)$ with radius $R/16$, yields $|e\cdot(x-\zz)|=V_{\zz,e}(x)\le C E R.$ For $y\in B_{16r}(x)$,
\[
\bigl|V_{\zz,e}(y)-|e\cdot(y-x)|\bigr| \le |e\cdot(x-\zz)| \le C E R.
\]
Combining this with \eqref{eq:recenter-bulk}, yields the desired result.
\end{proof}

Neck-like balls have radii comparable to the threshold radii of bad centers around them:

\begin{lem}
\label{lem:fb-tree-nonfinal-neck-terminals-neck-scale}
Let $B\in\cN$ and $\zz\in\cZ\cap\lsup{4}{B}$. After decreasing $\delta$, we have
\[
C_\delta^{-1}\rb(\zz)\le r_B\le C_\delta\rb(\zz),
\]
for some constant depending on $\delta$. In particular, if $\zz\in B_{8\Rk}(\zk)$ as well, then
\[
r_B\le C_\delta\epk\Rk .
\]
\end{lem}

\begin{proof}
Let $B\in\cN$ and $\zz\in\cZ\cap\lsup{4}{B}$. The lower bound follows from \eqref{eq:fb-tree-active-child-flatness} and \Cref{lem:fb-tree-small-neck-radii-near-flat-active-balls}.

For the upper bound, apply \Cref{lem:fb-tree-recenter-neck-monneau} with $L=32$; this is admissible because $|x_B-\zz|\le4r_B$. Choose $\varepsilon>0$ so small that this recentering estimate gives $h_B<\delta$ whenever $\bbE_\zz(32r_B)\le\varepsilon$. By \Cref{lem:decay-of-monneau-excess}, this happens whenever $r_B\ge M_\delta\rb(\zz)$, with $M_\delta$ large enough. Since $\lsup{4}{B}\cap\cZ\ne\varnothing$, \eqref{eq:fb-tree-branching-condition} would then hold, contradicting $B\in\cN$. Thus $r_B\le C_\delta\rb(\zz)$.

Finally, if $\zz\in B_{8\Rk}(\zk)$, \Cref{lem:fb-tree-small-neck-radii-near-flat-active-balls} gives $\rb(\zz)\le C\epk\Rk$, and the upper comparison gives $r_B\le C_\delta\epk\Rk$.
\end{proof}

To formulate the corresponding connectivity property of the signed domains, we recall the notion of a John domain.
\begin{defn}[$c$-John domains]
\label{defn:c-john-domains}
Let $0<c<1$. A bounded domain $\Omega\subset\mathbb R^n$ is a $c$-John domain with center $x_\Omega\in\Omega$ if, for every $x\in\Omega$, there is a rectifiable curve $\gamma\colon[0,L]\to\Omega$, parametrized by arclength, such that $\gamma(0)=x$, $\gamma(L)=x_\Omega$, and
\[
\dist(\gamma(t),\partial\Omega)\ge c\,t \qquad\text{for every }0\le t\le L .
\]
We say that $\Omega$ is a $c$-John domain if such a center exists.
\end{defn}

\begin{lem}
\label{lem:fb-tree-localized-raw-signed-shadows-john}
After decreasing $\delta$ and increasing $k$, there is a dimensional constant $c_J>0$ with the following property. Let
\[
\zz\in B_{\Rk/8}(\zk)\cap\cZ, \qquad \epk\Rk\le c_JR, \qquad R\le\Rk/16 .
\]
For each $\ast\in\{+,-\}$, there is a $c_J$-John domain $\Omega_\ast(\zz,R)$ such that
\[
U_\ast\cap B_{R/2}(\zz) \subset \Omega_\ast(\zz,R)\subset U_\ast\cap B_{3R/4}(\zz).
\]
\[
|\Omega_\ast(\zz,R)|\ge c_JR^n.
\]
More generally, the same conclusion holds for any $\zz\in B_{\Rk/8}(\zk)\cap\cZ$ and $R\le\Rk/16$ such that every terminal neck ball $N\in\cN$ meeting $B_{3R/4}(\zz)$ satisfies $r_N\le c_JR .$
\end{lem}

\begin{proof}
We prove it for $U_+$; the proof for $U_-$ is identical. For $D\in\cB\cup\cR$, set
\[
p_D:=x_D+\frac12r_Da_D .
\]
After decreasing $\delta$, there is $c>0$ such that, for every $x\in D^+$,
\begin{equation}\label{eq:localized-john-piece-cone}
B_{ctr_D}\bigl((1-t)x+tp_D\bigr)\subset D^+, \qquad 0<t\le1.
\end{equation}
If $D$ is branching, this is immediate from \eqref{eq:fb-tree-branching-signed-pieces} and convexity. If $D$ is regular, it follows from \Cref{lem:fb-tree-regular-terminal-balls-split-into-two-signed-sheets}.

For every non-root $D\in\cB\cup\cR$, fix a positive gate
\begin{equation}\label{eq:localized-john-positive-gate}
G_D:=B_{c_gr_D}(g_D) \subset D^+\cap(PD)^+
\end{equation}
given by \eqref{eq:fb-tree-parent-child-overlap}. The same half-ball geometry joins $p_D$ to the center $g_D$ with a fixed tube inside $D^+$. Consequently, every $x\in D^+$ can be joined to $g_D$ by a curve of length at most $Cr_D$ such that
\begin{equation}\label{eq:localized-john-route-to-gate}
\dist(\gamma(t),\partial D^+) \ge c\min\{t,r_D\}.
\end{equation}
For $D\in \mathcal{B}$ this is simply the interpolation of balls inside the convex set $D^+$; for a regular piece it is \eqref{eq:localized-john-piece-cone}.

We next obtain a common region in which to join the gates. Under the first hypothesis, \Cref{lem:fb-tree-small-neck-radii-near-flat-active-balls} gives
\[
\rb(\zz)\le C\epk\Rk\le Cc_JR.
\]
Under the general hypothesis, iterate the local-child cover \eqref{eq:fb-tree-active-cover-near-branching} through $\zz$. It ends at a neck-like terminal ball $N$ containing $\zz$, and hence \Cref{lem:fb-tree-nonfinal-neck-terminals-neck-scale} gives
\[
\rb(\zz)\le C_\delta r_N\le C_\delta c_JR.
\]
Thus, after choosing $c_J$ small, \Cref{lem:decay-of-monneau-excess,lem:fb-tree-mesoscopic-signed-slab-flatness} give an orientation $e\in\mathbb S^{n-1}$ such that
\begin{equation}\label{eq:localized-john-common-positive-cap}
\begin{aligned}
H^+ &:=B_{3R/4}(\zz)\cap \{e\cdot(x-\zz)>2\varepsilon R\} \subset U_+\cap B_{3R/4}(\zz) \subset\{e\cdot(x-\zz)>-\varepsilon R\},
\end{aligned}
\end{equation}
where $\varepsilon>0$ is a sufficiently small dimensional constant.

Let $x\in U_+\cap B_{R/2}(\zz)$, and choose $D\in\cD^l\cap(\cB\cup\cR)$ with $x\in D^+$. Choose $m$ so that, recalling $r_m = 2^{-m}\Rk$,
\begin{equation}\label{eq:localized-john-stopping-generation}
\frac{R}{16}<r_m\le\frac{R}{8}.
\end{equation}
Suppose first that $l\ge m+1$, and write
\[
D_j:=P^{\,l-j}D,\qquad m\le j\le l.
\]
Starting with \eqref{eq:localized-john-route-to-gate}, pass successively through the gates $G_{D_l},G_{D_{l-1}},\ldots,G_{D_{m+1}},$ interpolating the incoming and outgoing gate balls inside each intermediate branching piece. At a connector of radius $r$, the length already accumulated is at most
\[
C\sum_{i=0}^{\infty}2^{-i}r\le Cr,
\]
whereas the interpolated balls have radius at least $cr$. Hence these balls, together with a sufficiently small ball about $x$, form a $c$-John domain $\mathcal O_x$, from $x$ to the top gate $G_{D_{m+1}}$.

The predecessor localization in \Cref{defn:fb-tree-branching-and-terminal-balls}, the geometric decay of the radii, and the stopping rule \eqref{eq:localized-john-stopping-generation} show that all the gates and interpolating balls in this construction lie in $B_{R/4}(x)$. Therefore
\begin{equation}\label{eq:localized-john-chain-domain}
x\in\mathcal O_x \subset U_+\cap B_{3R/4}(\zz), \qquad B_{\kappa R}(q_x)\subset\mathcal O_x,
\end{equation}
where $q_x=g_{D_{m+1}}$ and $\kappa>0$ is dimensional. The distinct top gates occurring here are uniformly bounded in number: their underlying balls have radius $r_{m+1}\simeq R$, lie in a fixed multiple of $B_R(\zz)$, and have the separation stated in \Cref{defn:fb-tree-dyadic-balls-centered-vee-excess}.

If $l\le m$, truncate \eqref{eq:localized-john-piece-cone} when $tr_D=r_{m+1}$. The interpolated balls, together with a sufficiently small ball about $x$, lie in $B_{R/4}(x)$ and form a $c$-John domain $\mathcal O_x$ satisfying \eqref{eq:localized-john-chain-domain}, with
\[
q_x=(1-r_{m+1}/r_D)x+(r_{m+1}/r_D)p_D .
\]
Thus \eqref{eq:localized-john-chain-domain} holds for every $x\in U_+\cap B_{R/2}(\zz)$.

The second inclusion in \eqref{eq:localized-john-common-positive-cap}, applied to the last ball in \eqref{eq:localized-john-chain-domain}, shows, after reducing $\kappa$, that
\begin{equation}\label{eq:localized-john-gates-in-common-cap}
B_{\kappa R}(q_x)\subset\mathcal O_x\cap H^+ \qquad\text{for every }x\in U_+\cap B_{R/2}(\zz).
\end{equation}
This is precisely the flat picture: in the limit $\delta=0$ all the endpoint balls lie with fixed room in the same positive half-ball, and the preceding estimates preserve that room.

The set $H^+$ is convex and contains a ball of radius $cR$ centered at $q_H:=\zz+(R/2)e$. In particular, the balls in \eqref{eq:localized-john-gates-in-common-cap} can be interpolated inside $H^+$ with the ball about $q_H$. Hence
\[
\Omega_+(\zz,R) := H^+\cup \bigcup_{x\in U_+\cap B_{R/2}(\zz)}\mathcal O_x
\]
is a uniform John domain with center $q_H$: first follow the John curve in $\mathcal O_x$, and then the uniformly thick segment in $H^+$. By construction,
\[
U_+\cap B_{R/2}(\zz) \subset\Omega_+(\zz,R) \subset U_+\cap B_{3R/4}(\zz).
\]
Finally, the ball about $q_H$ gives
\[
|\Omega_+(\zz,R)|\ge cR^n.
\]
Decreasing $c_J$ once more proves the assertion.
\end{proof}

\section{Quantitative consequences of the signed tree}
\label{sec:fb-tree-quantitative}

Throughout this section, the selected configuration from \Cref{sec:neck-radii-from-symmetric-excess} and the signed tree from \Cref{sec:fb-tree} remain fixed. In particular, the constant $\delta> 0$ is from now on fixed, according to the choices in \Cref{sec:fb-tree}.

We derive three estimates used in the linearization: mesoscopic signed geometry, counting and terminal-neck bounds, and a higher-power estimate for the boundary flux.

\subsection{Mesoscopic signed geometry}

We first show that, above the neck scale, the two signed domains are trapped in oppositely oriented slabs.

\begin{lem}
\label{lem:fb-tree-mesoscopic-signed-slab-flatness}
There exists $C\ge1$ such that, after increasing $k$, the following holds.

Let $\zz\in\cZ$, let $R>0$, and assume
\[
B_R(\zz)\subset B_{\Rk/2}(\zk), \qquad R\ge C\rb(\zz).
\]
Then, for every $e\in\mathbb S^{n-1}$, after possibly replacing $e$ by $-e$,
\begin{equation}\label{eq:fb-tree-mesoscopic-signed-slab-flatness}
\begin{aligned}
B_R(\zz)\cap \{\pm e\cdot(x-\zz)>C\bbE_\zz(8R;e)R\} \subset U_\pm\cap B_R(\zz) &\subset \{\pm e\cdot(x-\zz)>-C\bbE_\zz(8R;e)R\}.
\end{aligned}
\end{equation}
In particular, the scale-separation hypothesis holds whenever $R\ge C\epk\Rk$. At the root scale,
\begin{equation}\label{eq:direct-sided-root-slab}
\begin{aligned}
B_{\Rk/8}(\zk)\cap \{\pm e_{\rm root}\cdot(x-\zk)>C\epk\Rk\} &\subset U_\pm\cap B_{\Rk/8}(\zk) \subset \{\pm e_{\rm root}\cdot(x-\zk)>-C\epk\Rk\}.
\end{aligned}
\end{equation}
\end{lem}

\begin{proof}
Put $E:=\bbE_\zz(8R;e)$. We may assume $E\le\tau$, with $\tau>0$ dimensional and small. Indeed, if $E>\tau$, choosing $C>\tau^{-1}$ makes the two inner caps empty and the two outer slabs contain $B_R(\zz)$, so the conclusion follows. By \Cref{prop:fixed-direction-monneau-doubling,cor:centered-vee-lower-slab}, with $C$ dimensional and allowed to change from line to line,
\begin{equation}\label{eq:fb-tree-mesoscopic-vee-slab-and-caps}
\begin{gathered}
u\ge V_{\zz,e}-CER\quad\text{in }B_{12R}(\zz),\qquad \FB(u)\cap B_{12R}(\zz) \subset\{|e\cdot(x-\zz)|\le CER\},\\
H^\pm:=B_R(\zz)\cap\{\pm e\cdot(x-\zz)>CER\} \subset\{u>0\}.
\end{gathered}
\end{equation}

By \eqref{eq:fb-tree-mesoscopic-vee-slab-and-caps} and \Cref{lem:upper-bound-for-nearby-neck-radii,lem:one-sided-vee-trapping-controls-neck-radius,lem:fb-tree-nonfinal-neck-terminals-neck-scale}, every terminal neck ball $N$ meeting $B_R(\zz)$ satisfies $r_N\le CER$. Moreover,
\begin{equation}\label{eq:fb-tree-mesoscopic-neck-slab}
\mathcal X_{\rm neck}\cap B_R(\zz) \subset\{|e\cdot(x-\zz)|\le CER\}.
\end{equation}

Iterating \eqref{eq:fb-tree-active-cover-near-branching} through $\zz$ gives a local-child chain, with the local children defined in \eqref{eq:fb-tree-local-children}, ending in a neck-like terminal ball containing $\zz$. By the preceding radius estimate this terminal ball has radius at most $CER\ll R$; hence the chain contains a branching ball $B$ such that
\[
\zz\in\lsup{1/4}{B},\qquad R/8<r_B\le R/4.
\]
Since $r_B\simeq R$, \eqref{eq:fb-tree-mesoscopic-vee-slab-and-caps}, \Cref{lem:fb-tree-recenter-neck-monneau,lem:unit-slope-vee-comparison}, \eqref{eq:fb-tree-centered-vee-excess}, and \eqref{eq:fb-tree-branching-signed-pieces} give, after possibly replacing $e$ by $-e$,
\begin{equation}\label{eq:fb-tree-mesoscopic-signed-cap-anchors}
B^\pm\cap H^\pm\ne\varnothing.
\end{equation}

The caps $H^\pm$ are connected and, by \eqref{eq:fb-tree-mesoscopic-vee-slab-and-caps} and \eqref{eq:fb-tree-mesoscopic-neck-slab}, lie in $\{u>0\}\setminus\mathcal X_{\rm neck}$. Thus \Cref{lem:fb-tree-signed-decomposition}, together with \eqref{eq:fb-tree-mesoscopic-signed-cap-anchors}, gives $H^\pm\subset U_\pm$. The reverse slab inclusion follows from the disjointness of $U_+$ and $U_-$: a point of $U_\pm\cap B_R(\zz)$ on the opposite side would belong to $H^\mp\subset U_\mp$. Enlarging $C$ proves \eqref{eq:fb-tree-mesoscopic-signed-slab-flatness}.

The scale-separation assertion follows directly from \Cref{lem:fb-tree-small-neck-radii-near-flat-active-balls}; the root assertion follows from the singleton case of \eqref{keyeqn-alpha} and the root orientation.
\end{proof}

For $j\ge0$, set
\begin{equation}\label{eq:fb-tree-inner-family}
\cD_{\rm in}^j:= \{B\in\cD^j:\lsup4B\cap B_{\Rk/8}(\zk)\ne\varnothing\}.
\end{equation}

\subsection{Counting and terminal-neck estimates}

We combine the tree geometry with \Cref{lem:fb-tree-mesoscopic-signed-slab-flatness} to derive the counting, descendant, and terminal-neck estimates needed for the flux argument below.

\begin{lem}
\label{lem:fb-tree-inner-counting-package}
After increasing $k$, for every $j\ge0$,
\begin{equation}\label{eq:fb-tree-higher-power-inputs}
\#\cD_{\rm in}^j\le C2^{j(n-2+2\beta_\circ)}, \qquad \sum_{A\in\cD_{\rm in}^j}h_A^{2/\alpha} \le C_\alpha2^{j(n-2)}\epk^{2/\alpha}.
\end{equation}
For the terminal subfamily one also has
\begin{equation}\label{eq:fb-tree-inner-terminal-count}
\#(\cN^j\cap\cD_{\rm in}^j) \le C_\alpha2^{j(n-2)} \min\{\epk^2,\epk^{2/\alpha}\}.
\end{equation}
Moreover, if $0\le m\le l$ and $A\in\cD_{\rm in}^m$, then
\begin{equation}\label{eq:fb-tree-surface-descendant-input}
\#\{B\in\cD_{\rm in}^l:P^{\,l-m}B=A\} \le C\min\{2^{(l-m)(n-1)},2^{l(n-2+2\beta_\circ)}\}.
\end{equation}
\end{lem}

\begin{proof}
After fixing $\delta$, choose $K$ sufficiently large. We begin with the construction needed when $K2^{-j}<1$. For $A\in\cD_{\rm in}^j$, choose
\[
\zz_A\in
\begin{cases}
\lsup4A\cap\cZ,&A\in\cB^j\cup\cN^j,\\
\lsup4{PA}\cap\cZ,&A\in\cR^j.
\end{cases}
\]
In the branching case, \Cref{lem:fb-tree-small-neck-radii-near-flat-active-balls} gives $\rb(\zz_A)\le C h_A r_j\le C\delta r_j$. In the neck-like case, \Cref{lem:fb-tree-nonfinal-neck-terminals-neck-scale} gives $\rb(\zz_A)\le C_\delta r_j$. In the regular case, $PA$ is branching and \Cref{lem:fb-tree-small-neck-radii-near-flat-active-balls}, applied at $PA$, gives the same bound. Since $\lsup4A\cap B_{\Rk/8}(\zk)\ne\varnothing$, and, when $A\in\cR^j$,
\[
r_{PA}=2r_j, \qquad |x_A-x_{PA}|\le\frac54r_{PA}=\frac52r_j,
\]
we have
\[
|\zz_A-\zk| \le
\begin{cases}
\Rk/8+8r_j,&A\in\cB^j\cup\cN^j,\\[2mm]
\Rk/8+\dfrac{29}{2}r_j,&A\in\cR^j.
\end{cases}
\]
In particular,
\[
|\zz_A-\zk|\le\Rk/8+15r_j \qquad\text{in all three cases}.
\]
Thus, after fixing $K$ large enough, in all three cases (using \Cref{lem:fb-tree-recenter-neck-monneau} in the last estimate):
\begin{equation}\label{eq:fb-tree-centered-neck-charge}
\zz_A\in\cZ_{Kr_j}\cap B_{\Rk}(\zk), \qquad |x_A-\zz_A|\le Kr_j, \qquad h_A^2\le C\,\bbE_{\zz_A}^2(8K r_j).
\end{equation}

From \Cref{defn:fb-tree-dyadic-balls-centered-vee-excess}, $\{x_A:A\in\cD_{\rm in}^j\}$ is a subfamily of the maximal $r_j/4$-separated free-boundary net. If
\[
B_{K r_j}(\zz_A)\cap B_{K r_j}(\zz_{A'})\ne\varnothing,
\]
then $|x_A-x_{A'}|\le C_K r_j$. The $r_j/4$-separation implies that each ball intersects at most $C(n,K)$ others. In particular, there are at most $N=N(n,K)$ subfamilies such that in each subfamily the balls $B_{K r_j}(\zz_A)$ are pairwise disjoint.

If $K2^{-j}\ge1$, there are only $O_K(1)$ such coarse levels and \eqref{eq:fb-tree-surface-cardinality} is absorbed into $C2^{j(n-2+2\beta_\circ)}$. Otherwise set $\zeta=K2^{-j}$. On each disjoint subfamily $\mathcal F$ of the centers assigned above,
\[
\#\mathcal F\,|B_{\zeta\Rk}| \le \left|\bigcup_{\zz\in\cZ_{\zeta\Rk}\cap B_{\Rk}(\zk)} B_{\zeta\Rk}(\zz)\right|.
\]
By \eqref{strongalternativeA},
\[
\#\mathcal F\le C\,\bN(\zeta,B_{\Rk}(\zk)) \le C\zeta^{2-n-2\beta_\circ} \le C2^{j(n-2+2\beta_\circ)}.
\]
Summing over the bounded number of subfamilies proves the first bound in \eqref{eq:fb-tree-higher-power-inputs}.

For the second bound, fix $j$. If $K2^{-j}<1/2$, set $\zeta=K2^{-j}$. On each disjoint subfamily, \eqref{keyeqn-alpha} gives
\[
\sum_A \bbE_{\zz_A}(8K r_j)^{2/\alpha} \le C_\alpha\zeta^{2-n}\epk^{2/\alpha} \le C_\alpha2^{j(n-2)}\epk^{2/\alpha}.
\]
Using \eqref{eq:fb-tree-centered-neck-charge} and summing over the bounded number of subfamilies yields the desired estimate.

If $K2^{-j}\ge1/2$ instead, there are only $O_K(1)$ balls in $\cD_{\rm in}^j$. Using the root center $\zk$, the inner condition gives
\[
|x_A-\zk|\le \Rk/8+4r_j\le (K/4+4)r_j .
\]
Since $\Rk/(2K)\le r_j\le\Rk$, \eqref{eq:selected-monneau-at-refined-scale} and \Cref{lem:fb-tree-recenter-neck-monneau} with $L=64K$, give $h_A\le C_K\epk$, from which \eqref{eq:fb-tree-higher-power-inputs} follows.

If $N\in\cN^j\cap\cD_{\rm in}^j$, then $N$ is neck-like but terminal, so \eqref{eq:fb-tree-branching-condition} gives $h_N\ge\delta$. The second estimate in \eqref{eq:fb-tree-higher-power-inputs} gives
\[
\#(\cN^j\cap\cD_{\rm in}^j) \le C_\alpha2^{j(n-2)}\epk^{2/\alpha}.
\]
If $K2^{-j}\ge1/2$, the preceding estimate $h_A\le C_K\epk<\delta$ shows that this terminal family is empty. Otherwise, for the centers $\zz_N$ chosen above, \Cref{lem:fb-tree-nonfinal-neck-terminals-neck-scale} and \eqref{eq:fb-tree-centered-neck-charge} give $\zz_N\in B_{\Rk}(\zk)$ and $\rb(\zz_N)\simeq_\delta r_j$. By \Cref{defn:fb-tree-dyadic-balls-centered-vee-excess}, the centers $x_N$ are $r_j/4$-separated; since $|x_N-\zz_N|\le4r_j$, the map $N\mapsto\zz_N$ has multiplicity at most $C$. Grouping the centers into the $O_\delta(1)$ dyadic neck-radius slices $\eta\simeq_\delta2^{-j}$, \eqref{eq:countingeta} gives
\[
\#(\cN^j\cap\cD_{\rm in}^j) \le C_\delta2^{j(n-2)}\epk^2.
\]
Taking the minimum proves \eqref{eq:fb-tree-inner-terminal-count}.

Finally we prove the descendant count. By \eqref{eq:fb-tree-inner-family}, the family is predecessor-closed: if $B\in\cD_{\rm in}^{j+1}$, then
\[
\dist\bigl(x_{PB},B_{\Rk/8}(\zk)\bigr) \le \frac54r_{PB}+4r_B =\frac{13}{4}r_{PB}<4r_{PB}.
\]
Hence $PB\in\cD_{\rm in}^j$. Iterating, if $0\le m\le l$ and $B\in\cD_{\rm in}^l$, then $P^{\,l-m}B\in\cD_{\rm in}^m$. The number of such balls is at most $\#\cD_{\rm in}^l$, since they all belong to that family. For the other bound, fix $A\in\cD_{\rm in}^m$ and let
\[
\mathcal F_{A,l}:=\{B\in\cD_{\rm in}^l:P^{\,l-m}B=A\}.
\]
If $l>m$, the predecessor geometry in \Cref{defn:fb-tree-branching-and-terminal-balls} gives $|x_B-x_A|\le \frac54\sum_{j=m}^{l-1}r_j\le \frac52 r_m.$ Thus all balls $\lsup{1/16}{B}$, $B\in\mathcal F_{A,l}$, are contained in a fixed dimensional dilation of $A$. They are pairwise disjoint and their centers lie on $\FB(u)$. The density estimate \eqref{eq:density-B} gives
\[
\cH^{n-1}\bigl(\FB(u)\cap\lsup{1/16}{B}\bigr)\ge c r_l^{n-1}.
\]
Summing over the disjoint family and applying \Cref{lem:perimeter-bound} in the fixed dilation of $A$, we get
\[
\#\mathcal F_{A,l}\,c r_l^{n-1} \le \cH^{n-1}\bigl(\FB(u)\cap B_{3r_m}(x_A)\bigr) \le C r_m^{n-1}.
\]
This proves \eqref{eq:fb-tree-surface-descendant-input}.
\end{proof}

We now define the set of neck balls $\mathcal{N}_{\rm in}$ as
\begin{equation}\label{eq:fb-tree-inner-neck-family}
\cN_{\rm in}:= \{B\in\cN:\lsup4B\cap B_{\Rk/8}(\zk)\ne\varnothing\}.
\end{equation}

Observe, also, that $U_\pm$ has finite perimeter in $B_{\Rk/8}(\zk)$. Indeed, since the tree is finite by \Cref{lem:fb-tree-small-neck-radii-near-flat-active-balls}, \Cref{lem:perimeter-bound} controls the free-boundary part, while the remaining boundary is contained in the finitely many spherical and affine cutoff faces from \Cref{defn:fb-tree-signed-pieces}; hence $\Per(U_\pm;B_{\Rk/8}(\zk))<\infty$. Write
\begin{equation}\label{eq:fb-tree-signed-boundary-pieces}
\Gamma_\pm:=\partial^*U_\pm\cap B_{\Rk/8}(\zk), \qquad \Gamma_\pm^{\rm reg}:=\Gamma_\pm\cap \bigl(\FB(u)\setminus\mathcal X_{\rm neck}\bigr), \qquad \Gamma_\pm^{\rm neck}:=\Gamma_\pm\setminus\Gamma_\pm^{\rm reg}.
\end{equation}

Set
\begin{equation}\label{eq:gamma-alpha}
\gamma_\alpha:=1+\frac2\alpha.
\end{equation}

We next prove the following useful bounds on the size of the neck set:
\begin{lem}
\label{lem:fb-tree-terminal-neck-bookkeeping}
After increasing $k$, we have
\begin{equation}\label{eq:fb-tree-terminal-neck-radius-bookkeeping}
r_B\le C\epk\Rk\qquad\text{for}\quad B \in \mathcal{N}_{\rm in},\qquad \text{and}\qquad \dist(x,\FB(u)) \le C\epk\Rk,\quad\text{for}\quad x\in \Gamma^{\rm neck}_\pm.
\end{equation}
Moreover, $\Gamma_\pm^{\rm neck} \subset \bigcup_{B\in\cN_{\rm in}}\overline B$ and
\begin{align}
\label{eq:fb-tree-terminal-neck-boundary-localization-bookkeeping} \qquad \cH^{n-1}(\Gamma_\pm^{\rm neck}) \le C\sum_{B\in\cN_{\rm in}}r_B^{n-1} \le C_\alpha\Rk^{\,n-1}\epk^{\gamma_\alpha},\qquad \left|\mathcal X_{\rm neck}\cap B_{\Rk/8}(\zk)\right| \le C_\alpha\Rk^n\epk^{1+\gamma_\alpha}.
\end{align}
\end{lem}

\begin{proof}
By \eqref{eq:fb-tree-opposite-sign-separation}, \eqref{eq:fb-tree-signed-cover}, and \Cref{defn:fb-tree-signed-pieces},
\[
\Gamma_\pm^{\rm neck} \subset \mathcal X_{\rm neck}\cap B_{\Rk/8}(\zk) \subset \bigcup_{N\in\cN_{\rm in}}\overline N.
\]
Moreover, every point of $(\partial U_\pm\cap B_{\Rk/8}(\zk))\setminus\FB(u)$ lies in $\partial Q^\pm\setminus\FB(u)$ for some $Q\in\cB\cup\cR$.

For each $N\in\cN_{\rm in}$, choose $\zz_N\in\cZ\cap\lsup4N$. Since $x_N\in B_{\Rk}(\zk)$, $r_N\le\Rk$, and $\zz_N\in\lsup4N$, we have $\zz_N\in B_{5\Rk}(\zk)$. Hence \Cref{lem:fb-tree-nonfinal-neck-terminals-neck-scale} gives $r_N\simeq_\delta\rb(\zz_N)$ and $r_N\le C_\delta\epk\Rk$. Together with the preceding covering and $x_N\in\FB(u)$, this proves \eqref{eq:fb-tree-terminal-neck-radius-bookkeeping}.

We claim that, for $Q\in\cB\cup\cR$ and $N\in\cN_{\rm in}$,
\[
\bigl(\partial Q^\pm\setminus\FB(u)\bigr)\cap\overline N \ne\varnothing \quad\Longrightarrow\quad r_Q\simeq_\delta r_N.
\]
Choose $y$ in the intersection. Suppose first that $Q\in\cB$. If $r_N\ge r_Q/4$, then $r_Q\le4r_N$; otherwise \eqref{eq:fb-tree-branching-signed-pieces}, \eqref{eq:fb-tree-branching-condition}, and \Cref{cor:centered-vee-lower-slab} give
\[
c_\ast r_Q\le|\ell_Q(y)| \le|\ell_Q(x_N)|+r_N \le C\delta r_Q+r_N,
\]
and again $r_Q\le Cr_N$. If $Q\in\cR$, the inequality $r_N<3r_Q/5$ would imply
\[
|\zz_N-x_Q|\le5r_N+r_Q<4r_Q,
\]
contrary to the regular stopping rule in \Cref{defn:fb-tree-branching-and-terminal-balls}. Thus in either case $r_Q\le C_\delta r_N$.

For the reverse inequality, it is enough to consider $r_Q\le r_N$. Set $A:=Q$ if $Q\in\cB$, and $A:=PQ$ if $Q\in\cR$. Then $A\in\cB$ and $r_A\in\{r_Q,2r_Q\}$. Choose $\zz_A\in\cZ\cap\lsup4A$. Since $y\in\overline N\cap\partial Q^\pm$, the choices above and \Cref{lem:fb-tree-nonfinal-neck-terminals-neck-scale,lem:fb-tree-small-neck-radii-near-flat-active-balls} give
\[
\rb(\zz_N)\simeq_\delta r_N, \qquad \rb(\zz_A)\le C\delta r_A, \qquad |\zz_A-\zz_N|\le C(r_Q+r_N)\le Cr_N.
\]
Therefore \Cref{lem:local-comparability-of-neck-radii} yields $\rb(\zz_A)\ge c_\delta r_N$, and hence $r_Q\ge c_\delta r_N$. This proves the claim.

For a fixed $N$, the claim restricts the balls $Q$ whose artificial faces meet $\overline N$ to only $O_\delta(1)$ generations. At each such generation, \Cref{defn:fb-tree-dyadic-balls-centered-vee-excess} gives $r_Q/4$-separated centers in a fixed dilation of $N$, so there are only $O_\delta(1)$ of them. The piecewise Lipschitz geometry in \Cref{defn:fb-tree-signed-pieces,lem:fb-tree-regular-terminal-balls-split-into-two-signed-sheets}, together with \Cref{lem:perimeter-bound}, now gives
\[
\cH^{n-1}(\Gamma_\pm^{\rm neck}\cap\overline N) \le Cr_N^{n-1}.
\]

It remains to sum these local estimates. Put $\cN_{\rm in}^j:=\cN_{\rm in}\cap\cN^j$. By \eqref{eq:fb-tree-inner-family}, \eqref{eq:fb-tree-inner-terminal-count}, and \eqref{eq:gamma-alpha},
\[
\#\cN_{\rm in}^j \le C_\alpha2^{j(n-2)}\epk^{\gamma_\alpha-1}.
\]
The radius bound above shows that $\cN_{\rm in}^j=\varnothing$ unless $2^{-j}\le C\epk$. Since $r_N=2^{-j}\Rk$ on $\cN^j$,
\[
\sum_{N\in\cN_{\rm in}}r_N^{n-1} \le C_\alpha\Rk^{n-1}\epk^{\gamma_\alpha-1} \sum_{\{j:\,2^{-j}\le C\epk\}}2^{-j} \le C_\alpha\Rk^{n-1}\epk^{\gamma_\alpha}.
\]
The covering and the local boundary estimate therefore give
\[
\cH^{n-1}(\Gamma_\pm^{\rm neck}) \le C\sum_{N\in\cN_{\rm in}}r_N^{n-1} \le C_\alpha\Rk^{n-1}\epk^{\gamma_\alpha}.
\]
Finally,
\[
\begin{aligned}
|\mathcal X_{\rm neck}\cap B_{\Rk/8}(\zk)| &\le C\sum_{N\in\cN_{\rm in}}r_N^n \le C\left(\sup_{N\in\cN_{\rm in}}r_N\right) \sum_{N\in\cN_{\rm in}}r_N^{n-1} \le C_\alpha\Rk^n\epk^{1+\gamma_\alpha}.
\end{aligned}
\]
\end{proof}

For $p>1$, we say that $p$ is \emph{flux-admissible} if
\begin{equation}\label{eq:flux-admissible}
p\alpha<\frac{n-1}{n-2}, \qquad 2p<\gamma_\alpha,
\end{equation}
where $\gamma_\alpha$ is defined in \eqref{eq:gamma-alpha}. Equivalently, the flux-admissible range is
\[
1<p<\frac{1}{\alpha}\left(1+\min\left\{ \frac{1}{n-2},\frac{\alpha}{2} \right\}\right).
\]
When $p\alpha\le1$, both conditions are immediate.

\subsection{Higher-power boundary flux}

We use \Cref{lem:fb-tree-inner-counting-package,lem:fb-tree-terminal-neck-bookkeeping} to upgrade the transverse boundary flux to a higher-power bound on the two signed pieces.

\begin{prop}
\label{prop:fb-tree-sided-higher-power-transverse-estimate}
After increasing $k$, the following holds. Let $p>1$ be flux-admissible, \eqref{eq:flux-admissible}. Then
\begin{equation}\label{eq:fb-tree-sided-higher-power-transverse}
\int_{\Gamma_+}|\partial_\nu(u-e_{\rm root}\cdot x)|^p\,d\cH^{n-1} +\int_{\Gamma_-}|\partial_\nu(u+e_{\rm root}\cdot x)|^p\,d\cH^{n-1} \le C_{p,\alpha}\Rk^{\,n-1}\epk^{2p}.
\end{equation}
\end{prop}

\begin{proof}
Set
\[
q:=n-1,\qquad d:=n-2+2\beta_\circ.
\]
Since $\beta_\circ<1/2$, we have $d<q$. Let $\Gamma_\pm^{\rm reg}$ be the regular boundary parts from \Cref{lem:fb-tree-terminal-neck-bookkeeping}, and let us define
\[
a_{\alpha,p}:=
\begin{cases}
d-2\beta_\circ p\alpha \quad &\mbox{if} \quad p\alpha\le1\\
(n-2)p\alpha & \mbox{if} \quad p\alpha>1.
\end{cases}
\]

We first observe that, by Lipschitz regularity, \Cref{lem:fb-tree-terminal-neck-bookkeeping}, and $2p<\gamma_\alpha$, the contribution on $\Gamma_\pm^{\rm neck}$ is bounded by $C\Rk^q\epk^{\gamma_\alpha}\le C_{p,\alpha}\Rk^q\epk^{2p}.$ We can therefore focus on the regular part $\Gamma_\pm^{\rm reg}$.

The counting and $2/\alpha$-power estimates in \Cref{lem:fb-tree-inner-counting-package}, with the family $\cD_{\rm in}^j$ defined in \eqref{eq:fb-tree-inner-family}, give (putting $t_A=h_A^{2/\alpha}$ and $s=p\alpha$, in both ranges of $s$),
\[
\sum_A t_A^s \le (\#\cD_{\rm in}^j)^{(1-s)_+} \left(\sum_A t_A\right)^s \le C_{p,\alpha}\epk^{2p}2^{j a_{\alpha,p}},
\]
so that, for every $j\ge0$,
\begin{equation}\label{eq:fb-tree-proof-h-2p-generation-bound}
\sum_{A\in\cD_{\rm in}^j}h_A^{2p} \le C_{p,\alpha}\epk^{2p}2^{j a_{\alpha,p}}.
\end{equation}

On the other hand, for $B\in\cR$, on $\lsup{3/2}{B^+}$, \Cref{lem:fb-tree-regular-terminal-balls-split-into-two-signed-sheets} gives $|\nabla u-a_B|\le Ch_B$, and on $\lsup{3/2}{B^-}$ it gives $|\nabla u+a_B|\le Ch_B$. Since $|\nabla u|=1$ on $\FB(u)$, we get
\[
\bigl|1\mp e_{\rm root}\cdot\nabla u\bigr| \le C\bigl(h_B^2+|a_B-e_{\rm root}|^2\bigr) \quad\text{on}\quad \lsup{3/2}{B^\pm}\cap \Gamma_\pm.
\]

Every point of $\Gamma_\pm^{\rm reg}$ lies in a regular terminal ball. Using \Cref{lem:perimeter-bound} on a fixed enlargement of each such ball,
\[
\begin{split}
&\int_{\Gamma_+^{\rm reg}} \bigl|1-e_{\rm root}\cdot\nabla u\bigr|^p\,d\cH^{q} + \int_{\Gamma_-^{\rm reg}} \bigl|1+e_{\rm root}\cdot\nabla u\bigr|^p\,d\cH^{q} \le C_p\sum_{l\ge0} r_l^q \sum_{B\in\cR^l\cap\cD_{\rm in}^l} \bigl(h_B^{2p}+|a_B-e_{\rm root}|^{2p}\bigr).
\end{split}
\]
For $B\in\cR^l\cap\cD_{\rm in}^l$, \eqref{eq:fb-tree-predecessor-affine-comparison} gives
\[
h_B^{2p}+|a_B-e_{\rm root}|^{2p} \le C_p (l+1)^{2p-1}\sum_{j=0}^{l}h_{P^{\,l-j}B}^{2p},
\]
where the last sum already contains $h_B^{2p}$. Grouping by the level $j$ ancestor and using \Cref{lem:fb-tree-inner-counting-package}, we get
\[
\begin{split}
\sum_{B\in\cR^l\cap\cD_{\rm in}^l} \bigl(h_B^{2p}+|a_B-e_{\rm root}|^{2p}\bigr) &\le C_p (l+1)^{2p-1} \sum_{j=0}^{l} \min\{2^{(l-j)q},C2^{ld}\} \sum_{A\in\cD_{\rm in}^j}h_A^{2p} \\
&\le C_{p,\alpha} (l+1)^{2p-1}\epk^{2p} \sum_{j=0}^{l} \min\{2^{(l-j)q},C2^{ld}\}2^{j a_{\alpha,p}}.
\end{split}
\]
Writing $s=l-j$, the regular-boundary integral is at most
\[
C_{p,\alpha}\Rk^q\epk^{2p} \sum_l (l+1)^{2p-1}2^{-lq} \sum_{s=0}^{l} \min\{2^{sq},C2^{ld}\}2^{(l-s)a_{\alpha,p}}.
\]
By \eqref{eq:flux-admissible}, $a_{\alpha,p} <q$. Splitting the inner sum where $sq=ld$, its two parts are bounded by $C_{p,\alpha}2^{l\{d+(1-d/q)a_{\alpha,p}\}}.$ Therefore, the regular-boundary integral is bounded by
\[
C_{p,\alpha}\Rk^q\epk^{2p} \sum_l (l+1)^{2p-1} 2^{l\{d-q+(1-d/q)a_{\alpha,p}\}}.
\]
Again, since $a_{\alpha,p} <q$ the exponent is negative $d-q+\Bigl(1-\frac dq\Bigr)a_{\alpha,p} <0.$ Thus, the series converges. This gives the desired bound on $\Gamma_+^{\rm reg}\cup\Gamma_-^{\rm reg}$, and hence proves \eqref{eq:fb-tree-sided-higher-power-transverse}.
\end{proof}

\subsection{Local weak Neumann and energy estimates}

We now turn the higher-power boundary estimate in \Cref{prop:fb-tree-sided-higher-power-transverse-estimate} into the local weak Neumann and energy controls needed for the linearization argument.

\begin{lem}
\label{lem:direct-sided-local-flux-slab}
Let $p_{\rm fl}>1$ be flux-admissible, \eqref{eq:flux-admissible}, and write $p_{\rm fl}'=p_{\rm fl}/(p_{\rm fl}-1)$. Fix $C_0\ge1$. After increasing $k$, the following holds. Let
\[
\zz\in B_{\Rk/16}(\zk)\cap\cZ, \qquad 0<R\le\Rk/16.
\]
Let $a\in\mathbb S^{n-1}$ satisfy $|a-e_{\rm root}|\le C_0\epk$ and set $w:=u-a\cdot x$. Then
\begin{equation}\label{eq:direct-sided-local-l1-flux}
\int_{\Gamma_+\cap B_R(\zz)}|w_{\nu}|\,d\cH^{n-1} \le C R^{(n-1)/p_{\rm fl}'}\Rk^{(n-1)/p_{\rm fl}}\epk^2 +C\Rk^{n-1}\epk^{2p_{\rm fl}} .
\end{equation}
Here $C=C(n,\alpha,C_0,p_{\rm fl})$. The analogous flux estimate holds on $U_-$, assuming $|a+e_{\rm root}|\le C_0\epk$.
\end{lem}

\begin{proof}
We prove the $\Gamma_+$ estimate; the $\Gamma_-$ estimate follows by replacing $\Gamma_+$ by $\Gamma_-$ and $e_{\rm root}$ by $-e_{\rm root}$. Put
\[
\Gamma_R^{\rm reg}:=\Gamma_+^{\rm reg}\cap B_R(\zz), \qquad \Gamma_R^{\rm neck}:=(\Gamma_+\cap B_R(\zz))\setminus\Gamma_R^{\rm reg}.
\]

On $\Gamma_R^{\rm reg}$, $|w_{\nu}|=|1-a\cdot\nabla u|.$ If $g,a,e\in\mathbb S^{n-1}$, then
\[
1-a\cdot g=\frac12|a-g|^2 \le |a-e|^2+|e-g|^2 =|a-e|^2+2(1-e\cdot g).
\]
Taking $e=e_{\rm root}$, $g=\nabla u$, and using $|a-e_{\rm root}|\le C_0\epk$, we get
\[
|w_{\nu}|^{p_{\rm fl}} \le C\bigl|1-e_{\rm root}\cdot\nabla u\bigr|^{p_{\rm fl}} +C\epk^{2p_{\rm fl}} \quad\text{on }\Gamma_R^{\rm reg}.
\]
\Cref{lem:perimeter-bound,prop:fb-tree-sided-higher-power-transverse-estimate} yield
\[
\int_{\Gamma_R^{\rm reg}}|w_{\nu}|^{p_{\rm fl}}\,d\cH^{n-1} \le C\Rk^{n-1}\epk^{2p_{\rm fl}}.
\]
H\"older's inequality and \Cref{lem:perimeter-bound} give the sharper regular $L^1$ estimate
\[
\int_{\Gamma_R^{\rm reg}}|w_{\nu}|\,d\cH^{n-1} \le C R^{(n-1)/p_{\rm fl}'}\Rk^{(n-1)/p_{\rm fl}}\epk^2.
\]
On the artificial (neck) part, \Cref{lem:lipschitz-bound} gives $|w_\nu|\le |\nabla u|+|a|\le2$. Hence \eqref{eq:fb-tree-terminal-neck-boundary-localization-bookkeeping} gives
\[
\int_{\Gamma_R^{\rm neck}}|w_{\nu}|\,d\cH^{n-1} \le C_\alpha\Rk^{n-1}\epk^{\gamma_\alpha} \le C\Rk^{n-1}\epk^{2p_{\rm fl}}.
\]
Combining the regular and artificial contributions yields the desired result.
\end{proof}

We shall use the standard Sobolev--Poincar\'e inequality on John domains in the form stated in \Cref{lem:john-domain-direct-sided-lq-upgrade}; its short proof is included in \Cref{app:proofs-standard-auxiliary-lemmas}.

\begin{prop}[Signed $W^{1,2}$ estimate from $L^2$ excess and flux]
\label{prop:direct-sided-w12-from-flux}
In the setting of \Cref{lem:direct-sided-local-flux-slab}, assume in addition that
\[
r_N\le c_JR \quad\text{for every }N\in\cN\text{ meeting }B_{3R/4}(\zz),
\]
where $c_J$ is the constant in \Cref{lem:fb-tree-localized-raw-signed-shadows-john}, decreased if necessary. Let $b\in\mathbb R$ satisfy
\[
|b+a\cdot\zk|\le C_0\epk\Rk,
\]
and set $w:=u-a\cdot x-b$. Define
\[
\Phi_R(w):= R^{2-n}\epk\Rk \int_{\Gamma_+\cap B_R(\zz)} |w_{\nu}|\,d\cH^{n-1}.
\]
Then, for every $2<q\le\frac{2n}{n-2}$,
\begin{equation}
\label{eq:direct-sided-w12-from-flux}
\begin{aligned}
R^2\fint_{U_+\cap B_{3R/4}(\zz)} |\nabla w|^2\,dx &+ \left( \fint_{U_+\cap B_{R/2}(\zz)} |w|^q\,dx \right)^{2/q} \le C_q\fint_{U_+\cap B_R(\zz)} |w|^2\,dx +C\Phi_R(w),
\end{aligned}
\end{equation}
where, by \eqref{eq:direct-sided-local-l1-flux} we know
\begin{equation}\label{eq:direct-sided-phi-flux-bound}
\Phi_R(w) \le C\left[ \Rk^2\epk^3 \left(\frac{R}{\Rk}\right)^{1-(n-1)/p_{\rm fl}} +\Rk^2\epk^{2p_{\rm fl}+1} \left(\frac{\Rk}{R}\right)^{n-2} \right].
\end{equation}
Here $C_q=C_q(n,\alpha,p_{\rm fl},q)$, whereas $C=C(n,\alpha,C_0,p_{\rm fl},q)$. The same conclusions hold analogously on $U_-$.
\end{prop}

\begin{proof}
We prove the $U_+$ estimate. Put
\[
\Omega_s:=U_+\cap B_s(\zz),\qquad \Gamma_s:=\Gamma_+\cap B_s(\zz),\qquad \Gamma_R^{\rm reg}:=\Gamma_R\cap\Gamma_+^{\rm reg}, \qquad \Gamma_R^{\rm neck}:=\Gamma_R\setminus\Gamma_R^{\rm reg}.
\]
Following the local-child chain through $\zz$ to a terminal ball gives $N_\zz\in\cN$ with $\zz\in\lsup{1/4}{N_\zz}$. Thus $N_\zz$ meets $B_{3R/4}(\zz)$, and
\[
\rb(\zz)\le C_\delta r_{N_\zz}\le C_\delta c_JR
\]
by \Cref{lem:fb-tree-nonfinal-neck-terminals-neck-scale}. For $c_J$ sufficiently small, \Cref{lem:decay-of-monneau-excess,lem:fb-tree-mesoscopic-signed-slab-flatness}, applied in a minimizing direction oriented toward $U_+$, gives the fixed volume comparability
\[
cR^n\le |\Omega_{R/2}|\le |\Omega_{3R/4}|\le CR^n .
\]

We also need the boundary size of $w$. On $\Gamma_R^{\rm reg}$, one has $u=0$, while \eqref{eq:direct-sided-root-slab} gives $|e_{\rm root}\cdot(x-\zk)|\le C\epk\Rk$. On $\Gamma_R^{\rm neck}$, \eqref{eq:fb-tree-terminal-neck-radius-bookkeeping} and the two root-cap inclusions in \eqref{eq:direct-sided-root-slab} give directly
\[
u(x)+|e_{\rm root}\cdot(x-\zk)|\le C\epk\Rk \quad\text{for a.e. }x.
\]
Since $|a-e_{\rm root}|\le C_0\epk$, $|b+a\cdot\zk|\le C_0\epk\Rk$, and $|x-\zk|\le\Rk/8$, it follows that
\[
|w|\le C\epk\Rk\qquad\text{on}\quad \Gamma_R.
\]

Let $\eta\in C_c^\infty(B_R(\zz))$ satisfy $\eta\equiv1$ on $B_{3R/4}(\zz)$ and $|\nabla\eta|\le C/R$. Integrate by parts on the finite-perimeter set $U_+$ with the compactly supported test function $w\eta^2$. Taking absolute values gives
\[
\int_{\Omega_R}|\nabla w|^2\eta^2 \le 2\left|\int_{\Omega_R}w\eta\,\nabla w\cdot\nabla\eta\right| +\int_{\Gamma_R}|w|\,|w_{\nu}|\,\eta^2\,d\cH^{n-1} .
\]
There is no spherical boundary term because $\eta$ is compactly supported in $B_R(\zz)$. Using $|w|\le C\epk\Rk$ on $\Gamma_R$, Cauchy's inequality, and \eqref{eq:direct-sided-local-l1-flux}, we obtain the gradient term in \eqref{eq:direct-sided-w12-from-flux}:
\[
R^2\fint_{\Omega_{3R/4}}|\nabla w|^2 \le C_q\fint_{\Omega_R}w^2 +C R^{2-n}\epk\Rk \int_{\Gamma_R}|w_{\nu}|\,d\cH^{n-1}.
\]

It remains to prove the $L^q$ non-concentration estimate. \Cref{lem:fb-tree-localized-raw-signed-shadows-john} gives a $c_J$-John domain $\mathcal O$ satisfying
\[
\Omega_{R/2}\subset\mathcal O \subset \Omega_{3R/4}, \qquad |\mathcal O|\ge c_JR^n .
\]
Since $\Omega_{R/2}\subset\mathcal O$ and both sets have volume comparable to $R^n$, \Cref{lem:john-domain-direct-sided-lq-upgrade} gives
\[
\left(\fint_{\Omega_{R/2}}|w|^q\,dx\right)^{2/q} \le C_q\left[ \fint_{\mathcal O}|w|^2\,dx +R^2\fint_{\mathcal O}|\nabla w|^2\,dx \right].
\]
Since $\mathcal O\subset\Omega_{3R/4}\subset\Omega_R$ and $|\mathcal O|\simeq R^n$,
\[
\fint_{\mathcal O}|w|^2\le C\fint_{\Omega_R}|w|^2.
\]
Moreover, the gradient estimate above gives
\[
R^2\fint_{\mathcal O}|\nabla w|^2 \le C_q\fint_{\Omega_R}|w|^2+C\Phi_R(w).
\]
Hence
\[
\left(\fint_{\Omega_{R/2}}|w|^q\,dx\right)^{2/q} \le C_q\fint_{\Omega_R}|w|^2+C\Phi_R(w).
\]
Together with the gradient term, this proves \eqref{eq:direct-sided-w12-from-flux}.
\end{proof}

\begin{remark}[Scaling of the weak Neumann residual]
Let us record the scaling of the weak Neumann residual used in both iterations below. In the setting of \Cref{prop:direct-sided-w12-from-flux}, let $r,H,\rho>0$ and set
\[
v(y):=\frac{w(\zz+ry)}{H}, \qquad \widetilde\Omega_r:=r^{-1}(U_+-\zz).
\]
If $S\ge\rho r$, then for every $\psi\in C_c^\infty(B_\rho)$, integration by parts and the definition of $\Phi_S$ give
\begin{equation}\label{eq:scaled-weak-residual-from-flux}
\begin{aligned}
\left|\int_{\widetilde\Omega_r}\nabla v\cdot\nabla\psi\,dy\right| &\le \left(\frac Sr\right)^{n-2} \frac{\Phi_S(w)}{H\epk\Rk}\|\psi\|_{L^\infty} = \left(\frac Sr\right)^{n-2} \frac{\Phi_S(w)}{H(S)^2} \left(\frac{H(S)}H\right)^2 \frac{H}{\epk\Rk}\|\psi\|_{L^\infty},
\end{aligned}
\end{equation}
where $H(S)>0$ is any comparison amplitude.
\end{remark}

\section{Neumann linearization and the flat scale}
\label{sec:linear}

The goal of this section is to prove the analogue of \cite[Proposition~8.1]{CFFS25} for the free-boundary centered domains constructed in \Cref{sec:fb-tree}. We first obtain affine control at the flat scale. The decay below that scale is postponed to the final section, where the strict inequality $\alpha_\star<1$ allows a flux exponent above the critical value.

\subsection{Neumann compactness}

Throughout this section $U_\pm\subset B_{\Rk/2}(\zk)$ are the signed domains from \Cref{defn:fb-tree-signed-pieces}.

We use the following compactness lemma for harmonic functions on domains trapped between two nearby half-spaces; it is proved in \Cref{lem:linear-neumann-compactness-half-ball-l2}.

\begin{lem}[Global Neumann compactness]
\label{lem:linear-neumann-compactness-global}
Let
\[
m\ge2, \qquad q>2, \qquad N\in\mathbb N_0, \qquad \vartheta\in(0,1).
\]
For every $\varepsilon_{\rm cmp}>0$ there exists $\delta_{\rm cmp}=\delta_{\rm cmp} (\varepsilon_{\rm cmp},m,q,N,\vartheta)\in(0,\frac14)$ with the following property. Let $\Omega\subset\mathbb R^m$ be open and let
\[
v\in H^1(\Omega\cap B_{\delta_{\rm cmp}^{-1}}) \cap L^q(\Omega\cap B_{\delta_{\rm cmp}^{-1}})
\]
be harmonic in $\Omega\cap B_{\delta_{\rm cmp}^{-1}}$. Suppose that there is $\nu\in\mathbb S^{m-1}$ such that, for every $1\le\rho\le\delta_{\rm cmp}^{-1}$,
\[
B_\rho\cap\{\nu\cdot x\ge\delta_{\rm cmp}\rho\} \subset \Omega\cap B_\rho \subset B_\rho\cap\{\nu\cdot x\ge-\delta_{\rm cmp}\rho\},
\]
and
\[
\left(\fint_{\Omega\cap B_\rho}|v|^2\,dx\right)^{1/2} +\rho\left(\fint_{\Omega\cap B_\rho}|\nabla v|^2\,dx\right)^{1/2} +\left(\fint_{\Omega\cap B_\rho}|v|^q\,dx\right)^{1/q} \le \rho^{N+\vartheta}.
\]
Assume also that, for every $1\le\rho\le\delta_{\rm cmp}^{-1}$ and every $\varphi\in C_c^\infty(B_\rho)$,
\[
\left| \int_{\Omega\cap B_\rho}\nabla v\cdot\nabla\varphi\,dx \right| \le \delta_{\rm cmp}\rho^{m+N+\vartheta-2} \|\varphi\|_{L^\infty(B_\rho)}.
\]
Then there is a harmonic polynomial $P_N$ of degree at most $N$, even with respect to the hyperplane $\nu^\perp$, such that
\[
\left( \fint_{\Omega\cap B_{1/2}}|v-P_N|^2\,dx \right)^{1/2} \le\varepsilon_{\rm cmp}.
\]
\end{lem}

The proof combines \Cref{lem:linear-neumann-compactness-half-ball-l2} with the Liouville classification of even harmonic functions of polynomial growth; details are given in \Cref{app:proofs-standard-auxiliary-lemmas}.

\subsection{Affine improvement to the flat scale}

We iterate \Cref{lem:linear-neumann-compactness-global} to obtain affine control down to the flat scale.

Since $3\le n<\nAC\le7$, fix the dimensional exponent
\begin{equation}\label{eq:asym-decay-first-flux-exponent}
p_1:=1+\frac{4}{5\max\{2,n-2\}}.
\end{equation}
Thus $p_1=7/5$ when $n=4$. This exponent is uniformly flux-admissible, \eqref{eq:flux-admissible}, for the selected range $3/4\le\alpha\le101/100$. Indeed, for $3\le n\le6$,
\[
p_1\alpha\le\frac{101}{100}p_1<\frac{n-1}{n-2}, \qquad 2p_1\le\frac{14}{5}<\frac{301}{101} \le1+\frac2\alpha=\gamma_\alpha.
\]

Choose $\cttc>0$ so that
\begin{equation}\label{eq:asym-decay-flat-scale-chi-smallness}
\cttc\le\frac12 \min\left\{ \frac{1}{\frac95+\frac{n-1}{p_1}}, \frac{2p_1-1}{n+\frac45}, \frac{2}{n+\frac13} \right\} = \frac{5}{18 +10\frac{n-1}{p_1}}.
\end{equation}
In the case $n=4$, $\alpha<1$, we make this choice small enough that also
\begin{equation}\label{eq:single-selection-flat-scale-reserves}
10\cttc<\frac{2(1-\alpha)}{2\alpha+1}.
\end{equation}
Define the flat scale, leaving its dependence on $k$ implicit, by
\begin{equation}\label{eq:flat-scale-definition}
R_\flat:=\epk^\cttc\Rk.
\end{equation}
Below, taking a smaller $\cttc$ only changes how large $k$ must be.

\begin{prop}[Flat-scale affine control]
\label{prop:asymmetric-excess-decay-almost-critical}
With the choices above, after increasing $k$, the following holds.

Let $\zz\in B_{\Rk/16}(\zk)\cap\cZ$. There exist coefficients, depending on $\zz$, with
\begin{equation}\label{eq:asymmetric-excess-flat-scale-normalization}
a_+^\flat,a_-^\flat\in\mathbb S^{n-1}, \quad b_+^\flat,b_-^\flat\in\mathbb R,\qquad |a_+^\flat-e_{\rm root}|+|a_-^\flat+e_{\rm root}|\le C\epk, \qquad |b_+^\flat+a_+^\flat\cdot\zk| +|b_-^\flat+a_-^\flat\cdot\zk| \le C\epk\Rk,
\end{equation}
such that, writing
\[
L_\ast^\flat(x):=a_\ast^\flat\cdot x+b_\ast^\flat, \qquad \ast\in\{+,-\},
\]
one has the flat-scale $W^{1,2}$ estimate
\begin{equation}\label{eq:asymmetric-excess-flat-scale-affine-w12}
\sum_{\ast=\pm} \fint_{U_\ast\cap B_{8R_\flat}(\zz)} |u-L_\ast^\flat|^2\,dx +(R_\flat)^2 \sum_{\ast=\pm} \fint_{U_\ast\cap B_{8R_\flat}(\zz)} |\nabla u-a_\ast^\flat|^2\,dx \le C\epk^{2+2\cttc/3}(R_\flat)^2,
\end{equation}
where $C$ is dimensional.
\end{prop}

\begin{proof}
We prove it in $U_+$; the construction on $U_-$ is identical after replacing $e_{\rm root}$ by $-e_{\rm root}$.

Set
\[
\gamma:=\frac25, \qquad q_{\rm nc}:=\frac52\le\frac{2n}{n-2}, \qquad h(r):=\epk\left(\frac r{\Rk}\right)^\gamma r, \qquad r_i:=2^{-i}\Rk, \qquad h_i:=h(r_i).
\]
We prove the following one-step statement. The constants $j_0$ and $C_{\rm ind}\ge1$ will be chosen below. Suppose that, for some $j\ge j_0$, affine maps
\[
L_i(x):=a_i\cdot x+b_i, \qquad a_i\in\mathbb S^{n-1}, \qquad c_i:=L_i(\zz), \qquad 4\le i\le j,
\]
have been constructed, starting from $L_4(x)=e_{\rm root}\cdot(x-\zk)$, satisfy
\begin{equation}\label{eq:alt-flat-induction}
\left( \fint_{U_+\cap B_{r_i}(\zz)}|u-L_i|^2\,dx \right)^{1/2} \le C_{\rm ind}h_i \qquad\text{for}\quad 4\le i\le j.
\end{equation}
If $r_{j+1}\ge R_\flat$, then, once $k$ is sufficiently large after all constants have been fixed, we construct $L_{j+1}$ so that \eqref{eq:alt-flat-induction} also holds at $j+1$. From here, we will later deduce the bound in $W^{1,2}$.

Since $\frac{\epk\Rk}{R_\flat} =\epk^{1-\cttc}\to0,$ after increasing $k$, every terminal neck ball $N$ meeting $B_{3r/4}(\zz)$, with $R_\flat\le r\le\Rk/16$, satisfies $r_N\le C_\delta\epk\Rk\le c_Jr$ by \Cref{lem:fb-tree-nonfinal-neck-terminals-neck-scale}. Thus the terminal-radius hypothesis in \Cref{prop:direct-sided-w12-from-flux} holds throughout this range. The root slab \eqref{eq:direct-sided-root-slab}, recentered at $\zz$, gives
\begin{equation}\label{eq:alt-flat-volume}
c r^n\le |U_\pm\cap B_r(\zz)|\le C r^n \qquad\text{for}\quad R_\flat\le r\le\Rk/16.
\end{equation}

\medskip
\noindent\emph{Step 1: Consequences of the induction hypothesis.}
For $4\le i<j$, \eqref{eq:alt-flat-induction} at the two consecutive scales, \eqref{eq:alt-flat-volume}, and \Cref{lem:affine-comparison-proportional} give
\begin{equation}\label{eq:alt-flat-jumps}
r_i|a_{i+1}-a_i|+|c_{i+1}-c_i| \le C C_{\rm ind}h_i.
\end{equation}
Summing this estimate gives
\begin{equation}\label{eq:alt-flat-slope-sum}
r_i|a_j-a_i|+|c_j-c_i|\le C C_{\rm ind}h_i.
\end{equation}
Summing from $L_4=e_{\rm root}\cdot(x-\zk)$ yields
\begin{equation}\label{eq:alt-flat-normalization}
|a_j-e_{\rm root}|\le C C_{\rm ind}\epk, \qquad |b_j+a_j\cdot\zk|\le C C_{\rm ind}\epk\Rk.
\end{equation}

We now use the flux estimate only for the map $L_j$. For $R_\flat\le r\le\Rk/16$, put $t:=r/\Rk$. The preceding slope and intercept bounds, \eqref{eq:direct-sided-phi-flux-bound}, and \eqref{eq:asym-decay-flat-scale-chi-smallness}, with $p_{\rm fl}=p_1$, give
\[
\frac{\Phi_r(u-L_j)}{h(r)^2} \le C(C_{\rm ind})\left[ \epk t^{-1-(n-1)/p_1-2\gamma} +\epk^{2p_1-1}t^{-n-2\gamma} \right].
\]
Set
\[
\eta_k :=\left[ \epk^{1-\cttc(1+(n-1)/p_1+2\gamma)} +\epk^{2p_1-1-\cttc(n+2\gamma)} \right]^{1/2}.
\]
The bound \eqref{eq:asym-decay-flat-scale-chi-smallness} gives $\eta_k\to0$. Uniformly in $\zz$, $r$, and $j$,
\begin{equation}\label{eq:alt-flat-flux}
\Phi_r(u-L_j)\le C(C_{\rm ind})\eta_k^2 h(r)^2.
\end{equation}
Here and below, $C(C_{\rm ind})$ may increase from line to line.

With $\delta_{\rm cmp}$ to be chosen below, suppose that $j_0$ is such that
\begin{equation}\label{eq:alt-flat-window}
2^{j_0}\ge64\delta_{\rm cmp}^{-1}.
\end{equation}
Set
\[
v_j(y):=\frac{u(\zz+r_jy)-L_j(\zz+r_jy)}{h_j}, \qquad \Omega_j:=r_j^{-1}(U_+-\zz).
\]
For $1\le\rho\le\delta_{\rm cmp}^{-1}$, there is an index $i\ge4$ such that
\[
B_{\rho r_j}(\zz)\subset B_{r_i/2}(\zz), \qquad 2\rho r_j\le r_i<4\rho r_j.
\]
Since $h_i/h_j\le(4\rho)^{1+\gamma}$, \eqref{eq:alt-flat-induction}, \eqref{eq:alt-flat-slope-sum}, \Cref{lem:affine-comparison-proportional,prop:direct-sided-w12-from-flux}, and \eqref{eq:alt-flat-flux} give
\begin{equation}\label{eq:alt-flat-growth}
\begin{aligned}
& \left(\fint_{\Omega_j\cap B_\rho}\hspace{-2mm}|v_j|^2\,dy\right)^{1/2} \hspace{-2mm}+\rho\left( \fint_{\Omega_j\cap B_\rho}\hspace{-2mm}|\nabla v_j|^2\,dy \right)^{1/2}\hspace{-2mm} +\left( \fint_{\Omega_j\cap B_\rho}\hspace{-2mm}|v_j|^{q_{\rm nc}}\,dy \right)^{1/q_{\rm nc}} \hspace{-2mm}\le \left( K C_{\rm ind} +C(C_{\rm ind})\eta_k \right)\rho^{1+\gamma}.
\end{aligned}
\end{equation}
Here $K$ is independent of $C_{\rm ind}$ and $\delta_{\rm cmp}$.

The same flux bound and \eqref{eq:scaled-weak-residual-from-flux}, with
\[
r=r_j, \qquad H=h_j, \qquad S=r_i, \qquad H(S)=C_{\rm ind}h_i,
\]
give, for $\varphi\in C_c^\infty(B_\rho)$,
\begin{equation}\label{eq:alt-flat-residual}
\left| \int_{\Omega_j}\nabla v_j\cdot\nabla\varphi\,dy \right| \le C(C_{\rm ind})\eta_k^2\rho^{n+2\gamma} \|\varphi\|_{L^\infty}.
\end{equation}
Finally, \eqref{eq:alt-flat-window} gives $B_{\rho r_j}(\zz)\subset B_{\Rk/8}(\zk)$. The root slab \eqref{eq:direct-sided-root-slab} and the bounds following \eqref{eq:alt-flat-slope-sum} yield
\begin{equation}\label{eq:alt-flat-slab}
B_\rho\cap\{a_j\cdot y\ge\sigma_{k,j}\rho\} \subset\Omega_j\cap B_\rho \subset B_\rho\cap\{a_j\cdot y\ge-\sigma_{k,j}\rho\},
\end{equation}
where
\begin{equation}\label{eq:alt-flat-slab-width}
0\le\sigma_{k,j} \le C C_{\rm ind}\left(\frac{\epk\Rk}{r_j}+\epk\right) \le C C_{\rm ind}\epk^{1-\cttc}.
\end{equation}
This proves the estimates needed for the one-step improvement.

\medskip
\noindent\emph{Step 2: Choice of constants and one-step improvement.}
Choose $\varepsilon_{\rm cmp}>0$ so that
\begin{equation}\label{eq:alt-flat-tolerance}
(K+1)\varepsilon_{\rm cmp}\le2^{-3-\gamma},
\end{equation}
and let $\delta_{\rm cmp}$ be supplied by \Cref{lem:linear-neumann-compactness-global} with
\[
(m,q,N,\vartheta)=(n,q_{\rm nc},1,\gamma).
\]
Choose $j_0\ge4$ so that \eqref{eq:alt-flat-window} holds.

For $4\le i\le j_0$, take
\[
L_i(x):=e_{\rm root}\cdot(x-\zk).
\]
The root slab \eqref{eq:direct-sided-root-slab}, \Cref{lem:fb-tree-small-neck-radii-near-flat-active-balls}, and \eqref{eq:fb-tree-centered-vee-excess} give
\[
\int_{U_+\cap B_{\Rk/8}(\zk)} |u-e_{\rm root}\cdot(x-\zk)|^2\,dx \le C\epk^2\Rk^{n+2}.
\]
Since the base block is finite, \eqref{eq:alt-flat-volume} proves \eqref{eq:alt-flat-induction} for $4\le i\le j_0$ once $C_{\rm ind}=C_{\rm ind}(j_0)$ is chosen sufficiently large.

Only now, after $\delta_{\rm cmp}$, $j_0$, and $C_{\rm ind}$ have been fixed, choose the lower bound on $k$. In addition to the preliminary geometric requirements and the applicability thresholds in \Cref{lem:direct-sided-local-flux-slab,prop:direct-sided-w12-from-flux} with normalization bound $C C_{\rm ind}$, require $r_{j_0}\ge32R_\flat$ and
\begin{equation}\label{eq:alt-flat-final-k-choice}
C(C_{\rm ind})\eta_k\le C_{\rm ind}, \qquad \frac{C(C_{\rm ind})\eta_k^2 \delta_{\rm cmp}^{-(1+\gamma)}} {(K+1)C_{\rm ind}}\le\delta_{\rm cmp}, \qquad C C_{\rm ind}\epk^{1-\cttc}\le\delta_{\rm cmp}.
\end{equation}
These inequalities are possible because $\delta_{\rm cmp}$ and $C_{\rm ind}$ are now fixed and $\eta_k,\epk\to0$.

Assume that the maps have been constructed through $j\ge j_0$, with $r_{j+1}\ge R_\flat$. Equations \eqref{eq:alt-flat-growth}, \eqref{eq:alt-flat-residual}, \eqref{eq:alt-flat-slab}, \eqref{eq:alt-flat-slab-width}, and \eqref{eq:alt-flat-final-k-choice} show that $(K+1)^{-1}C_{\rm ind}^{-1}v_j$ satisfies every hypothesis of \Cref{lem:linear-neumann-compactness-global} on $1\le\rho\le\delta_{\rm cmp}^{-1}$. Hence there is
\[
P_j(y)=q_j\cdot y+d_j, \qquad q_j\cdot a_j=0,
\]
such that
\begin{equation}\label{eq:alt-flat-output}
\left( \fint_{\Omega_j\cap B_{1/2}} |((K+1)C_{\rm ind})^{-1}v_j-P_j|^2\,dy \right)^{1/2} \le\varepsilon_{\rm cmp}.
\end{equation}
The induction estimate \eqref{eq:alt-flat-induction}, restricted to $B_{r_j/2}(\zz)$ using \eqref{eq:alt-flat-volume}, together with \Cref{lem:affine-comparison-proportional} and \eqref{eq:alt-flat-output}, yields
\[
|q_j|+|d_j| \le C((K+1)^{-1}+\varepsilon_{\rm cmp}), \qquad (K+1)(|q_j|+|d_j|)\le C.
\]

Set $A_j:=(K+1)C_{\rm ind}h_j$ and define the corrected affine map
\[
\bar L_{j+1}(x) :=L_j(x)+A_jP_j\!\left(\frac{x-\zz}{r_j}\right), \qquad \bar a_{j+1}:=a_j+\frac{A_j}{r_j}q_j.
\]
Normalize its slope while preserving its value at $\zz$:
\[
a_{j+1}:=\frac{\bar a_{j+1}}{|\bar a_{j+1}|}, \qquad L_{j+1}(x) :=\bar L_{j+1}(x)+(a_{j+1}-\bar a_{j+1})\cdot(x-\zz).
\]
Since $q_j\perp a_j$ and
\[
\frac{A_j}{r_j} =(K+1)C_{\rm ind}\epk\left(\frac{r_j}{\Rk}\right)^\gamma=o(1),
\]
the unit-slope normalization changes the corrected affine map on $B_{r_j/2}(\zz)$ by
\[
\sup_{B_{r_j/2}(\zz)}|L_{j+1}-\bar L_{j+1}| \le C r_j\left(\frac{A_j}{r_j}\right)^2 =o(C_{\rm ind}h_j).
\]
Multiplying \eqref{eq:alt-flat-output} by $A_j$, using \eqref{eq:alt-flat-tolerance}, and enlarging the final lower bound on $k$ to absorb this quadratic error,
\[
\left( \fint_{U_+\cap B_{r_{j+1}}(\zz)} |u-L_{j+1}|^2\,dx \right)^{1/2} \le2^{-2-\gamma}C_{\rm ind}h_j \le C_{\rm ind}h_{j+1}.
\]
This proves the one-step statement. The comparison leading to \eqref{eq:alt-flat-jumps} then applies also to the new consecutive pair, so the induction closes from the base block.

\medskip
\noindent\emph{Step 3: Terminal extraction.}
Run the same construction on $U_-$, beginning with $-e_{\rm root}\cdot(x-\zk)$. By the final choice of $k$, choose $i\ge j_0$ so that $16R_\flat\le r_i<32R_\flat$ and set $L_\pm^\flat:=L_{i,\pm}$. Summing \eqref{eq:alt-flat-jumps} gives \eqref{eq:asymmetric-excess-flat-scale-normalization}.

Here $r_{i+1}\ge8R_\flat$, so \eqref{eq:alt-flat-growth} was established at $i$. Since $r_i\simeq R_\flat$, taking $\rho=1$ there on both sides and restricting, using \eqref{eq:alt-flat-volume}, to $B_{8R_\flat}(\zz)$ gives
\[
\begin{aligned}
&\sum_{\ast=\pm} \fint_{U_\ast\cap B_{8R_\flat}(\zz)} |u-L_\ast^\flat|^2\,dx +(R_\flat)^2 \sum_{\ast=\pm} \fint_{U_\ast\cap B_{8R_\flat}(\zz)} |\nabla u-a_\ast^\flat|^2\,dx \le C C_{\rm ind}^2h_i^2.
\end{aligned}
\]
Finally, $h_i^2 \le C\epk^{2+2\cttc\gamma}(R_\flat)^2 \le C\epk^{2+2\cttc/3}(R_\flat)^2,$ because $\gamma=2/5>1/3$ and $C_{\rm ind}$ is fixed. This is \eqref{eq:asymmetric-excess-flat-scale-affine-w12}.
\end{proof}

\subsection{Comparing the two flat-scale planes}
\label{sec:dilation-test-functions}

We compare the two affine approximations at the flat scale, obtaining a dichotomy between improved slope cancellation and geometric separation.

\begin{lem}
\label{lem:asymmetric-plane-dichotomy}
Let $a_\pm^\flat,b_\pm^\flat$ be the flat-scale affine maps from \Cref{prop:asymmetric-excess-decay-almost-critical} at $\zz\in\cZ\cap B_{\Rk/16}(\zk)$. Write
\[
\bar b_\ast^\flat:=a_\ast^\flat\cdot\zz+b_\ast^\flat, \qquad L_\ast^\flat(x)=a_\ast^\flat\cdot(x-\zz)+\bar b_\ast^\flat, \qquad \ast\in\{+,-\}.
\]
Then, we have
\begin{equation}\label{eq:asymmetric-plane-max-lower-bound}
u\ge \max\{L_+^\flat,L_-^\flat\} -C\epk^{1+\cttc/6}R_\flat \qquad\text{in }B_{5R_\flat}(\zz).
\end{equation}
Moreover, one of the following alternatives holds:
\begin{enumerate}
\item[(a)] either the two slopes match to the improved order
\begin{equation}\label{eq:asymmetric-plane-good-alternative}
|a_+^\flat+a_-^\flat|\le C\epk^{1+\cttc/6};
\end{equation}
\item[(b)] or the two positive half-spaces do not meet in $B_{4R_\flat}(\zz)$:
\begin{equation}\label{eq:asymmetric-plane-separated-alternative}
\{L_+^\flat>0\}\cap\{L_-^\flat>0\}\cap B_{4R_\flat}(\zz)=\varnothing.
\end{equation}
\end{enumerate}
\end{lem}

\begin{proof}
Set
\begin{equation}\label{eq:asymmetric-plane-dichotomy-theta-flat}
\theta_\flat:=\epk^{1+\cttc/6}.
\end{equation}
Recall that $a_+^\flat,a_-^\flat\in\mathbb S^{n-1}$. We use the flat-scale $W^{1,2}$ estimate \eqref{eq:asymmetric-excess-flat-scale-affine-w12} throughout the proof.

Let $e_0$ minimize $\bbE_\zz(8R_\flat;e)$. By \Cref{lem:fb-tree-small-neck-radii-near-flat-active-balls}, $\frac{\rb(\zz)}{8R_\flat} \le C\epk^{1-\cttc}=o(1),$ and so \Cref{lem:decay-of-monneau-excess} gives
\[
\bbE_\zz(8R_\flat;e_0)=\bbE_\zz(8R_\flat)=o(1).
\]
This $o(1)$, and all those below, are uniform for $\zz\in\cZ\cap B_{\Rk/16}(\zk)$. The fixed-direction estimate in \eqref{eq:fixed-direction-monneau-doubling}, applied with outer radius $64R_\flat$ and ratio $1/8$, gives
\[
\bbE_\zz(64R_\flat;e_0)\le C\bbE_\zz(8R_\flat;e_0)=o(1).
\]
Apply \Cref{lem:fb-tree-mesoscopic-signed-slab-flatness} with radius $8R_\flat$, orienting $e_0$ so that its positive cap belongs to $U_+$. This is admissible for large $k$: the ball lies in $B_{\Rk/2}(\zk)$, the relation $R_\flat\gg\epk\Rk$ makes the scale-separation hypothesis automatic, and $\bbE_\zz(64R_\flat;e_0)=o(1)$ makes the relative slab width vanish. Comparing $L_+^\flat$ and $L_-^\flat$ with the two affine branches of $V_{\zz,e_0}$ on the resulting proportional signed caps, using \eqref{eq:asymmetric-excess-flat-scale-affine-w12} and \Cref{lem:affine-comparison-proportional}, yields
\begin{equation}\label{eq:flat-dichotomy-rough-vee-alignment}
|a_+^\flat-e_0|+|a_-^\flat+e_0| +R_\flat^{-1}\bigl(|\bar b_+^\flat|+|\bar b_-^\flat|\bigr)=o(1).
\end{equation}

For $\ast\in\{+,-\}$, let
\[
y_\ast:=\zz-\bar b_\ast^\flat a_\ast^\flat, \qquad L_\ast^\flat(x)=a_\ast^\flat\cdot(x-y_\ast).
\]
By \eqref{eq:flat-dichotomy-rough-vee-alignment}, $|y_\ast-\zz|=o(R_\flat)$. Set
\[
D_\ast:=B_{7R_\flat}(y_\ast)\cap \{L_\ast^\flat>7R_\flat/8\}.
\]
For large $k$, the outer inclusion in \eqref{eq:fb-tree-mesoscopic-signed-slab-flatness} and \eqref{eq:flat-dichotomy-rough-vee-alignment} give
\[
D_\ast\subset U_\ast\cap B_{8R_\flat}(\zz).
\]
Since $|D_\ast|\simeq R_\flat^n$, the $L^2$ estimate \eqref{eq:asymmetric-excess-flat-scale-affine-w12} and Cauchy's inequality imply
\[
\frac1{7R_\flat}\fint_{D_\ast}(u-L_\ast^\flat)_-\,dx \le C\epk^{1+\cttc/3}.
\]
Applying \Cref{lem:one-sided-lower-bound} with center $y_\ast$, direction $a_\ast^\flat$, and radius $7R_\flat$, and using $B_{5R_\flat}(\zz)\subset B_{21R_\flat/4}(y_\ast)$ for large $k$, gives
\[
u\ge L_\ast^\flat-C\epk^{1+\cttc/3}R_\flat \qquad\text{in }B_{5R_\flat}(\zz).
\]

Set
\[
\widetilde u(y):=R_\flat^{-1}u(\zz+R_\flat y), \qquad \widetilde L_\ast(y):=R_\flat^{-1}L_\ast^\flat(\zz+R_\flat y).
\]
Then
\[
\widetilde u\ge \widetilde L_\ast-C\epk^{1+\cttc/3} \qquad\text{in }B_5,\qquad \ast\in\{+,-\}.
\]
Since $\epk^{1+\cttc/3}\le\theta_\flat$, we have
\begin{equation}\label{eq:flat-dichotomy-max-lower-bound}
\widetilde u\ge \max\{\widetilde L_+,\widetilde L_-\}-C\theta_\flat \qquad\text{in }B_5.
\end{equation}
Scaling this estimate back proves \eqref{eq:asymmetric-plane-max-lower-bound}.

Now suppose that \eqref{eq:asymmetric-plane-separated-alternative} fails, namely
\[
\{\widetilde L_+>0\}\cap\{\widetilde L_->0\}\cap B_4\ne\varnothing .
\]
We claim that this forces
\begin{equation}\label{eq:flat-dichotomy-rescaled-good}
|a_+^\flat+a_-^\flat|\le C\theta_\flat .
\end{equation}

Indeed, choose $y'\in B_4$ with $\widetilde L_+(y'),\widetilde L_-(y')>0$. The rough alignment \eqref{eq:flat-dichotomy-rough-vee-alignment} gives
\[
|e_0\cdot y'|=o(1),\qquad |a_+^\flat+a_-^\flat|=o(1).
\]
Since the two slopes have unit length, $a_+^\flat\cdot (a_+^\flat+a_-^\flat)=\frac{|a_+^\flat+a_-^\flat|^2}{2}.$ Consequently, with $v:=(a_+^\flat+a_-^\flat)/|a_+^\flat+a_-^\flat|$,
\[
|e_0\cdot v| \le |e_0-a_+^\flat|+\frac{|a_+^\flat+a_-^\flat|}{2}=o(1).
\]
Set $y_1:=y'+\frac12v$. Since $\widetilde L_+(y')+\widetilde L_-(y')>0$ and $\nabla \widetilde L_++\nabla \widetilde L_-=a_+^\flat+a_-^\flat$, for every $y\in B_{1/4}(y_1)$,
\[
\widetilde L_+(y)+\widetilde L_-(y)\ge \frac14|a_+^\flat+a_-^\flat|, \qquad \max\{\widetilde L_+(y),\widetilde L_-(y)\}\ge\frac18|a_+^\flat+a_-^\flat|.
\]
Moreover $B_{1/4}(y_1)\subset B_5$ and $e_0\cdot y_1=o(1)$. Projecting $y_1$ orthogonally onto $e_0^\perp$, we obtain $y_0\in e_0^\perp$ such that
\[
B_{1/8}(y_0)\subset B_{1/4}(y_1)\subset B_5.
\]
Suppose, toward a contradiction, that $|a_+^\flat+a_-^\flat|\ge M\theta_\flat$. Choosing $M$ dimensional large in \eqref{eq:flat-dichotomy-max-lower-bound} gives $B_{1/8}(y_0)\subset \{\widetilde u>0\}.$ The vee term in $\bbE_\zz(64R_\flat;e_0)=o(1)$, followed by Lipschitz interpolation, gives in the present rescaling
\[
\bigl\|\widetilde u-|e_0\cdot y|\bigr\|_{L^\infty(B_{16})}=o(1).
\]
Since $y_0\in e_0^\perp\cap B_5$, \Cref{lem:consequences-of-closeness-to-a-vee} yields
\[
\dist(y_0,\{\widetilde u=0\})=o(1),
\]
contradicting $B_{1/8}(y_0)\subset\{\widetilde u>0\}$. This proves \eqref{eq:flat-dichotomy-rescaled-good}.

Thus, if the half-spaces are not separated in $B_4$, the slope-matching estimate above must hold. Scaling back gives \eqref{eq:asymmetric-plane-good-alternative}.
\end{proof}

\subsection{Stability and the neck-charge gain}

We now use the stability inequalities at the flat scale to improve the power controlling the neck charge.

\begin{lem}
\label{lem:positive-dilation-gap-at-the-neck-scale}
There exist constants $c>0$ and $\mu\in(0,\frac1{32})$ with the following property.

Let $\lambda > 0$, $a\in \R^n$, and $Q\in O(n)$, and denote the similarity transformation $v$ of $u$:
\[
v(y)=\lambda^{-1}u(a+\lambda Qy).
\]
Let us also denote
\[
(\zz, r) = \left(\lambda^{-1}Q^T(\zeta-a), \lambda^{-1}\rb(\zeta)\right),\qquad \zeta\in\cZ,
\]
a transported bad center and bad radius pair for $v$. Suppose that $\zz\in B_{3/4}$, $r\le\mu$, and
\[
p=(t,0,\ldots,0),\qquad |t|\ge1.
\]
Then
\[
\int_{B_{2r}(\zz)\cap\{v>0\}}\cK_p^2(v)\,dy \ge cr^{n-2},
\]
where $\cK_p$ is computed in the current coordinates and it is given as in \Cref{lem:dilation-sternberg-zumbrun-inequality}.
\end{lem}

\begin{proof}
Suppose the conclusion fails. For every $j$, we can find a similarity transformation $\widehat u_j$ of $u$, a transported neck pair $(\zz_j,r_j)$, and a point
\[
p_j=(t_j,0,\ldots,0),\qquad |t_j|\ge1,
\]
such that
\[
\zz_j\in B_{3/4},\qquad r_j\le j^{-1}, \qquad r_j^{2-n}\int_{B_{2r_j}(\zz_j)\cap\{\widehat u_j>0\}} \cK_{p_j}^2(\widehat u_j)\,dx<j^{-1}.
\]
Rescale once more at the transported neck radius:
\[
v_j(y):=\frac{\widehat u_j(\zz_j+r_jy)}{r_j}.
\]
The neck normalization \eqref{eq:def_rB} and \Cref{lem:local-comparability-of-neck-radii} are invariant under these similarities. Hence
\[
\int_{B_1\cap\{v_j>0\}}|D^2v_j|^n\,dy=\eta_0^n, \qquad \|D^2v_j\|_{L^\infty(B_4\cap\{v_j>0\})}\le C.
\]

Set
\[
q_j:=\frac{\zz_j-p_j}{|t_j|}\in B_{7/4}\cap \{|x_1|\ge 1/4\}, \qquad \varepsilon_j:=\frac{r_j}{|t_j|}\to 0.
\]
In the rescaled variables the dilation family is
\[
\widehat J_j:= \left( q_j\cdot\nabla v_j +\varepsilon_j(y\cdot\nabla v_j-v_j), \partial_2v_j,\ldots,\partial_{n-1}v_j \right).
\]
Using the projective defect $\cK_\Phi$ from \Cref{lem:finite-jacobi-sternberg-zumbrun}, the exact scaling identity is
\[
\int_{B_2\cap\{v_j>0\}}\cK_{\widehat J_j}^2\,dy\to0.
\]
After passing to a subsequence,
\[
q_j\to q_\infty\in\overline{B_{7/4}}\cap\{|x_1|\ge1/4\}.
\]
Arguing as in the proof of \Cref{lem:positive-excess-at-the-neck-scale}, we obtain a global classical stable limit $v_\infty$ satisfying
\begin{equation}\label{eq:positive-dilation-gap-retained-hessian-mass}
\int_{B_1\cap\{v_\infty>0\}}|D^2v_\infty|^n\,dy=\eta_0^n,
\end{equation}
and $\widehat J_j\to\widehat J_\infty$ in $C^1_{\rm loc}(\{v_\infty>0\})$, where
\[
\widehat J_\infty:= \bigl(q_\infty\cdot\nabla v_\infty, \partial_2v_\infty,\ldots,\partial_{n-1}v_\infty\bigr).
\]
The lower-semicontinuity argument at the end of the proof of \Cref{lem:neck-scale-compactness} gives
\[
\cK_{\widehat J_\infty}^2\equiv0 \qquad\text{in }B_2\cap\{v_\infty>0\}.
\]
Since $|(q_\infty)_1|\ge1/4$, the argument in the proof of \Cref{lem:zero-defect-two-dimensional-reduction}, with $F$ replaced by $\widehat J_\infty$, shows that $v_\infty$ is two-dimensional on every positivity component meeting $B_1$. Repeating the final paragraph of the proof of \Cref{lem:positive-excess-at-the-neck-scale} contradicts \eqref{eq:positive-dilation-gap-retained-hessian-mass}.
\end{proof}

Combining \Cref{lem:asymmetric-plane-dichotomy,lem:positive-excess-at-the-neck-scale,lem:positive-dilation-gap-at-the-neck-scale} now yields a quantitative bound for the total neck charge.

\begin{prop}
\label{prop:flat-scale-asymmetric-planes-control-neck-radii}
With $R_\flat$ as in \eqref{eq:flat-scale-definition}, after increasing $k$, for every center $\zz\in\cZ\cap B_{\Rk/16}(\zk)$ we have
\begin{equation}\label{eq:flat-scale-asymmetric-planes-neck-radius-bound}
\varrho_\zz(R_\flat) \le C\epk^{1+\cttc/6}.
\end{equation}
\end{prop}

\begin{proof}
Let $a_\pm^\flat,b_\pm^\flat$ be the coefficients supplied by \Cref{prop:asymmetric-excess-decay-almost-critical} at $\zz$, and let $L_\pm^\flat$ be the corresponding affine functions. Write $\theta_\flat:=\epk^{1+\cttc/6}.$

Translate $\zz$ to the origin and rescale $R_\flat$ to one, retaining the same notation for the solution, signed domains, affine maps, and neck data. The threshold identities \eqref{eq:def_rB}, the packing properties in \Cref{defn:neck-centers}, and the neck-scale gaps in \Cref{lem:positive-excess-at-the-neck-scale,lem:positive-dilation-gap-at-the-neck-scale} are invariant under this change of variables. The required estimate is
\[
\sum_{\zeta\in\cZ_{1/2}\cap B_{1/2}}\rb(\zeta)^{n-2} \le C(u)\theta_\flat^2.
\]
The displayed sum is exactly $\varrho_0(1)^2$ in the normalized variables, and every center in the sum has $\rb(\zeta)\le1/2$.

By \Cref{prop:asymmetric-excess-decay-almost-critical},
\begin{equation}\label{eq:flat-plane-radius-w12-input}
\sum_{\ast=\pm}\int_{U_\ast\cap B_{8}} \Bigl(|u-L_\ast^\flat|^2+|\nabla u-a_\ast^\flat|^2\Bigr)\,dx \le C\theta_\flat^2 .
\end{equation}

For large $k$,
\[
B_{4R_\flat}(\zz)\subset B_{\Rk/8}(\zk),
\]
because $|\zz-\zk|\le\Rk/16$ and $R_\flat/\Rk=\epk^\cttc\to0$. The part of $\{u>0\}\cap B_4$ not covered by $U_+\cup U_-$ is contained in the terminal-neck region $\mathcal X_{\rm neck}$. Since $R_\flat=\epk^\cttc\Rk$, the terminal-neck volume estimate \eqref{eq:fb-tree-terminal-neck-boundary-localization-bookkeeping} gives, in the present normalized variables,
\[
|\mathcal X_{\rm neck}\cap B_4| \le C_\alpha\epk^{1+\gamma_\alpha-n\cttc} =C_\alpha\epk^{2+2/\alpha-n\cttc},
\]
where we used \eqref{eq:gamma-alpha}. The selected range of $\alpha$ and the choice \eqref{eq:asym-decay-flat-scale-chi-smallness} give $\frac2\alpha-\left(n+\frac13\right)\cttc \ge \frac{99}{101}>0.$ Thus, after increasing $k$,
\begin{equation}\label{eq:flat-plane-radius-exceptional-volume}
|\mathcal X_{\rm neck}\cap B_4|\le C\theta_\flat^2 .
\end{equation}
Also, \Cref{lem:fb-tree-small-neck-radii-near-flat-active-balls} gives, in these normalized variables,
\begin{equation}\label{eq:flat-plane-radius-small-necks}
\rb(\zeta)\le C\epk\Rk/R_\flat=C\epk^{1-\cttc}\to0 \quad\text{for}\quad \zeta\in\cZ\cap B_1.
\end{equation}
After increasing $k$, every such radius is smaller than the fixed small constants required below, and
\[
B_{2\rb(\zeta)}(\zeta)\subset B_{3/4} \qquad\text{for every }\zeta\in\cZ_{1/2}\cap B_{1/2}.
\]
Finally, the construction in \Cref{defn:neck-centers} makes the threshold balls within each dyadic neck-radius class disjoint, so their doubles have dimensional overlap; doubles from different classes are disjoint by the recursive exclusion in \eqref{eq:Xkprime}. Hence
\[
\sum_{\zeta\in\cZ_{1/2}\cap B_{1/2}} \mathbf1_{B_{2\rb(\zeta)}(\zeta)}\le N_n .
\]

Choose $\eta\in C_c^\infty(B_4)$ with $\eta\equiv1$ on $B_{3/4}$. First suppose that either \Cref{lem:asymmetric-plane-dichotomy}-(a), \eqref{eq:asymmetric-plane-good-alternative}, holds, or the separated zero hyperplanes are parallel. In the first case,
\[
|a_+^\flat+a_-^\flat|\le C\theta_\flat .
\]
In the second case, the root-direction bounds \eqref{eq:asymmetric-excess-flat-scale-normalization} exclude $a_-^\flat=a_+^\flat$ for large $k$, and hence $a_-^\flat=-a_+^\flat$. Thus the displayed bound holds in both cases.

Set $e:=a_+^\flat$. The tangential Sternberg--Zumbrun inequality in \Cref{lem:tangential-sternberg-zumbrun-inequality} gives
\[
\int_{B_{3/4}\cap\{u>0\}}\cA_e^2(u)\,dx \le C\int_{B_4\cap\{u>0\}}|\nabla^e u|^2\,dx .
\]
On $U_+$ the integrand is controlled by $|\nabla u-a_+^\flat|^2$; on $U_-$ it is controlled by $|\nabla u-a_-^\flat|^2+|a_+^\flat+a_-^\flat|^2$. On the remaining terminal-neck region we use the Lipschitz bound. Consequently, \eqref{eq:flat-plane-radius-w12-input} and \eqref{eq:flat-plane-radius-exceptional-volume} give
\begin{equation}\label{eq:flat-plane-radius-transverse-gap}
\int_{B_{3/4}\cap\{u>0\}}\cA_e^2(u)\,dx\le C\theta_\flat^2 .
\end{equation}

It remains to consider the case in which matching fails and the separated zero hyperplanes are nonparallel. Let
\[
P:=\{L_+^\flat=0\}\cap\{L_-^\flat=0\}
\]
be their affine codimension-two intersection. The separated alternative, \Cref{lem:asymmetric-plane-dichotomy}-(b) gives $P\cap B_4=\varnothing$. Rotate the two-dimensional normal plane to $P$ so that the closest point of $P$ to the origin is
\[
p=(t,0,\ldots,0),\qquad |t|\ge4,
\]
and
\[
P=p+\operatorname{span}\{e_2,\ldots,e_{n-1}\}.
\]
For the affine model $L_\ast^\flat=a_\ast^\flat\cdot y+b_\ast^\flat$, this means
\[
a_{\ast,i}^\flat=0\quad\text{for}\quad 2\le i\le n-1, \qquad a_\ast^\flat\cdot p+b_\ast^\flat=0, \qquad \ast\in\{+,-\}.
\]
Consequently the dilation Jacobi field $A_p$ and the derivatives $\partial_i$, $2\le i\le n-1$, vanish on the affine two-plane model. More precisely, on $U_\ast\cap B_4$, the vector $J_p$ from \Cref{lem:dilation-sternberg-zumbrun-inequality} satisfies
\[
|J_p|^2 \le C\Bigl(|u-L_\ast^\flat|^2+|\nabla u-a_\ast^\flat|^2\Bigr)
\]
with $C$ universal because $|t|\ge4$ and $x\in B_4$. Applying \eqref{eq:finite-jacobi-sternberg-zumbrun} to $J_p$ with the cutoff $\eta$, and using \eqref{eq:flat-plane-radius-w12-input} and \eqref{eq:flat-plane-radius-exceptional-volume}, gives
\begin{equation}\label{eq:flat-plane-radius-dilation-gap}
\int_{B_{3/4}\cap\{u>0\}}\cK_p^2(u)\,dx \le C\theta_\flat^2 .
\end{equation}
Here the exceptional contribution is estimated using only the bound $|J_p|\le C$ on $\mathcal X_{\rm neck}\cap B_4$: after the present normalization, $|t|\ge4$, $|x-p|/|t|\le C$ for $x\in B_4$, $|\nabla u|\le1$, and $u(0)=0$ with the Lipschitz bound gives $|u|\le C$ in $B_4$.

For every $\zeta\in\cZ_{1/2}\cap B_{1/2}$, the neck-scale gap \Cref{lem:positive-excess-at-the-neck-scale} in the first two cases and \Cref{lem:positive-dilation-gap-at-the-neck-scale} in the last gives
\[
c\,\rb(\zeta)^{n-2}\le
\begin{cases}
\displaystyle \int_{B_{2\rb(\zeta)}(\zeta)\cap\{u>0\}}\cA_e^2(u)\,dx, &\text{in the matching or parallel cases},\\[6pt]
\displaystyle \int_{B_{2\rb(\zeta)}(\zeta)\cap\{u>0\}}\cK_p^2(u)\,dx, &\text{in the nonparallel case}.
\end{cases}
\]
The dilation gap in \Cref{lem:positive-dilation-gap-at-the-neck-scale} is applicable because $\zeta\in B_{1/2}$, $|t|\ge4$, and \eqref{eq:flat-plane-radius-small-necks} makes $\rb(\zeta)$ smaller than its threshold. The rigid change of coordinates used above preserves transported neck pairs. Summing and using bounded overlap, followed by the global upper bound \eqref{eq:flat-plane-radius-transverse-gap} or \eqref{eq:flat-plane-radius-dilation-gap} then gives
\[
\sum_{\zeta\in\cZ_{1/2}\cap B_{1/2}}\rb(\zeta)^{n-2} \le C(u)\theta_\flat^2 .
\]

Scaling back by $R_\flat$ gives \eqref{eq:flat-scale-asymmetric-planes-neck-radius-bound}.
\end{proof}

\begin{cor}
\label{cor:flat-scale-neck-charge-improves-critical-exponent}
For the fixed exponent $\cttc>0$ above,
\[
\alpha_\star\le \frac{1}{1+\cttc/6}<1.
\]
\end{cor}

\begin{proof}
Use the selected sequence and the flat-scale objects above. In particular, $R_\flat=\epk^\cttc\Rk$.

Apply \Cref{prop:optimized-denominator-logarithmic-lower-bounds} with $R=4R_\flat$ and the selected center $\zk$. It gives
\[
\bbE_{\zk}(4R_\flat) \ge \frac{\eps_k}{C(1+\log(R_k/R_\flat))}.
\]
Since $\epk\le \eps_k$, and \eqref{eq:optimized-denominator-selected-scale-comparable} gives $R_k\le C\Rk$, while $R_\flat=\epk^\cttc\Rk$, we have
\[
1+\log(R_k/R_\flat)\le C|\log\epk|
\]
after increasing $k$. Hence
\[
\bbE_{\zk}(4R_\flat) \ge \frac{\epk}{C|\log\epk|}.
\]
On the other hand, \Cref{prop:flat-scale-asymmetric-planes-control-neck-radii} gives
\[
\varrho_{\zk}(R_\flat)\le C(u)\epk^{1+\cttc/6}.
\]
We may evaluate the selection functionals at $R=R_\flat/2$. Indeed, \Cref{lem:refined-center-and-scale-selection} gives $\frac{\rb(\zk)}{R_\flat} \le C\epk^{2/(n-2)-\cttc}\to0,$ so $\zk\in\cZ_{R_\flat/2}$. The inequalities $\rb(\zk)\ge\rmin$ and $\rb(\zk)\le C\Rk\epk^{2/(n-2)}$ also give
\[
R_\flat=\epk^\cttc\Rk \ge c\,\epk^{\cttc-2/(n-2)}\to\infty,
\]
because $\cttc<2/(n-2)$. In the definition of $F_u^s(R_\flat/2)$, the numerator at the test center $\zk$ is $\bbE_{\zk}(4R_\flat)$ and the denominator is $\varrho_{\zk}(R_\flat)^s$. Therefore, for every fixed $s>1/(1+\cttc/6)$,
\[
F_u^s(R_\flat/2) \ge c(u)\,\frac{\epk}{|\log\epk|} \frac{1}{\epk^{s(1+\cttc/6)}} \to+\infty.
\]
Hence $\limsup_{R\to\infty}F_u^s(R)=+\infty$. Since it holds for every $s>1/(1+\cttc/6)$, the definition of $\alpha_\star$ gives
\[
\alpha_\star\le \frac{1}{1+\cttc/6},
\]
as desired.
\end{proof}

\section{Final contradiction argument}
\label{sec:final-contradiction-argument}

From now on, we fix $n=4$, and we remain in the standing contradiction setup of \Cref{ass:standing-global-assumptions}. \Cref{prop:critical-dimension-bigger-than-four} ensures that $4<\nAC$, as required by that setup. \Cref{cor:flat-scale-neck-charge-improves-critical-exponent} rules out the branch $\alpha_\star=1$ in \eqref{eq:single-selection-alpha-choice}. Consequently, the exponent fixed before the selection satisfies
\[
\alpha=\frac{1+\max\{\alpha_\star,\frac12\}}2\in\left(\frac12,1\right).
\]
Define now
\begin{equation}\label{eq:below-critical-second-flux-exponent}
p_2:=1+\frac1{2\alpha}.
\end{equation}
Then $p_2>3/2$.

\subsection{An above-critical flux exponent}

\begin{remark}
\label{rem:above-critical-flux-exponent}
The reserve \eqref{eq:single-selection-flat-scale-reserves} gives
\begin{equation}\label{eq:below-critical-parameter-gaps}
2-\frac3{p_2} =\frac{2(1-\alpha)}{2\alpha+1}>10\cttc, \qquad 2p_2-3=\frac{1-\alpha}{\alpha} >2-\frac3{p_2}>10\cttc.
\end{equation}
Moreover, in dimension four, $1<p_2\alpha=\alpha+1/2<3/2=(n-1)/(n-2)$, while $\gamma_\alpha-2p_2=2p_2-3>0$. Thus $p_2$ is flux-admissible in \eqref{eq:flux-admissible}.
\end{remark}

For the rest of the section, we use $p_2$ and fix
\begin{equation}
\label{eq:sigma-def} \sigma:=\cttc/100.
\end{equation}
Recalling \eqref{eq:flat-scale-definition}, set
\begin{equation}\label{eq:below-critical-lower-endpoint}
A_\#:=\epk^{1+\sigma}R_\flat .
\end{equation}

The proof now has three steps. First we verify the estimates from \Cref{sec:fb-tree-quantitative} at every scale between $A_\#$ and $R_\flat$. We then replace the two flat-scale maps by a centered opposite pair and estimate the combined excess at a slightly power-smaller scale. Finally we compare this upper bound with the logarithmic lower bound from the optimized selection in \Cref{prop:optimized-denominator-logarithmic-lower-bounds}.

\subsection{Frozen-plane decay}

We now freeze a single direction across the intermediate scales and derive uniform signed-sheet oscillation and energy bounds.

\begin{prop}[Oscillation decay from $R_\flat$ to $A_\#$]
\label{prop:below-critical-frozen-plane-oscillation}
After increasing $k$, there is $e_k\in\mathbb S^3$, with $|e_k-e_{\rm root}|\le C\epk$, such that the centered opposite maps
\[
L_{k,+}^\flat(x):=e_k\cdot(x-\zk), \qquad L_{k,-}^\flat(x):=-e_k\cdot(x-\zk)
\]
satisfy, for each sign and every $r\in[A_\#,R_\flat]$,
\begin{equation}\label{eq:below-critical-frozen-plane-oscillation-w12}
\fint_{U_\pm\cap B_r(\zk)} \bigl|u-L_{k,\pm}^\flat\bigr|^2\,dx+r^2 \fint_{U_\pm\cap B_{3r/4}(\zk)} \bigl|\nabla(u-L_{k,\pm}^\flat)\bigr|^2\,dx \le C A_\#^2 .
\end{equation}
Finally,
\begin{equation}\label{eq:below-critical-centered-vee-lower-bound}
u\ge |e_k\cdot(x-\zk)|-CA_\# \qquad\text{in }B_{5R_\flat}(\zk).
\end{equation}
\end{prop}

\begin{proof}
Set $\zz:=\zk$. Let $L_\pm^\flat=a_\pm^\flat\cdot x+b_\pm^\flat$ be the flat-scale maps supplied by \Cref{prop:asymmetric-excess-decay-almost-critical} at $\zk$. We first work on $U_+$, writing $L:=L_+^\flat$; the same argument applies on $U_-$. The flat-scale estimate in \Cref{prop:asymmetric-excess-decay-almost-critical} is $o(A_\#)$, because $\cttc/3>4\sigma$. This starts the iteration below.

Fix $\beta_{\rm osc}\in(0,1/2)$ so that
\begin{equation}\label{eq:below-critical-proof-beta-choice}
2\beta_{\rm osc}(1+\sigma)<6\cttc-4\sigma,
\end{equation}
and set
\[
r_i:=2^{-i}R_\flat,\qquad H(r):=A_\#\left(\frac r{R_\flat}\right)^{\beta_{\rm osc}},\qquad H_i:=H(r_i).
\]
As in Step~2 of the proof of \Cref{prop:asymmetric-excess-decay-almost-critical}, we prove a one-step statement. The constants $j_0$ and $C_{\rm ind}\ge1$ will be chosen below. Suppose that $j\ge j_0$, $r_{j+1}\ge A_\#$, and that constants $c_0,\ldots,c_j$, with $c_0=0$, satisfy
\begin{equation}\label{eq:below-critical-proof-induction}
\left( \fint_{U_+\cap B_{r_i}(\zz)}|u-L-c_i|^2\,dx \right)^{1/2} \le C_{\rm ind}H_i \qquad\text{for}\quad 0\le i\le j.
\end{equation}

\medskip
\noindent\emph{Step 1: Slab trapping at the inductive scales.}
Set
\begin{equation}\label{eq:below-critical-neck-scale}
r_{\rm neck}:=C_\flat\epk^{1+\cttc/6}R_\flat =C_\flat\epk^{1+7\cttc/6}\Rk,
\end{equation}
where $C_\flat$ is fixed sufficiently large. By \Cref{lem:fb-tree-small-neck-radii-near-flat-active-balls}, $\rb(\zz)\le C\epk\Rk=o(R_\flat)$, so $\zz$ occurs in the sum defining $\varrho_\zz(R_\flat)^2$. Hence \eqref{eq:flat-scale-asymmetric-planes-neck-radius-bound} gives, after increasing $C_\flat$,
\begin{equation}\label{eq:below-critical-proof-local-neck-ratio}
\rb(\zz)\le r_{\rm neck},\qquad 
\frac{r_{\rm neck}}{A_\#} \le C\epk^{\cttc/6-\sigma}\to0.
\end{equation}
In particular, $\rb(\zz)/r_i=o(1)$ at every inductive scale. For $A_\#\le R\le R_\flat$, let $\nu_R$ minimize $\bbE_\zz(8R;\cdot)$, oriented toward $U_+$, and put
\begin{equation}\label{eq:below-critical-proof-uniform-slab-error}
E_R:=C\bbE_\zz(8R;\nu_R) \le C\omega_*\!\left(C\frac{r_{\rm neck}}R\right) \le\Theta_k:=C\omega_*\!\left(C\frac{r_{\rm neck}}{A_\#}\right) \to0.
\end{equation}
By \Cref{lem:fb-tree-mesoscopic-signed-slab-flatness},
\begin{equation}\label{eq:below-critical-proof-variable-local-slab}
\begin{aligned}
B_R(\zz)\cap\{\pm\nu_R\cdot(x-\zz)>E_RR\} \subset U_\pm\cap B_R(\zz) \subset\{\pm\nu_R\cdot(x-\zz)>-E_RR\}.
\end{aligned}
\end{equation}
In particular, $|U_\pm\cap B_R(\zz)|\simeq R^n$ uniformly on this range. The induction at two consecutive scales and \Cref{lem:affine-comparison-proportional} therefore give
\begin{equation}\label{eq:below-critical-proof-jumps}
|c_i-c_{i-1}|\le C C_{\rm ind}H_{i-1}\quad\text{for}\quad 1\le i\le j, \qquad |c_j-c_i|\le C C_{\rm ind}H_i\quad\text{for}\quad 0\le i\le j, \qquad |c_j|\le C C_{\rm ind}A_\#.
\end{equation}
Consequently,
\begin{equation}\label{eq:below-critical-proof-intercept}
\frac{|c_j|}{\epk\Rk} \le C C_{\rm ind}\epk^{\cttc+\sigma}\to0
\end{equation}
once $C_{\rm ind}$ is fixed.

We now improve this coarse control at each scale. Let $r_j\le R\le R_\flat$, choose $i\le j$ so that $R\le r_i<2R$, and set
\[
Q_R:=B_{R/8}\!\left(\zz+\frac R4\nu_R\right).
\]
For large $k$, the first inclusion in \eqref{eq:below-critical-proof-variable-local-slab} gives $Q_R\subset U_+\cap B_R(\zz)$. On $Q_R$, $V_{\zz,\nu_R}=\nu_R\cdot(x-\zz)$; hence \eqref{eq:below-critical-proof-induction}, \Cref{defn:monneau-excess}, and $|Q_R|\simeq R^n$ give
\[
\left( \fint_{Q_R}|L+c_i-\nu_R\cdot(x-\zz)|^2\,dx \right)^{1/2} \le C C_{\rm ind}H_i+C E_RR.
\]
Using \eqref{eq:below-critical-proof-jumps} and \Cref{lem:affine-comparison-proportional}, we obtain the two required conclusions in one estimate:
\begin{equation}\label{eq:below-critical-proof-direction-comparison}
\begin{aligned}
|a_+^\flat-\nu_R|+\frac{|L(\zz)+c_j|}{R} &\le C C_{\rm ind}\frac{H(R)}R+C E_R \le C C_{\rm ind}\epk^{\beta_{\rm osc}(1+\sigma)} +C\omega_*\!\left(C\epk^{\cttc/6-\sigma}\right) =:\vartheta_k\to0.
\end{aligned}
\end{equation}
Thus the induction first controls the accumulated displacement, while \eqref{eq:below-critical-proof-direction-comparison} shows that the affine zero plane of $L+c_j$ is in fact $o(R)$-centered and identifies $a_+^\flat$ with $\nu_R$ at every inductive scale. After increasing $\vartheta_k$ by a dimensional factor, combining \eqref{eq:below-critical-proof-variable-local-slab} and \eqref{eq:below-critical-proof-direction-comparison} gives
\begin{equation}\label{eq:below-critical-proof-uniform-local-slab}
\begin{aligned}
B_R(\zz)\cap\{\pm a_+^\flat\cdot(x-\zz)>\vartheta_kR\} \subset U_\pm\cap B_R(\zz) \subset\{\pm a_+^\flat\cdot(x-\zz)>-\vartheta_kR\}
\end{aligned}
\qquad\text{for}\quad r_j\le R\le R_\flat.
\end{equation}

\medskip
\noindent\emph{Step 2: Flux control and one-step improvement.}

As in Section~11, we use a dyadic compactness iteration, but here we improve oscillation (a $C^{\beta_{\rm osc}}$ iteration) rather than flatness (a $C^{1,\gamma}$ iteration). Accordingly, we subtract constants rather than planes. Because we ask only for H\"older decay of the oscillation, the required flux estimates are weaker; after proving $\alpha<1$, the above-critical flux exponent $p_2$ from \Cref{rem:above-critical-flux-exponent} provides them down to the target scale $A_\#$.

The point that must be checked anew is that the flux remains negligible on the whole interval $[A_\#,R_\flat]$. Set
\[
\lambda_{\rm osc}:= 6\cttc-4\sigma-2\beta_{\rm osc}(1+\sigma)>0, \qquad \eta_k:=\epk^{\lambda_{\rm osc}/2}.
\]
If $r=\epk^s\Rk$, then $\cttc\le s\le1+\cttc+\sigma$. Both exponents below decrease with $s$, and equations \eqref{eq:below-critical-parameter-gaps} and \eqref{eq:below-critical-proof-beta-choice} give
\begin{equation}\label{eq:below-critical-proof-flux-exponents}
\begin{aligned}
1+s(1-3/p_2)-2\cttc-2\sigma -2\beta_{\rm osc}(s-\cttc) &>\lambda_{\rm osc},\\
2p_2-1-2s-2\cttc-2\sigma -2\beta_{\rm osc}(s-\cttc) &>\lambda_{\rm osc}.
\end{aligned}
\end{equation}
Consequently, \eqref{eq:direct-sided-phi-flux-bound} and the invariance of $\Phi_r$ under constant shifts give, uniformly on this range,
\begin{equation}\label{eq:below-critical-proof-flux}
\frac{\Phi_r(u-L-c_j)}{H(r)^2} =\frac{\Phi_r(u-L)}{H(r)^2} \le C\epk^{\lambda_{\rm osc}} =C\eta_k^2.
\end{equation}

With $\delta_{\rm cmp}$ to be chosen below, assume $2^{j_0}\ge32\delta_{\rm cmp}^{-1}$, and set
\[
v_j(y):=\frac{u(\zz+r_jy)-L(\zz+r_jy)-c_j}{H_j}, \qquad \Omega_j:=r_j^{-1}(U_+-\zz).
\]
For $1\le\rho\le\delta_{\rm cmp}^{-1}$, choose $i\le j$ so that
\[
2\rho r_j\le r_i<4\rho r_j.
\]
Then \eqref{eq:below-critical-proof-induction}, \eqref{eq:below-critical-proof-jumps}, and \eqref{eq:below-critical-proof-uniform-local-slab} give
\begin{equation}\label{eq:below-critical-proof-l2-input}
\left( \fint_{U_+\cap B_{2\rho r_j}(\zz)} |u-L-c_j|^2\,dx \right)^{1/2} \le C C_{\rm ind}H_i \le K C_{\rm ind}\rho^{\beta_{\rm osc}}H_j.
\end{equation}
Here we used $H_i/H_j\le(4\rho)^{\beta_{\rm osc}}$. Since
\[
4A_\#\le2\rho r_j\le2\delta_{\rm cmp}^{-1}r_j\le R_\flat/16,
\]
equations \eqref{eq:asymmetric-excess-flat-scale-normalization} and \eqref{eq:below-critical-proof-intercept} give the required affine normalization. It remains only to verify the terminal-radius condition in \Cref{prop:direct-sided-w12-from-flux}. Indeed, let $A_\#\le R\le R_\flat/16$, let $N\in\cN$ meet $B_{3R/4}(\zz)$, and choose $\zeta_N\in\cZ\cap\lsup{4}{N}$. The estimates in \Cref{lem:fb-tree-nonfinal-neck-terminals-neck-scale,lem:fb-tree-small-neck-radii-near-flat-active-balls} give
\[
r_N+\rb(\zeta_N)\le C_\delta\epk\Rk=o(R_\flat), \qquad |\zeta_N-\zz|\le \frac34R+5r_N<R_\flat/2.
\]
Thus $\zeta_N$ occurs in the same neck-charge sum as $\zz$, and \eqref{eq:flat-scale-asymmetric-planes-neck-radius-bound} gives
\begin{equation}\label{eq:below-critical-proof-terminal-neck-control}
r_N\le C_\delta r_{\rm neck}\le c_JR.
\end{equation}
Here the last inequality follows from \eqref{eq:below-critical-proof-local-neck-ratio}. Therefore \Cref{prop:direct-sided-w12-from-flux}, applied with $R=2\rho r_j$, and \eqref{eq:below-critical-proof-l2-input} and \eqref{eq:below-critical-proof-flux} give
\begin{equation}\label{eq:below-critical-proof-growth}
\begin{aligned}
&\left(\fint_{\Omega_j\cap B_\rho}|v_j|^2\,dy\right)^{1/2} +\rho\left( \fint_{\Omega_j\cap B_\rho}|\nabla v_j|^2\,dy \right)^{1/2} +\left(\fint_{\Omega_j\cap B_\rho}|v_j|^3\,dy\right)^{1/3} \\
&\hspace{30mm}\le \bigl(KC_{\rm ind}+C(C_{\rm ind})\eta_k\bigr) \rho^{\beta_{\rm osc}} \qquad\text{for}\quad 1\le\rho\le\delta_{\rm cmp}^{-1},
\end{aligned}
\end{equation}
where $K$ is independent of $C_{\rm ind}$ and $\delta_{\rm cmp}$. Similarly, \eqref{eq:scaled-weak-residual-from-flux}, with $S=2\rho r_j$, gives
\begin{equation}\label{eq:below-critical-proof-residual}
\left| \int_{\Omega_j}\nabla v_j\cdot\nabla\varphi\,dy \right| \le C\eta_k^2\rho^{n-2+2\beta_{\rm osc}}\|\varphi\|_{L^\infty} =C\eta_k^2\rho^{2+2\beta_{\rm osc}}\|\varphi\|_{L^\infty} \qquad \quad\text{for}\quad \varphi\in C_c^\infty(B_\rho),\ 1\le\rho\le\delta_{\rm cmp}^{-1}.
\end{equation}
These are precisely the $N=0$ counterparts of \eqref{eq:alt-flat-growth} and \eqref{eq:alt-flat-residual}. Finally, \eqref{eq:below-critical-proof-uniform-local-slab}, at $R=\rho r_j$, gives
\begin{equation}\label{eq:below-critical-proof-slab}
B_\rho\cap\{a_+^\flat\cdot y>\vartheta_k\rho\} \subset\Omega_j\cap B_\rho \subset B_\rho\cap\{a_+^\flat\cdot y>-\vartheta_k\rho\} \qquad\text{for}\quad 1\le\rho\le\delta_{\rm cmp}^{-1}.
\end{equation}

We now follow exactly the order of choices in Step 2 of the proof of \Cref{prop:asymmetric-excess-decay-almost-critical}. Choose $\varepsilon_{\rm cmp}>0$ so that
\begin{equation}\label{eq:below-critical-proof-tolerance}
(K+1)\varepsilon_{\rm cmp}\le2^{-1-\beta_{\rm osc}},
\end{equation}
let $\delta_{\rm cmp}$ be furnished by \Cref{lem:linear-neumann-compactness-global} for
\[
(m,q,N,\vartheta)=(4,3,0,\beta_{\rm osc}),
\]
choose $j_0$ with $2^{j_0}\ge32\delta_{\rm cmp}^{-1}$, and then fix $C_{\rm ind}\ge1$. Only after these choices do we increase $k$. The flat-scale estimate \eqref{eq:asymmetric-excess-flat-scale-affine-w12} and \eqref{eq:below-critical-proof-variable-local-slab} give
\begin{equation}\label{eq:below-critical-proof-base}
\max_{0\le i\le j_0} \frac1{H_i} \left( \fint_{U_+\cap B_{r_i}(\zz)}|u-L|^2\,dx \right)^{1/2} \le C2^{(2+\beta_{\rm osc})j_0}\epk^{\cttc/3-\sigma}.
\end{equation}
We require $r_{j_0+1}\ge A_\#$, that the right-hand side of \eqref{eq:below-critical-proof-base} be at most $C_{\rm ind}$, all applicability thresholds in \Cref{lem:direct-sided-local-flux-slab,prop:direct-sided-w12-from-flux}, and
\begin{equation}\label{eq:below-critical-proof-final-k}
C C_{\rm ind}\epk^{\cttc+\sigma}\le1, \qquad C(C_{\rm ind})\eta_k\le C_{\rm ind}, \qquad \frac{C\eta_k^2\delta_{\rm cmp}^{-\beta_{\rm osc}}} {(K+1)C_{\rm ind}}\le\delta_{\rm cmp}, \qquad \vartheta_k\le\delta_{\rm cmp}.
\end{equation}
Thus $c_i=0$, $0\le i\le j_0$, initializes \eqref{eq:below-critical-proof-induction}.

Suppose now that the induction has reached $j\ge j_0$ and $r_{j+1}\ge A_\#$. Equations \eqref{eq:below-critical-proof-growth}, \eqref{eq:below-critical-proof-residual}, \eqref{eq:below-critical-proof-slab}, and \eqref{eq:below-critical-proof-final-k} show that $((K+1)C_{\rm ind})^{-1}v_j$ satisfies \Cref{lem:linear-neumann-compactness-global} on $1\le\rho\le\delta_{\rm cmp}^{-1}$. Hence there is a constant $d_j$ such that
\begin{equation}\label{eq:below-critical-proof-output}
\left( \fint_{\Omega_j\cap B_{1/2}} \left|((K+1)C_{\rm ind})^{-1}v_j-d_j\right|^2\,dy \right)^{1/2} \le\varepsilon_{\rm cmp}.
\end{equation}
As in the estimate following \eqref{eq:alt-flat-output},
\[
|d_j|\le C\bigl((K+1)^{-1}+\varepsilon_{\rm cmp}\bigr), \qquad (K+1)|d_j|\le C.
\]
Set
\[
c_{j+1}:=c_j+(K+1)C_{\rm ind}H_jd_j.
\]
Then
\[
|c_{j+1}-c_j|\le C C_{\rm ind}H_j, \qquad \left( \fint_{U_+\cap B_{r_{j+1}}(\zz)}|u-L-c_{j+1}|^2\,dx \right)^{1/2} \le (K+1)C_{\rm ind}\varepsilon_{\rm cmp}H_j \le\frac12C_{\rm ind}H_{j+1}.
\]
This closes the induction.

\medskip
\noindent\emph{Step 3: Remove the shifts and conclude.}
The jumps in \eqref{eq:below-critical-proof-jumps} form a geometric series:
\begin{equation}\label{eq:below-critical-proof-constant-sum}
|c_j|\le C C_{\rm ind}\sum_{i<j}H_i\le C A_\#.
\end{equation}
Combining this with \eqref{eq:below-critical-proof-induction} gives, at every dyadic radius in the stated range,
\[
\left( \fint_{U_+\cap B_{r_j}(\zz)}|u-L|^2\,dx \right)^{1/2} \le C A_\#.
\]
For a non-dyadic radius, restrict \eqref{eq:below-critical-proof-induction} from the next larger dyadic ball and use \eqref{eq:below-critical-proof-uniform-local-slab}. Thus
\[
\left( \fint_{U_+\cap B_r(\zz)}|u-L_+^\flat|^2\,dx \right)^{1/2} \le CA_\# \qquad\text{for}\quad A_\#\le r\le R_\flat.
\]

To estimate $L_+^\flat(\zz)$, choose an index $j_\ast$ such that $r_{j_\ast}\in[A_\#,2A_\#]$; this radius belongs to $[A_\#,R_\flat]$. Since $u(\zz)=0$, the Lipschitz bound for $u$, the unit slope of $L$, volume comparability, and \eqref{eq:below-critical-proof-induction} give $|L(\zz)+c_{j_\ast}|\le C A_\#$.
Together with \eqref{eq:below-critical-proof-constant-sum}, this gives
\[
|L_+^\flat(\zz)|\le CA_\#.
\]

Finally, for $r\le R_\flat/16$, apply \eqref{eq:direct-sided-w12-from-flux} directly to $u-L$, using \eqref{eq:below-critical-proof-terminal-neck-control}. The just-proved $L^2$ estimate controls its first term, while \eqref{eq:below-critical-proof-flux} bounds its flux term by $CA_\#^2$. For the fixed range $r\in[R_\flat/16,R_\flat]$, use \eqref{eq:asymmetric-excess-flat-scale-affine-w12} and \eqref{eq:below-critical-proof-uniform-local-slab}. Hence
\[
r^2\fint_{U_+\cap B_{3r/4}(\zz)}|\nabla(u-L_+^\flat)|^2\,dx \le C A_\#^2,
\]
for every $r\in[A_\#,R_\flat]$. Applying the same argument on $U_-$ gives these three estimates with $+$ replaced by $-$.

It remains to center the two maps. Write
\[
\bar b_\pm:=L_\pm^\flat(\zz), \qquad g:=a_+^\flat+a_-^\flat.
\]
We have just proved
\begin{equation}\label{eq:below-critical-proof-original-intercepts}
|\bar b_+|+|\bar b_-|\le CA_\#.
\end{equation}
In alternative \eqref{eq:asymmetric-plane-good-alternative} of \Cref{lem:asymmetric-plane-dichotomy},
\[
|g|\le C\epk^{1+\cttc/6} \le C\epk^{1+\sigma}=C A_\#/R_\flat.
\]
If instead \eqref{eq:asymmetric-plane-separated-alternative} holds and $g\ne0$, take $x=\zz+3R_\flat g/|g|$. Since $|a_\pm^\flat|=1$,
\[
L_\pm^\flat(x)=\frac32R_\flat|g|+\bar b_\pm.
\]
The two values cannot both be positive, and \eqref{eq:below-critical-proof-original-intercepts} again gives $|g|\le CA_\#/R_\flat$. Thus this estimate holds in either alternative.

Set $e_k:=a_+^\flat$ and $L_{k,\pm}^\flat(x):=\pm e_k\cdot(x-\zk)$. Then
\[
\sup_{B_{5R_\flat}(\zz)} |L_\pm^\flat-L_{k,\pm}^\flat| +R_\flat |\nabla L_\pm^\flat-\nabla L_{k,\pm}^\flat| \le CA_\#.
\]
The triangle inequality now proves \eqref{eq:below-critical-frozen-plane-oscillation-w12}. The normalization \eqref{eq:asymmetric-excess-flat-scale-normalization} gives $|e_k-e_{\rm root}|\le C\epk$.

Finally, \eqref{eq:asymmetric-plane-max-lower-bound}, the preceding comparison, and $\epk^{1+\cttc/6}R_\flat\le A_\#$ yield
\[
u\ge \max\{L_{k,+}^\flat,L_{k,-}^\flat\}-CA_\# =|e_k\cdot(x-\zk)|-CA_\#
\]
in $B_{5R_\flat}(\zz)$, proving \eqref{eq:below-critical-centered-vee-lower-bound}.
\end{proof}

\subsection{Gaussian calibration}

We next convert the frozen-plane estimates into a below-flat-scale bound for the combined Gaussian excess.

\begin{lem}
\label{lem:combined-excess-below-flat-scale}
Let $e_k$ and $P_\pm:=L_{k,\pm}^\flat$ be given by \Cref{prop:below-critical-frozen-plane-oscillation}. Fix
\[
\tau:=\frac{\sigma}{100}, \qquad \rho_k:=\epk^\tau R_\flat.
\]
Then, after increasing $k$,
\[
\bbE_{\zk}(\rho_k;e_k)^2\le C_\alpha\epk^{2+\sigma}.
\]
\end{lem}

\begin{proof}
We divide the proof into three steps. Set
\[
\rho:=\epk^\tau R_\flat, \qquad T:=\frac{R_\flat}{4},
\]
and fix
\[
s\in(\zk+e_k^\perp)\cap B_\rho(\zk).
\]
Since $\rho=o(R_\flat)$, for $k$ sufficiently large we have
\[
B_T(s)\cup B_\rho(\zk) \subset B_{R_\flat/2}(\zk) \subset B_{\Rk/8}(\zk).
\]
For $\ast\in\{+,-\}$, write
\[
e_+:=e_k, \qquad e_-:=-e_k, \qquad P_\ast:=e_\ast\cdot(x-s), \qquad w_\ast:=u-P_\ast, \qquad H_\ast:=\{P_\ast>0\}.
\]
Since $s-\zk\perp e_k$, the maps $P_\pm$ coincide with $L_{k,\pm}^\flat$ from \Cref{prop:below-critical-frozen-plane-oscillation}. We denote by $\nu_\ast$ the exterior unit normal to $U_\ast$ and, for almost every $0<r<R_\flat/4$, set
\[
\mathcal W_{\ast,s}(r):= r^{-4}\int_{U_\ast\cap B_r(s)}(|\nabla u|^2+1)\,dx -r^{-5}\int_{U_\ast\cap\partial B_r(s)}u^2\,d\cH^3.
\]

\medskip\noindent\emph{Step 1: Fixed-radius and Gaussian-averaging
identities.} Arguing as in Step 3 of the proof of \cite{CFFS25}*{Lemma 9.2}, and retaining the artificial-boundary term, we obtain
\begin{equation}\label{eq:combined-excess-fixed-radius-phase-identity}
\begin{split}
\mathcal W_{\ast,s}(r)-\frac{|B_1|}{2} &=r^{-5}\int_{U_\ast\cap\partial B_r(s)} ((x-s)\cdot\nabla u-u)w_\ast\,d\cH^3 +r^{-5}\int_{\partial B_r(s)}P_\ast^2 (\mathbf1_{U_\ast}-\mathbf1_{H_\ast})\,d\cH^3 \\
&\quad+r^{-4}\int_{\Gamma_\ast^{\rm reg}\cap B_r(s)} P_\ast(1+e_\ast\cdot\nu_\ast)\,d\cH^3 +r^{-4}\int_{\Gamma_\ast^{\rm neck}\cap B_r(s)} \bigl[w_\ast\partial_{\nu_\ast}u +(u+P_\ast)e_\ast\cdot\nu_\ast\bigr]\,d\cH^3.
\end{split}
\end{equation}

By \Cref{lem:fb-tree-signed-decomposition}, $U_+\cap U_-=\varnothing$, while \eqref{eq:fb-tree-signed-cover} gives
\[
U_+\cup U_-\subset\{u>0\}\cap B_{\Rk/2}(\zk).
\]
Set
\[
Y:=\{u>0\}\setminus(U_+\cup U_-), \qquad {\bf W}_s(u,r):={\bf W}(u(s+\cdot),r).
\]
The phase partition then gives, for almost every $0<r<T$,
\begin{equation}\label{eq:combined-excess-whole-energy-splitting}
\begin{split}
{\bf W}_s(u,r) ={}&\sum_{\ast=\pm}\mathcal W_{\ast,s}(r) +r^{-4}\int_{Y\cap B_r(s)}(|\nabla u|^2+1)\,dx -r^{-5}\int_{Y\cap\partial B_r(s)}u^2\,d\cH^3.
\end{split}
\end{equation}

In dimension $4$, the coarea formula and an integration by parts give
\begin{equation}\label{eq:combined-excess-gaussian-spherical-average}
W_{\rho,s}(u) =\rho^{-6}\int_0^\infty r^5e^{-r^2/\rho^2} {\bf W}_s(u,r)\,dr, \qquad \rho^{-6}\int_0^\infty r^5e^{-r^2/\rho^2}\,dr=1.
\end{equation}
We shall also use $W_\infty=|B_1|$. Since $P_\ast=L_{k,\ast}^\flat$, applying \eqref{eq:below-critical-frozen-plane-oscillation-w12} with $r=R_\flat$, restricting to $B_{R_\flat/2}(\zk)$, and using $|U_\ast\cap B_{R_\flat}(\zk)|\le |B_{R_\flat}|$, we obtain
\begin{equation}\label{eq:combined-excess-endpoint-w12}
\sum_{\ast=\pm}\int_{U_\ast\cap B_{R_\flat/2}(\zk)} \bigl(w_\ast^2+R_\flat^2|\nabla w_\ast|^2\bigr)\,dx \le CA_\#^2R_\flat^4.
\end{equation}

For later use, we record the kernel associated with the averaging weight in \eqref{eq:combined-excess-gaussian-spherical-average}:
\[
\begin{aligned}
K_{\rho,T}(z) &:=\rho^{-6}\mathbf1_{\{|z|<T\}} \int_{|z|}^{T}r e^{-r^2/\rho^2}\,dr =\frac{1}{2\rho^4} \left(e^{-|z|^2/\rho^2}-e^{-T^2/\rho^2}\right) \mathbf1_{\{|z|<T\}} \le C\rho^{-4}e^{-|z|^2/\rho^2}.
\end{aligned}
\]

\medskip\noindent\emph{Step 2: Estimates for the four terms.}
We multiply \eqref{eq:combined-excess-fixed-radius-phase-identity} by the averaging weight, integrate over $0<r<T$, and sum over the two signs. By the coarea formula and Fubini's theorem, we obtain four terms:
\[
\begin{split}
\mathrm I & := \rho^{-6}\sum_{\ast=\pm} \int_{U_\ast\cap B_T(s)} e^{-|x-s|^2/\rho^2} ((x-s)\cdot\nabla u-u)w_\ast\,dx,\\
\mathrm {II} & := \rho^{-6}\sum_{\ast=\pm} \int_{B_T(s)}e^{-|x-s|^2/\rho^2}P_\ast^2 (\mathbf1_{U_\ast}-\mathbf1_{H_\ast})\,dx, \\
\mathrm {III} & := \sum_{\ast=\pm} \int_{\Gamma_\ast^{\rm reg}\cap B_T(s)} P_\ast(1+e_\ast\cdot\nu_\ast) K_{\rho,T}(x-s)\,d\cH^3, \\
\mathrm {IV} & := \sum_{\ast=\pm} \int_{\Gamma_\ast^{\rm neck}\cap B_T(s)} \bigl[w_\ast\partial_{\nu_\ast}u +(u+P_\ast)e_\ast\cdot\nu_\ast\bigr] K_{\rho,T}(x-s)\,d\cH^3.
\end{split}
\]

\medskip\noindent We first estimate $\mathrm I$.
Since $(x-s)\cdot\nabla P_\ast=P_\ast$, $(x-s)\cdot\nabla u-u=(x-s)\cdot\nabla w_\ast-w_\ast.$ Consequently, Young's inequality, the bound $|x-s|\le T\le R_\flat$, the inclusion $B_T(s)\subset B_{R_\flat/2}(\zk)$, and \eqref{eq:combined-excess-endpoint-w12}, give
\[
\begin{aligned}
|\mathrm I| &\le \rho^{-6}\sum_{\ast=\pm} \int_{U_\ast\cap B_T(s)} e^{-|x-s|^2/\rho^2} \bigl(|x-s|\,|\nabla w_\ast|\,|w_\ast|+w_\ast^2\bigr)\,dx \\
&\le C\rho^{-6}\sum_{\ast=\pm} \int_{U_\ast\cap B_T(s)} \bigl(R_\flat^2|\nabla w_\ast|^2+w_\ast^2\bigr)\,dx \le CA_\#^2R_\flat^4\rho^{-6}.
\end{aligned}
\]

\medskip\noindent We next consider $\mathrm {II}$. Notice that
$P_-^2=P_+^2$ and $\mathbf1_{H_+}+\mathbf1_{H_-}=1$. Since $U_+\cap U_-=\varnothing$ by \Cref{lem:fb-tree-signed-decomposition}, we obtain
\[
\mathrm {II} =-\rho^{-6}\int_{B_T(s)\setminus(U_+\cup U_-)} P_+^2e^{-|x-s|^2/\rho^2}\,dx\le0.
\]

On $\{u=0\}$, the lower bound \eqref{eq:below-critical-centered-vee-lower-bound} gives $|P_+|\le CA_\#$. Hence
\[
\begin{aligned}
\rho^{-6}\int_{\{u=0\}\cap B_T(s)} P_+^2e^{-|x-s|^2/\rho^2}\,dx &\le CA_\#^2\rho^{-6} \int_{\R^4}e^{-|x-s|^2/\rho^2}\,dx \le C\left(\frac{A_\#}{\rho}\right)^2.
\end{aligned}
\]
On the other hand, \eqref{eq:fb-tree-signed-cover} gives
\[
Y\cap B_{R_\flat/2}(\zk)\subset\mathcal X_{\rm neck}.
\]
Thus \Cref{lem:lipschitz-bound} and \eqref{eq:below-critical-centered-vee-lower-bound}, together with $A_\#\le\epk\Rk$, give the pointwise bound
\[
u+|P_\ast|\le C\epk\Rk \quad\text{on }Y\cap B_{R_\flat/2}(\zk), \qquad \ast\in\{+,-\}.
\]
Since $\frac{\epk\Rk}{\rho} =\epk^{1-\cttc-\tau}\to0,$ we conclude that
\[
0\le-\mathrm {II} \le C\left(\frac{A_\#}{\rho}\right)^2 +C\rho^{-4} \left|\mathcal X_{\rm neck}\cap B_{R_\flat/2}(\zk)\right|.
\]

\medskip\noindent We now estimate $\mathrm {III}$.
Since $u=0$ on $\Gamma_\ast^{\rm reg}\cap B_T(s)$, \eqref{eq:below-critical-centered-vee-lower-bound} gives $|P_\ast|\le CA_\#$. The bound for $K_{\rho,T}$ therefore gives
\[
|\mathrm {III}| \le CA_\#\rho^{-4} \sum_{\ast=\pm} \int_{\Gamma_\ast^{\rm reg}} |1+e_\ast\cdot \nu_\ast|\,d\cH^3.
\]
Set $e_+^0:=e_{\rm root}$ and $e_-^0:=-e_{\rm root}$. Since $e_\ast$, $e_\ast^0$, and $\nabla u$ are unit vectors, we have
\[
1-e_\ast\cdot \nabla u \le2(1-e_\ast^0\cdot \nabla u)+|e_\ast-e_\ast^0|^2,
\]
On $\Gamma_\ast^{\rm reg}$, $\nu_\ast=-\nabla u$, and hence
\[
1-e_\ast^0\cdot\nabla u =\left|\partial_{\nu_\ast}(u-e_\ast^0\cdot x)\right|.
\]
Using the preceding unit-vector inequality and H\"older's inequality, we obtain
\[
\begin{aligned}
\sum_{\ast=\pm} \int_{\Gamma_\ast^{\rm reg}} |1+e_\ast\cdot\nu_\ast|\,d\cH^3 &\le C\sum_{\ast=\pm} \cH^3(\Gamma_\ast^{\rm reg})^{1-1/p_2} \left( \int_{\Gamma_\ast^{\rm reg}} \left|\partial_{\nu_\ast} (u-e_\ast^0\cdot x)\right|^{p_2}\,d\cH^3 \right)^{1/p_2} +C|e_k-e_{\rm root}|^2 \sum_{\ast=\pm}\cH^3(\Gamma_\ast^{\rm reg}) \\
&\le C\Rk^3\epk^2.
\end{aligned}
\]
Recall from \Cref{rem:above-critical-flux-exponent} that $p_2$ is flux-admissible. In the last inequality we used \Cref{lem:perimeter-bound} and \Cref{prop:fb-tree-sided-higher-power-transverse-estimate}, the latter with $p=p_2$, together with \Cref{prop:below-critical-frozen-plane-oscillation}, which gives $|e_k-e_{\rm root}|\le C\epk$. Thus
\[
|\mathrm {III}| \le CA_\#\rho^{-4}\Rk^3\epk^2.
\]

\medskip\noindent Finally, we estimate $\mathrm {IV}$.
The kernel estimate gives
\[
|\mathrm {IV}| \le C\rho^{-4}\sum_{\ast=\pm} \int_{\Gamma_\ast^{\rm neck}\cap B_T(s)} \bigl(|w_\ast|\,|\partial_{\nu_\ast}u|+u+|P_\ast|\bigr)\,d\cH^3.
\]
By \Cref{lem:lipschitz-bound},
\[
|\partial_{\nu_\ast}u|\le1, \qquad |w_\ast|\le u+|P_\ast|.
\]
Moreover, \eqref{eq:fb-tree-terminal-neck-radius-bookkeeping} gives $u\le C\epk\Rk$ on $\Gamma_\ast^{\rm neck}\cap B_T(s)$, and hence \eqref{eq:below-critical-centered-vee-lower-bound} gives $|P_\ast|\le u+CA_\#\le C\epk\Rk$. Therefore
\[
|\mathrm {IV}| \le C\epk\Rk\,\rho^{-4} \cH^3\left( (\Gamma_+^{\rm neck}\cup\Gamma_-^{\rm neck}) \cap B_{R_\flat/2}(\zk) \right).
\]

\medskip\noindent\emph{Step 3: Gaussian remainder terms and the Monneau
estimate.} We first account for the remaining terms in \eqref{eq:combined-excess-whole-energy-splitting} and for the part of the Gaussian average corresponding to $r>T$. Define
\[
\begin{aligned}
\mathcal J_{\rm bulk} & :=\rho^{-6}\int_0^T r e^{-r^2/\rho^2} \int_{Y\cap B_r(s)}(|\nabla u|^2+1)\,dx\,dr\ge0,\\
\mathcal J_{\rm bdry} &:=\rho^{-6}\int_0^T e^{-r^2/\rho^2} \int_{Y\cap\partial B_r(s)}u^2\,d\cH^3\,dr =\rho^{-6}\int_{Y\cap B_T(s)} u^2e^{-|x-s|^2/\rho^2}\,dx,\\
\mathcal J_{\rm tail} & :=\rho^{-6}\int_T^\infty r^5e^{-r^2/\rho^2} \bigl(W_\infty- {\bf W}_s(u,r)\bigr)\,dr.
\end{aligned}
\]

Combining \eqref{eq:combined-excess-whole-energy-splitting} with \eqref{eq:combined-excess-gaussian-spherical-average}, we obtain
\[
\wW_{\rho,s} =-\mathrm I-\mathrm {II}-\mathrm {III}-\mathrm {IV} -\mathcal J_{\rm bulk}+\mathcal J_{\rm bdry} +\mathcal J_{\rm tail}.
\]
Since $\mathcal J_{\rm bulk}\ge0$, the term $-\mathcal J_{\rm bulk}$ can be omitted when taking an upper bound. The pointwise bound $u\le C\epk\Rk=o(\rho)$ on $Y\cap B_T(s)$, obtained above, gives
\[
\mathcal J_{\rm bdry} \le C\rho^{-4} \left|\mathcal X_{\rm neck}\cap B_{R_\flat/2}(\zk)\right|.
\]

For the tail, recall that $\zk\in\FB(u)$, and hence $u(\zk)=0$. By \Cref{lem:lipschitz-bound}, $|\nabla u|\le1$. Since $|s-\zk|\le\rho$, it follows that $u\le r+\rho$ on $\partial B_r(s)$. Consequently, for $r\ge T$,
\[
W_\infty- {\bf W}_s(u,r) \le W_\infty+r^{-5}\int_{\partial B_r(s)}u^2\,d\cH^3 \le C.
\]
It follows that
\[
\mathcal J_{\rm tail} \le C\rho^{-6}\int_T^\infty r^5e^{-r^2/\rho^2}\,dr \le Ce^{-c(R_\flat/\rho)^2}.
\]
Combining the estimates for $\mathrm I$--$\mathrm {IV}$ with the two remainder estimates, we obtain
\begin{equation}\label{eq:combined-excess-four-term-error}
\begin{aligned}
\wW_{\rho,s} \le{}& CA_\#^2R_\flat^4\rho^{-6} +CA_\#\rho^{-4}\Rk^3\epk^2 +C\rho^{-4} \left|\mathcal X_{\rm neck}\cap B_{R_\flat/2}(\zk)\right| \\
& \quad +C\epk\Rk\,\rho^{-4} \cH^3\left( (\Gamma_+^{\rm neck}\cup\Gamma_-^{\rm neck}) \cap B_{R_\flat/2}(\zk) \right) +Ce^{-c(R_\flat/\rho)^2}.
\end{aligned}
\end{equation}

It remains to show that every term on the right-hand side of \eqref{eq:combined-excess-four-term-error} has the required order. Since
\[
A_\#=\epk^{1+\sigma}R_\flat, \qquad R_\flat=\epk^\cttc\Rk, \qquad \rho=\epk^\tau R_\flat,
\]
the first two terms are respectively
\[
\epk^{2+2\sigma-6\tau} \qquad\text{and}\qquad \epk^{3+\sigma-3\cttc-4\tau}.
\]
The terminal-neck volume and boundary bounds \eqref{eq:fb-tree-terminal-neck-boundary-localization-bookkeeping} give
\[
\begin{aligned}
&\rho^{-4} \left|\mathcal X_{\rm neck}\cap B_{R_\flat/2}(\zk)\right| +\epk\Rk\,\rho^{-4} \cH^3\left( (\Gamma_+^{\rm neck}\cup\Gamma_-^{\rm neck}) \cap B_{R_\flat/2}(\zk) \right) \le C_\alpha\epk^{1+\gamma_\alpha-4\cttc-4\tau}.
\end{aligned}
\]
Finally,
\[
e^{-c(R_\flat/\rho)^2}=e^{-c\epk^{-2\tau}}.
\]
Recalling from \eqref{eq:gamma-alpha} that $\gamma_\alpha=1+2/\alpha$, and using $\sigma=\cttc/100$, $\tau=\sigma/100$, and \eqref{eq:asym-decay-flat-scale-chi-smallness}, we have
\[
6\tau<\sigma, \qquad 3\cttc+4\tau<1, \qquad 4\cttc+4\tau+\sigma<\gamma_\alpha-1.
\]
Thus every term on the right-hand side of \eqref{eq:combined-excess-four-term-error} is at most $C_\alpha\epk^{2+\sigma}$, and hence
\[
\wW_{\rho,s}\le C_\alpha\epk^{2+\sigma}
\]
uniformly for $s\in(\zk+e_k^\perp)\cap B_\rho(\zk)$.

It remains to estimate the Monneau term. The endpoint estimate \eqref{eq:combined-excess-endpoint-w12}, the centered vee lower bound \eqref{eq:below-critical-centered-vee-lower-bound}, and \eqref{eq:fb-tree-signed-cover} give
\[
M_{\rho,\zk}^{e_k} \le CA_\#^2R_\flat^4\rho^{-6}+C(A_\#/\rho)^2 +C\rho^{-4}|\mathcal X_{\rm neck}\cap B_{R_\flat/2}(\zk)| \le C_\alpha\epk^{2+\sigma}.
\]

Taking the supremum over $s$ in the Weiss estimate and combining it with the Monneau bound, we obtain the desired result, which concludes the proof.
\end{proof}

\subsection{End of the proof}

We now combine the preceding estimates to prove the classification theorem and its two corollaries.

\begin{proof}[Proof of \Cref{thm:main-classification}]
Suppose that the Hessian conclusion fails. By \Cref{lem:reduction,lem:selection-of-initial-center-and-scale,lem:refined-center-and-scale-selection}, we obtain the contradiction setup, selected sequence, and geometric objects fixed above. Use the exponent $\alpha$ chosen there. \Cref{cor:flat-scale-neck-charge-improves-critical-exponent} and \eqref{eq:single-selection-alpha-choice} give $\alpha<1$, while \Cref{rem:above-critical-flux-exponent} shows that the second flux exponent $p_2$ is admissible with the reserves used above. By \Cref{lem:combined-excess-below-flat-scale}, recalling \eqref{eq:flat-scale-definition} and taking
\[
\tau=\sigma/100,\qquad \rho_k=\epk^\tau R_\flat,
\]
there is $e_k\in\mathbb S^3$ such that
\[
\bbE_{\zk}(\rho_k) \le \bbE_{\zk}(\rho_k;e_k) \le C\epk^{1+\sigma/2}.
\]
On the other hand, \Cref{prop:optimized-denominator-logarithmic-lower-bounds} gives
\[
\bbE_{\zk}(\rho_k) \ge \frac{\eps_k}{C\bigl(1+\log(8R_k/\rho_k)\bigr)} .
\]
Combining the two estimates,
\begin{equation}\label{eq:final-optimized-denominator-contradiction}
\eps_k \le C\bigl(1+\log(8R_k/\rho_k)\bigr)\epk^{1+\sigma/2}.
\end{equation}
Here $\rho_k\le R_\flat\le\Rk\le R_k$ for large $k$, so the scale $\rho_k$ lies in the range of \Cref{prop:optimized-denominator-logarithmic-lower-bounds}.

We now estimate the logarithm. Since $\rho_k=\epk^{\cttc+\tau}\Rk$, and \eqref{eq:optimized-denominator-selected-scale-comparable} gives $\Rk/R_k\ge c>0$, we have
\[
1+\log(8R_k/\rho_k)\le C|\log\epk|.
\]
Moreover $\epk=\tilde\zeta_k^{\alpha\beta_\circ}\eps_k\le\eps_k$. Therefore \eqref{eq:final-optimized-denominator-contradiction} yields
\[
1\le C|\log\epk|\,\epk^{\sigma/2},
\]
after division by $\eps_k>0$. The right-hand side tends to zero, a contradiction. Hence every global classical stable solution in $\mathbb R^4$ satisfies $D^2u=0$ in each positivity component.
\end{proof}

With the classification established, we conclude by deriving its two stated corollaries.
\begin{proof}[Proof of \Cref{cor:main-local-hessian}]
The point-picking and compactness argument of \cite[Corollary~1.7 and the proof of Proposition~9.4]{CFFS25} applies with \Cref{thm:main-classification}: failure of the estimate would produce a global classical stable solution $v$ in $\mathbb R^4$ with $|D^2v(0)|=1$, contradicting \Cref{thm:main-classification}. Differentiating the free-boundary conditions controls its second fundamental form by $D^2u$.
\end{proof}

\begin{proof}[Proof of \Cref{cor:main-finite-index}]
It follows from \cite[Theorem~1.2]{FFS26} once we have \Cref{thm:main-classification} and \cite[Theorem~1.1]{Jerison-Savin}.
\end{proof}

\appendix

\section{Auxiliary proofs}
\label{app:proofs-standard-auxiliary-lemmas}

For the reader's convenience, we collect here proofs of the auxiliary results used in the paper.

\subsection{Affine comparison on a set of proportional volume}

We begin with an $L^2$ estimate for an affine function on any subset of proportional volume.

\begin{lem}
\label{lem:affine-comparison-proportional}
Let $c_1\in(0,1)$, $R>0$, $y\in\mathbb R^n$, and let $E\subset B_R(y)$ be measurable. Assume
\[
|E|\ge c_1|B_R|.
\]
Then every affine function $\ell(x)=a\cdot x+b$ satisfies
\begin{equation}\label{eq:affine-comparison-proportional}
R|a|+|\ell(y)| \le C(n,c_1) \left(\fint_E|\ell|^2\,dx\right)^{1/2}.
\end{equation}
\end{lem}

\begin{proof}
After translating and rescaling, assume that $y=0$ and $R=1$, and set $N:=|a|+|b|$. If $|b|\ge2|a|$, then $|\ell|\ge |b|-|a|\ge N/3$ throughout $B_1$, and the conclusion follows.

Suppose instead that $|b|<2|a|$, so $|a|\ge N/3$. For $s\in(0,1)$, the set
\[
\{x\in B_1:|\ell(x)|\le sN\}
\]
is contained in a slab orthogonal to $a$ of width at most $6s$. Consequently, its measure is at most $C_n s|B_1|$. Choose $s=s(n,c_1)$ so that $C_n s\le c_1/2$. Since $|E|\ge c_1|B_1|$,
\[
|E\cap\{|\ell|>sN\}|\ge \frac{c_1}{2}|B_1|.
\]
Therefore
\[
\left(\fint_E|\ell|^2\,dx\right)^{1/2} \ge c(n,c_1)N.
\]
Scaling back proves \eqref{eq:affine-comparison-proportional}.
\end{proof}

\subsection{Comparison of unit-slope vees}

We deduce the comparison estimate for unit-slope vees from \Cref{lem:affine-comparison-proportional}.

\begin{proof}[Proof of \Cref{lem:unit-slope-vee-comparison}]
The two sets
\[
E_\sigma:=\{x\in B_r(x_0):\sigma\ell_1(x)\ell_2(x)\ge0\}, \qquad \sigma\in\{\pm1\},
\]
cover $B_r(x_0)$. Choose $\sigma$ so that $|E_\sigma|\ge |B_r|/2$. On $E_\sigma$,
\[
|\ell_1-\sigma\ell_2| =\bigl||\ell_1|-|\ell_2|\bigr|.
\]
Applying \Cref{lem:affine-comparison-proportional} to $\ell_1-\sigma\ell_2$ on $E_\sigma$, and then using
\[
\|\ell_1-\sigma\ell_2\|_{L^\infty(B_r(x_0))} \le |(\ell_1-\sigma\ell_2)(x_0)| + r|\nabla\ell_1-\sigma\nabla\ell_2|,
\]
gives the result.
\end{proof}

\subsection{Sobolev--Poincar\'e on John domains}

We record the Sobolev--Poincar\'e estimate that upgrades the sided $L^2$ and energy control on the John domains.

\begin{lem}[Sobolev--Poincar\'e on a John domain]
\label{lem:john-domain-direct-sided-lq-upgrade}
Let $n\ge3$, $0<J\le1$, and $R>0$. Let $\Omega\subset\mathbb R^n$ be a $J$-John domain (see \Cref{defn:c-john-domains}) such that
\[
\diam\Omega\le J^{-1}R, \qquad |\Omega|\ge JR^n .
\]
Then, for every $w\in W^{1,2}(\Omega)$ and $2<q\le 2n/(n-2)$,
\begin{equation}\label{eq:john-domain-direct-sided-lq-upgrade}
\left(\fint_{\Omega} |w|^q\,dx\right)^{1/q} \le C_{n,q,J} \left[ \left(\fint_{\Omega}|w|^2\,dx\right)^{1/2} +R\left(\fint_{\Omega}|\nabla w|^2\,dx\right)^{1/2} \right].
\end{equation}
\end{lem}

\begin{proof}
By the Sobolev--Poincar\'e inequality for John domains \cite[Theorem~8]{Hajlasz2001},
\[
\|w-w_\Omega\|_{L^{2^*}(\Omega)} \le C(n,J)\|\nabla w\|_{L^2(\Omega)}, \qquad 2^*:=\frac{2n}{n-2}.
\]
The arclength-carrot condition in \Cref{defn:c-john-domains} is the form used there; density gives the estimate for $W^{1,2}$ functions. The assumptions on the diameter and volume yield
\[
\left(\fint_\Omega |w-w_\Omega|^{2^*}\,dx\right)^{1/2^*} \le C(n,J)R\left(\fint_\Omega |\nabla w|^2\,dx\right)^{1/2}.
\]
H\"older's inequality gives the same bound for $2<q\le2^*$, while $|w_\Omega|\le(\fint_\Omega|w|^2)^{1/2}.$ Combining the two estimates proves the claim.
\end{proof}

\subsection{Neumann compactness in an almost half-ball}

We establish a local compactness lemma that identifies small-residual harmonic functions on nearly flat half-domains with even harmonic limits.

\begin{lem}
\label{lem:linear-neumann-compactness-half-ball-l2}
Let $m\ge2$ and $p>2$. For every $\eta>0$ there exists $\delta=\delta(\eta,m,p)>0$ such that the following holds. Let $\Omega_\delta\subset\mathbb R^m$ be open, and let $v\in H^1(\Omega_\delta\cap B_1)$ be harmonic, with
\[
\|v\|_{W^{1,2}(\Omega_\delta\cap B_1)} +\|v\|_{L^p(\Omega_\delta\cap B_1)}\le1.
\]
Assume
\[
B_1\cap\{x_m\ge\delta\}\subset\Omega_\delta, \qquad \Omega_\delta\cap B_1\subset\{x_m\ge-\delta\},
\]
and
\[
\left|\int_{\Omega_\delta\cap B_1} \nabla v\cdot\nabla\varphi\,dx\right| \le\delta\|\varphi\|_{L^\infty(B_1)} \quad\text{for}\quad \varphi\in C_c^\infty(B_1).
\]
Then there is a harmonic function $w$ in $B_1$, even in $x_m$, such that
\[
\left(\int_{\Omega_\delta\cap B_1}|v-w|^2\,dx\right)^{1/2}\le\eta.
\]
\end{lem}

\begin{proof}[Proof of \Cref{lem:linear-neumann-compactness-half-ball-l2}] Suppose the conclusion fails for some $\eta_0>0$. Then there are $\delta_j\downarrow0$, domains $\Omega_j$, and functions $v_j$ satisfying the hypotheses, but whose $L^2(\Omega_j\cap B_1)$-distance from every harmonic function even in $x_m$ is larger than $\eta_0$.

For each fixed $\tau>0$, the set $D_\tau:=B_1\cap\{x_m>\tau\}$ lies in $\Omega_j$ for large $j$. Rellich compactness and a diagonal argument give
\[
v_j\rightharpoonup v_\infty\quad\text{in }W^{1,2}_{\rm loc}(B_1^+), \qquad v_j\to v_\infty\quad\text{in }L^2(D_\tau),
\]
where $B_1^+:=B_1\cap\{x_m>0\}$, and $v_\infty$ is harmonic in $B_1^+$.

We identify its boundary condition. If $\varphi\in C_c^\infty(B_1)$, then for $j$ large,
\[
\left| \int_{\Omega_j\cap B_1}\nabla v_j\cdot\nabla\varphi - \int_{B_1\cap\{x_m>\tau\}}\nabla v_j\cdot\nabla\varphi \right| \le C_\varphi\tau^{1/2}.
\]
Indeed, the difference is supported in $\{|x_m|\le\tau\}$, and the gradients are uniformly bounded in $L^2$. The weak Neumann residual makes the first integral $o(1)$. Letting first $j\to\infty$ and then $\tau\downarrow0$, we obtain
\[
\int_{B_1^+}\nabla v_\infty\cdot\nabla\varphi=0 \qquad\text{for every }\varphi\in C_c^\infty(B_1).
\]
Thus $v_\infty$ has zero weak normal derivative on $\{x_m=0\}$, and its even reflection $w_\infty$ is harmonic in $B_1$.

It remains to pass from convergence away from the boundary to convergence on the moving domains. The $L^p$ bound, with $p>2$, gives
\[
\int_{\Omega_j\cap B_1\cap\{|x_m|\le\tau\}}|v_j|^2 \le C\tau^{1-2/p},
\]
while the $L^2$-mass of $w_\infty$ in the same slab tends to zero with $\tau$. Hence
\[
\limsup_{j\to\infty} \int_{\Omega_j\cap B_1}|v_j-w_\infty|^2 \le C\tau^{1-2/p} +C\int_{B_1\cap\{|x_m|\le\tau\}}|w_\infty|^2 .
\]
Letting $\tau\downarrow0$ contradicts the assumed separation. The same argument also proves the sequential version used in \Cref{lem:linear-neumann-compactness-global}.
\end{proof}

\subsection{Global Neumann compactness}

We globalize \Cref{lem:linear-neumann-compactness-half-ball-l2} and identify the limiting even harmonic function as a polynomial.

\begin{proof}[Proof of \Cref{lem:linear-neumann-compactness-global}]
This is the $L^2$, weak-residual version of \cite[Lemma~8.7]{CFFS25}. If the conclusion failed, rotate the fixed normals to $e_m$ and take a sequence with $\delta_j\downarrow0$ that stays a fixed normalized $L^2$-distance from every admissible polynomial.

For each integer $R\ge1$, rescale by $s=2R$ and normalize by $C_ms^{N+\vartheta}$, where $C_m:=4\max\{1,\omega_m^{1/2}\}$ and $\omega_m$ is the volume of the unit ball in $\mathbb R^m$. The hypotheses then give the slab, growth, and weak-residual assumptions of \Cref{lem:linear-neumann-compactness-half-ball-l2}. A nested diagonal subsequence gives a limit on every $B_R$. The limits agree on overlaps, because each is the strong $L^2$ limit of the same unscaled sequence on the moving domains, and therefore define one entire harmonic function $h$, even with respect to $e_m^\perp$, satisfying
\[
\left(\fint_{B_R}|h|^2\,dx\right)^{1/2} \le CR^{N+\vartheta}\quad\text{for}\quad R\ge1.
\]
Interior estimates give, for fixed $x$,
\[
|D^{N+1}h(x)| \le C_xR^{-N-1} \left(\fint_{B_{2R}(x)}|h|^2\,dx\right)^{1/2} \le C_xR^{\vartheta-1}\to0.
\]
Thus $h$ is an even harmonic polynomial of degree at most $N$. The strong $L^2$ convergence on the moving domains now contradicts the assumed separation from every such polynomial.
\end{proof}

\bibliographystyle{plain}

\bibliography{references}
\end{document}